\documentclass[11pt,a4paper,leqno,twoside]{book}

\usepackage[utf8]{inputenc}
\usepackage[T1]{fontenc}       %
\usepackage[english]{babel}     %

\usepackage[pdftex]{graphicx}          

\usepackage{amsmath}
\usepackage{amssymb}
\usepackage{amsthm}
\usepackage{anysize}

\usepackage[mathscr]{eucal}

\usepackage{accents}

\usepackage{comment}

\date{Wroc{\l}aw, \today}

\theoremstyle{plain}
\newtheorem{thm}{Theorem}[section]
\newtheorem*{thm*}{Theorem}
\newtheorem{lem}[thm]{Lemma}
\newtheorem{prop}[thm]{Proposition}
\newtheorem{clm}[thm]{Claim}
\newtheorem{cor}[thm]{Corollary}
\newtheorem{conj}[thm]{Conjecture}

\theoremstyle{definition}
\newtheorem{df}[thm]{Definition}
\newtheorem{denot1}[thm]{Notation}
\newtheorem{denot}[thm]{Notations}
\newtheorem{assm}[thm]{Assumption}
\newtheorem{obs}[thm]{Observation}
\newtheorem{case}{Case}

\theoremstyle{remark}
\newtheorem{rem}[thm]{Remark}
\newtheorem{exmp}[thm]{Example}

\numberwithin{equation}{section}

\renewcommand{\Pr}{\mathrm {P}}
\newcommand{\Est}{\mathrm {E}}
\newcommand{\ind}{{\mathbf {1}}}
\newcommand{\im}{{\mathrm{im}}}
\renewcommand{\SS}{{\mathbb{S}^2}}
\newcommand{\HH}{{\mathbb{H}^2}}
\newcommand{\HHH}{{\mathbb{H}^3}}
\newcommand{\g}{\gamma}
\newcommand{\al}{\alpha}
\renewcommand{\b}{\beta}
\newcommand{\s}{\sigma}
\renewcommand{\phi}{\varphi}
\renewcommand{\rho}{\varrho}
\newcommand{\e}{\varepsilon}
\renewcommand{\d}{\delta}

\newcommand{\N}{{\mathbb{N}}}
\newcommand{\Z}{{\mathbb{Z}}}
\newcommand{\ud}{\mathrm{d}}

\newcommand{\bd}{\partial\,}
\newcommand{\bdt}{\mathrm{bd}\,}
\newcommand{\bdi}{\partial\,}

\newcommand{\Hd}{{\mathbb{H}^d}}
\newcommand{\cX}{{\hat{X}}}
\newcommand{\cl}[1]{{\overline{#1}}}
\renewcommand{\c}[1]{\widehat{#1}}

\newcommand{\clcX}[1]{{\overline{#1}^\cX}}
\newcommand{\cHHH}{{\widehat{\mathbb{H}}^3}}

\newcommand{\sm}{\setminus}
\newcommand{\sint}{\mathrm{int}\,}
\newcommand{\Isom}{\mathrm {Isom}}
\newcommand{\id}{\mathrm {Id}}
\newcommand{\pc}{{p_\mathrm c}}
\newcommand{\pu}{{p_\mathrm u}}
\newcommand{\pbb}{probability}
\newcommand{\Bbp}{Bernoulli bond percolation}

\renewcommand{\underline}[1]{#1}
\newcommand{\uwaga}[1]{}
\newcommand{\uwagap}[1]{}
\newcommand{\uwagaj}[1]{}
\newcommand{\uwagajp}[1]{}
\newcommand{\emd}[1]{\emph{#1}}
\newcommand{\work}{dissertation}
\newcommand{\partw}{chapter}

\renewcommand{\(}{\left(}
\renewcommand{\)}{\right)}

\newlength{\bigcupdotwidth}
\newcommand{\bigcupdotsymb}{\makebox[\bigcupdotwidth]
  {\makebox[0pt]{$\displaystyle{\bigcup}$}\makebox[0pt]{$\cdot$}}}
\DeclareMathOperator*{\bigcupdot}{\bigcupdotsymb}

\newlength{\cupdotwidth}
\newcommand{\cupdotsymb}{\makebox[\cupdotwidth]
  {\makebox[0pt]{$\displaystyle{\cup}$}\makebox[0pt]{$\cdot$}}}
\DeclareMathOperator*{\cupdot}{\cupdotsymb}

\newcommand{\voidindop}[2]{#1_{\makebox[0pt]{$\scriptstyle#2$}}}

\newlength{\rvoidindopOPWD}
\newlength{\rvoidindopAUX}

\newcommand{\rvoidindop}[2]{%
\settowidth{\rvoidindopOPWD}{$#1$}%
\settowidth{\rvoidindopAUX}{$\scriptstyle#2$}%
\addtolength{\rvoidindopAUX}{-\rvoidindopOPWD}%
\rule{0.5\rvoidindopAUX}{0mm}%
#1_{\makebox[0pt]{$\scriptstyle#2$}}}

\newcommand{\lvoidindop}[2]{%
\settowidth{\rvoidindopOPWD}{$#1$}%
\settowidth{\rvoidindopAUX}{$\scriptstyle#2$}%
\addtolength{\rvoidindopAUX}{-\rvoidindopOPWD}%
#1_{\makebox[0pt]{$\scriptstyle#2$}}%
\rule{0.5\rvoidindopAUX}{0mm}}

\newcommand{\voidindunderbrace}[2]{\underbrace{#1}_{\makebox[0pt]{$\scriptstyle#2$}}}

\begin{document}

\frontmatter

\begin{titlepage}
\begin{center}
\Large{{\sc Uniwersytet Wrocławski}}\\
\large{{\sc Wydział Matematyki i Informatyki}}\\
\vspace{3cm}

\LARGE{{\sc Jan Czajkowski}}\\
\vspace{1cm}
\huge{ Perkolacja na przestrzeni hiperbolicznej:}\\
\huge{ faza niejednoznaczności i włókniste klastry}\\
\vspace{5mm}
\large{\sc rozprawa doktorska}
\end{center}

\vspace{65mm}

\begin{flushright}
\parbox{8cm}{\Large {\sc Promotor rozprawy\\dr~hab.\ Jan Dymara}}\\
\vspace{2cm}
\normalsize{{\sc Wrocław, czerwiec 2013}}
\end{flushright}
\end{titlepage}

\begin{titlepage}\setcounter{page}{3}
\begin{center}
\Large{{\sc University of Wrocław}}\\
\large{{\sc Faculty of Mathematics and Computer Science}}\\
\vspace{3cm}

\LARGE{{\sc Jan Czajkowski}}\\
\vspace{1cm}
\huge{ Percolation in~the~hyperbolic~space:}\\  
\huge{ non-uniqueness~phase and~fibrous~clusters}\\
\vspace{5mm}
\large{\sc PhD dissertation}
\end{center}

\vspace{65mm}

\begin{flushright}
\parbox{8cm}{\Large {\sc Supervisor\\dr~hab.\ Jan Dymara}}\\
\vspace{2cm}
\normalsize{{\sc Wrocław, June 2013}}
\end{flushright}
\end{titlepage}

\tableofcontents

\chapter*{Acknowledgements}
I express my gratitude to my advisor, Jan Dymara, for his supervision, for many important and helpful remarks and words of advice. I am also grateful to Itai Benjamini and Ruth Kellerhals for helpful correspondence.
\par I thank my family for support and patience.

\mainmatter
\newcommand{\Bp}{Bernoulli percolation}
\newcommand{\Bsp}{Bernoulli site percolation}

\chapter{Introduction}

\section{Overview of the dissertation contents}

In this \work, I am going to consider \Bp\ on graphs admitting vertex-transitive actions of groups of isometries of $d$-dimensional hyperbolic spaces $\Hd$, where $d\ge2$\uwaga{d2}.
\par In Chapter \ref{ch3ph}, I prove the existence of a non-trivial non-uniqueness phase of \Bp\ on Cayley graphs for a wide class of Coxeter reflection groups of finite type polyhedra in $\HHH$.
\par In Chapter \ref{ch1ptbd}, I consider some geometric property of the clusters in \Bbp\ in the non-uniqueness phase on a class of connected, transitive, locally finite graphs in $\Hd$, much wider than in Chapter \ref{ch3ph}, for any $d\ge2$\uwaga{d2}.
\par In this general introduction I give an overview concerning percolation and its terminology. In the separate introduction to each of the Chapters \ref{ch3ph} and \ref{ch1ptbd}, I explain its contents more precisely, giving also some preliminaries needed in the given chapter.

\section{\Bp}

To explain what \Bp\ is, I start with a motivation from physics. Suppose we are given different kinds of porous materials. For example: a ceramic roof tile and a stick of chalk. When we drip some water on each of them, it will percolate through the chalk and not through the tile. What matters here is, roughly speaking, the size of the void spaces in either of the materials.
\par This situation can be modelled by a random subgraph of a given graph, called percolation process. Namely, we start with a graph (e.g. the standard cubic lattice $\Z^3$ or its fragment) which represents the whole space occupied by the piece of the material. We choose at random a subgraph of it that plays the role of the set of locations of the void spaces. A sensible way to do it is to fix a probability $p\in[0;1]$ and declare each edge of the original graph independently to belong to the random subgraph with \pbb\ $p$. The phenomenon which can be observed here is that the greater is $p$, the greater (in some sense) are the connected components of the random subgraph, which represent the void spaces, and the more likely is the water to percolate through the material modelled this way. Roughly speaking, this simple model relates the ability of the material to soak and the fraction of the void spaces in the material, represented by $p$.
\par Obviously, the procedure of choosing a random subgraph described above can be applied to any graph $G$. The resulting random subgraph $\omega$ is called $p$-\Bbp\ on $G$ and the parameter $p$---the (Bernoulli) percolation parameter.
\par One can perform an analogous process, removing vertices from the original graph independently with \pbb\ $1-p$ and taking the subgraph induced by the remaining vertices. This process is called $p$-\Bsp\ and often exhibit the same properties as the bond version.

\section{Threshold parameters $\lowercase\pc$ and $\lowercase\pu$}\label{secpcpu}

Of the main interest in the percolation theory are the connected components of $\omega$, called \emd{clusters}, and their ``size''. For example, one may ask when there are infinite components in the random subgraph $\omega$ with positive \pbb. If $G$ is a connected, locally finite graph, it turns out that, due to Kolmogorov's 0-1 law, the \pbb\ of that event is always $0$ or $1$. On the other hand, it is an increasing function of $p$ because the event in question in an {increasing} event---see e.g.\ Sections 1.4 and 2.1 of \cite{Grim}. Hence, there exists $\pc=\pc(G)\in[0;1]$, called the \emd{critical \pbb}, such that for $p<\pc$ a.s.\ there is no infinite cluster in $\omega$ and for $p>\pc$ a.s.\ there is some. (Due to this behaviour, we can view $\pc$ as a phase transition, where $p$ plays the role of the temperature.)
\par Percolation theory seems particularly interesting in the case of {transitive} or {quasi-transitive} graphs---maybe because, while they usually have simple description, still the percolation problems (e.g.\ of finding the value of $\pc$) are often hard. Below I define those graph classes.
\begin{df}
Unless indicated otherwise, I call a graph $G$ \emd{vertex-transitive}, or \emd{transitive} for short, if its automorphism group acts transitively on the set of vertices of $G$. If, instead, there are just finitely many orbits of vertices of $G$ under the action of its automorphism group, then I call $G$ \emd{quasi-transitive}.
\end{df}
If the graph $G$ is connected, locally finite and transitive, then it is known by \cite[Thm.~1]{NewmSchul} that the number of infinite clusters in $\omega$ is a.s.~constant and equal to $0$, $1$ or $\infty$ (see also \cite[Thm.~7.6]{LP}). Let us focus on the question when this number is $\infty$. It turns out that the number $\infty$ here is possible for some $p$ only for those connected, transitive, locally finite graphs which are {non-amenable} (see \cite[Thm.\ 7.7]{LP} and also an original paper \cite{BK89}), as defined below.
\begin{df}
We call a locally finite graph $G$ \emd{non-amenable} if
$$\inf \frac{|\bd K|}{|K|}>0,$$
where the infimum is taken over all non-empty finite sets of vertices of $G$ and where $\bd K$ is the set of the edges of $G$ having exactly one vertex in $K$. We call $G$ \emd{amenable} if the infimum equals $0$.
\end{df}
Actually, there are many examples of connected, transitive, non-amenable graphs such that for $p>\pc$ closer to $\pc$ there are a.s.\ infinitely many infinite components in the $p$-\Bp\ and for $p$ closer to $1$ there is exactly one infinite component there. In this situation, it is natural to define the next percolation threshold: \emd{unification probability} $\pu=\pu(G)$. It is defined as the infimum of $p\in[0;1]$ such that there is a.s.\ a unique infinite cluster in $\omega$. By this definition, for $p\in(\pc,\pu)$, a.s.~$\omega$ has infinitely many components. This range of $p$ is called the \emd{middle phase} or the \emd{non-uniqueness phase} of \Bp.
\par Mathematicians are interested when this interval of $p$ is non-degenerate, i.e.~when $\pc<\pu$. It is conjectured that non-amenability is also a sufficient condition for this in the case of quasi-transitive graphs:
\begin{conj}[\cite{BS96}]\label{conjBS96}
If $G$ is a non-amenable, quasi-transitive\footnote{This notion is referred to as ``almost transitive'' in \cite{BS96}.} graph, then
$$\pc(G)<\pu(G).$$
\end{conj}

\section{The motivation for phase transitions}

As I mentioned above, there are several classes of graphs for which Conjecture \ref{conjBS96} has been established. The first such result is due to \cite{GrimNewm} for the product of $\mathbb{Z}^d$ and an infinite regular tree of sufficiently high degree; then, among others, there is the paper \cite{Lal} establishing it for Cayley graphs of a wide class of Fuchsian groups and the paper \cite{BS}---for transitive, non-amenable, planar graphs with one end\footnote{An infinite connected graph has \emd{one end} if after throwing out any finite set of vertices it still has exactly one infinite component.}. There is also a quite general result in \cite{PSN}---any finitely-generated non-amenable group has a Cayley graph $G$ with $\pc(G)<\pu(G)$---and an analogous result for a continuous model of percolation in \cite{Tyk}. (Many of those results are obtained for Bernoulli site percolation as well.)
\par In the context of the results above concerning graphs arising naturally from the hyperbolic plane $\HH$, there is interest in seeking such concrete examples of graphs of natural tilings of the $d$-dimensional hyperbolic space $\Hd$ for $d\ge 3$ that $\pc<\pu$ for the Bernoulli bond or site percolation. (It is only known that such examples exist when, roughly speaking, the degree of the graph is sufficiently high.) Motivated by that, in Chapter \ref{ch3ph}, I find such concrete examples arising from reflection groups in $\HHH$.
\par Simultaneously, my advisor proposed that I seek for yet another phase transition in \Bp\ in the $3$-dimensional case. It is motivated as follows. Visualise $\HHH$ as the Poincar\'e disc model (cf.~\cite[Chapter I.6]{BH}). Consider the three phases of Bernoulli percolation on the graph of natural tiling of $\HHH$ with compact tiles: $[0;\pc)$, $(\pc,\pu)$ and $(\pu;1]$ (the middle one known e.g. by one of the above results or by Chapter \ref{ch3ph}; for the first and the last phase, see Remark \ref{remph123}). In the first phase, a.s.\ all the clusters are finite. Thus, viewed in large scale, they resemble points, hence $0$-dimensional objects. In the last phase there is a unique infinite cluster. It is known that then this cluster has one end and its closure in the Poincar\'e disc model contains the whole boundary sphere of the model. Such an infinite cluster deserves calling it a $3$-dimensional object. The phenomenon that my advisor suspects is that in the middle phase we can have ``fibrous'', or ``$1$-dimensional'', infinite clusters for $p$ close to $\pc$ and ``fan-shaped'', or ``$2$-dimensional'', infinite clusters for $p$ close to $\pu$.
\par In Chapter \ref{ch1ptbd}, I formalise the notion of a ``fibrous'' infinite cluster: I define it as a cluster with (only) one-point boundaries of ends (see Definition \ref{dfbd}). As the main result there, I find a sufficient condition  for $p$-\Bbp\ to exhibit this behaviour.
\par According to the intuition of $0$-, $1$-, $2$- and $3$-dimensional clusters for $\HHH$, for $\HH$ we should have only three phases arising from the ``dimension'' of the infinite clusters. It turns out that the phase with $1$-dimensional infinite clusters (in the sense of boundaries of ends) and the non-uniqueness phase, for vertex-transitive graphs of tilings of $\HH$ by bounded hyperbolic polygons, exactly agree, as the sets of respective $p\in[0;1]$. It is established in \cite{CzBCP}, but is also easily implied for Cayley graphs of a wide class of Fuchsian groups by much earlier paper \cite{Lal}.

\section{Technical preliminaries}

Here I am going to give some precise definitions and useful terminology which are used throughout this \work. It may be also helpful to consult e.g.\ \cite{Grim} and \cite{LP}, which give quite wide introduction to percolation theory.
\par The \Bp\ process on a graph $G$ can be formalised as follows.  By $V(G)$ and $E(G)$ I denote the sets of vertices and edges of $G$, respectively. For $p\in[0;1]$ the $p$-\Bbp\ process is defined as the probability measure $\Pr_p$ on the space $\Omega=2^{E(G)}$ of all subsets of $E(G)$, called (\emd{percolation}) \emd{configurations}, such that:
\begin{itemize}
\item the underlying $\s$-algebra is generated by all the cylindrical subsets of $2^{E(G)}$, i.e.\ the sets of configurations of the form
$$\{\omega\subseteq E(G): \omega\cap F = F'\}$$
for finite $F,F'$ such that\ $F'\subseteq F \subseteq E(G)$;
\item for any $e\in E(G)$
$$\Pr_p(\{\omega:e\in\omega\})=p;$$
\item the sets (random events) $\{\omega:e\in\omega\}$ for $e\in E(G)$ are independent with regard to $\Pr_p$.
\end{itemize}
In other words, if we identify $2^{E(G)}$ in the canonical way with the space $\{0;1\}^{E(G)}$ of all 0-1 functions on $E(G)$, $\Pr_p$ arises as the product measure
\begin{equation}
\Pr_p = \prod_{e\in E(G)} ((1-p)\d_0 + p\d_1)
\end{equation}
(where $\d_x$ is the Dirac delta).
\par For $\omega\in 2^{E(G)}$ I think of $\omega$ also as the subgraph of $G$ with the vertices set $V(G)$ and the edges set $\omega$.
\par For the site version of $p$-\Bp, I just take $\Omega=2^{V(G)}$ and then a \emd{configuration} $\omega\in\Omega$ means a subset of $V(G)$ or, interchangeably, the subgraph of $G$ induced by this set of vertices. The $\s$-algebra being the domain of $\Pr_p$ is again generated by the cylindrical subsets of $2^{V(G)}$ and $\Pr_p$ itself is the product measure
\begin{equation}
\Pr_p = \prod_{v\in V(G)} ((1-p)\d_0 + p\d_1)
\end{equation}
if we canonically identify $2^{V(G)}$ with $\{0;1\}^{V(G)}$.
\par Whenever I consider Bernoulli bond or site percolation and I use the symbol $\omega$, I mean $\omega\in\Omega$ and I treat it as a graph-valued random variable distributed according to $\Pr_p$. Random events in a percolation process are just measurable sets of configurations, unless some other random process is involved. We call the edges (or vertices for site percolation) belonging to $\omega$ \emd{open} and the other ones we call \emd{closed}. This terminology applies as well to sets of (more) edges (or vertices). By the \emd{state} of an edge (or vertex, respectively) I mean the logical value of the statement that the edge (or the vertex) is open.
\par Now, I complete the definitions of $\pc$ and $\pu$ given in Subsection \ref{secpcpu} by more formal definitions:
\begin{df}
The \emd{critical probability} for Bernoulli bond or site percolation on $G$ is
\begin{equation}
\pc(G):=\inf\{p\in[0;1]:\Pr_p(\omega\textrm{ has some infinite cluster})>0\}
\end{equation}
and the \emd{unification probability} for that percolation model is
\begin{equation}
\pu(G):=\inf\{p\in[0;1]:\textrm{$\Pr_p$-a.s.~there is a unique infinite cluster in $\omega$}\}.
\end{equation}
\end{df}
\par Because I am going to consider percolation in the context of hyperbolic spaces $\Hd$, some geometry and some properties of the models of those spaces are used here. In this \work, I do not introduce the basic concepts of that geometry and those models. Instead, I refer the reader to Chapters I.2 and I.6 in \cite{BH}. (This book contains also an introduction to basic concepts of geometry of metric spaces in Chapter I.1.)

\newcommand{\op}{\omega^{(p)}}
\newcommand{\opd}{\omega^{(p)\dag}}
\newcommand{\0}{\varepsilon}
\newcommand{\B}{\mathfrak{B}}

\newcommand{\ot}{{\tilde{o}}}
\newcommand{\rt}{{\tilde\varrho}}
\newcommand{\gs}{\gamma^*}
\newcommand{\gr}{\mathrm{gr}}
\newcommand{\F}{\mathrm{F}}

\newcommand{\oE}{{\vec{E}}}
\newcommand{\CG}{\Gamma}
\renewcommand{\d}{{^\dag}}
\newcommand{\rb}{{^\mathrm{rb}}}
\renewcommand{\r}{{\mathrm{r}}}
\newcommand{\R}{{\mathrm{R}}}

\newcommand{\fL}[1]{{f_{#1}^L}}
\newcommand{\ft}[1]{{f_{#1}^\triangle}}
\newcommand{\dk}{\frac{\ud}{\ud k}}

\newcommand{\appsec}{section}
\newcommand{\appsecs}{Sections \ref{appxgs} and \ref{appxgr}}

\chapter
{Non-uniqueness phase on reflection groups in $\HHH$}\label{ch3ph}

\section{Introduction}\label{intr.ch3ph}

As I mentioned in the introduction to this \work, in this chapter I prove the existence of an essential non-uniqueness phase (i.e.\ that $\pc<\pu$) in \Bp\ on Cayley graphs of a very large class of reflection Coxeter groups in $\HHH$. I state this result as the two main theorems of this chapter below. The result is split in this way because I find it convenient to prove them separately.

\begin{thm}\label{3phgen}
For $\CG$ the Cayley graph of the reflection group of a Coxeter polyhedron in $\HHH$ (in the sense of Definition \ref{dfpolyh}) with $k\ge13$ faces, with the standard generating set, we have
$$\pc(\CG)<\pu(\CG)$$
for Bernoulli bond and site percolation on $\CG$.
\end{thm}
\begin{thm}\label{3phRACpt}
If $\Pi$ is a compact right-angled polyhedron in $\HHH$, then for $\CG$ the Cayley graph of the reflection group of $\Pi$ with the standard generating set, we have
$$\pc(\CG)<\pu(\CG)$$
for Bernoulli bond and site percolation on $\CG$.
\end{thm}
Proofs of these theorems are completed in Sections \ref{secgen} and \ref{secRACpt}, respectively, although they use some facts established in Section \ref{secbasic} and \appsecs.
\begin{rem}\label{remph123}
In the setting of either Theorem \ref{3phgen} or Theorem \ref{3phRACpt} we have $\pc(\CG)\ge\frac{1}{k-1}>0$, which is well-known (see e.g.~\cite[Thm.~1.33]{Grim} or \cite[Prop.~7.13]{LP}). Also, if $\Pi$ in that setting is compact, then $\CG$ is the Cayley graph of a finitely presented group with one end, so from \cite[Thm.~10]{BB} we have $\pu(\CG)<1$ for bond percolation. Thus, for such a polyhedron $\Pi$, Theorems \ref{3phgen} and \ref{3phRACpt} give $3$ non-degenerate phases of Bernoulli bond percolation on $\CG$, which provides a picture analogous to that from \cite{BS} for $2$-dimensional case (where also $3$ such non-degenerate phases are established).
\end{rem}
In Section \ref{secbasic} an introductory version of Theorems \ref{3phgen} and \ref{3phRACpt} is proved:
\begin{thm}\label{3phbasic}
If $\Pi$ is a compact right-angled polyhedron in $\HHH$ with $k\ge 18$ faces, then for $\CG$ the Cayley graph of the reflection group of $\Pi$ with the standard generating set, we have
$$\pc(\CG)<\pu(\CG)$$
for Bernoulli bond and site percolation on $\CG$.
\end{thm}
The proof of this theorem uses relatively narrow set of facts, but the result is less general than Theorems \ref{3phgen}, \ref{3phRACpt}. More precise remarks on how particular tools used in this \partw{} can improve Theorem \ref{3phbasic} are in Section \ref{secrems}.
\par In \appsecs, two of those tools (Theorems \ref{thmgs}, \ref{gr>=}), which are somewhat interesting theorems themselves, are obtained.
\par The basis for proving all three Theorems \ref{3phgen}, \ref{3phRACpt}, \ref{3phbasic} are the following bounds for $\pc$ and $\pu$.
\begin{df}\label{dfgr}
For a transitive graph $\CG$, I define the (\emd{upper}) \emd{growth rate} of $\CG$
$$\gr(\CG) = \limsup_{n\to\infty} \sqrt[n]{\#B(n)},$$
where $B(n)$ is the closed ball of radius $n$ around some fixed vertex of $\CG$.
\end{df}
\begin{thm}[\protect{\cite[remarks in section 3]{Lyo95}}]\label{bdpc}
For Bernoulli bond and site percolation on any Cayley graph $\CG$,
$$\pc(\CG)\le\frac{1}{\gr(\CG)}.$$
\qed
\end{thm}
\begin{rem}
This estimate is established in the proof of  \cite[Theorem 7.21]{LP}. Although that proof is performed for \Bbp, it is valid also for the site version (together with the desired estimate). The reason is that it uses a formula for $\pc$ in \Bp\ on a tree, which is the same for the bond and the site version. Cf.\ also remarks in \cite[Section 3]{Lyo95}.
\end{rem}
\begin{denot}
For any graph $\CG$, not necessarily simple, its vertex $o$ and any $n\in\N$, I define $a_n(\CG,o)$ to be the number of cycles (see the second paragraph of Remark \ref{pathsinnonsmpl}) of length $n$ in $\CG$ starting from $o$, which do not use any vertex more than once. Then I put
$$\g(\CG,o) = \limsup_{n\to\infty}\sqrt[n]{a_n(\CG,o)}$$
(I may drop ``$o$'' when $\CG$ is transitive because then it does not play any role).
\end{denot}
\begin{thm}[O.~Schramm, \protect{\cite[Thm.~3.9]{Lyo00}}]\label{bdpu}
For Bernoulli bond and site percolation on any transitive graph $\CG$, we have
$$\pu(\CG)\ge\frac{1}{\g(\CG)}.$$
\qed
\end{thm}
In order to use this theorem, we need estimates for $\rt(\CG)$ (a value related to $\g(\CG)$---see Notations~\ref{Cas}, Remark \ref{g<=gs<=rt} and Theorem \ref{thmgs}). The estimates that I use are inspired by \cite{Nag} and based on Gabber's lemma (see Subsection \ref{subsecGab}).  
\section{Preliminaries}\label{secprelim}
\begin{denot1}
In this \work{} I assume that $0\in\N$ (i.e.~$0$ is a natural number).
\end{denot1}
\begin{rem}\label{pathsinnonsmpl}
I adopt the convention that paths and cycles can use some vertices and edges more than once and that a cycle is a particular case of a path, i.e.\ that it has endpoints.
\par For graphs which are not simple (i.e.~admit loops and multiple edges), I always consider paths not as sequences of vertices, but rather as sequences of edges. In particular, I distinguish between any two paths passing through the same vertices in the same order but through different edges.
\end{rem}
\begin{denot}\label{Cas}
Let $\CG$ be any graph, not necessarily simple, and $o$---its vertex. For any $n\in\N$, I define (in the context of Remark \ref{pathsinnonsmpl}):
\begin{itemize}
\item $C_n(\CG,o)$ to be the number of (oriented) cycles of length $n$ in $\CG$ starting at $o$;
\item $a^*_n(\CG,o)$ to be the number of oriented cycles of length $n$ in $\CG$, starting at $o$, without backtracks, i.e.~without pairs of consecutive passes through the same edge of $\CG$, forth and back. (I do not regard backtracks cyclically, i.e.~passing the same edge as the first and the last one is not a backtrack.)
\end{itemize}
Based on that, I define:
\begin{itemize}
\item $\rt(\CG,o) = \limsup_{n\to\infty}\sqrt[n]{C_n(\CG,o)}$,
\item $\gs(\CG,o) = \limsup_{n\to\infty}\sqrt[n]{a^*_n(\CG,o)}$.
\end{itemize}
Sometimes I may drop ``$o$'' or ``$(\CG,o)$'' in all the above notations if the context is obvious. Also, I will drop ``$o$'' when $\CG$ is transitive because then the above values do not depend on $o$.
\end{denot}
\begin{rem}\label{g<=gs<=rt}
Note that for any $\CG$ and its vertex $o$,
$$a_n(\CG,o)\le a^*_n(\CG,o)\le C_n(\CG,o)$$
for $n\in\N$,\footnote{It is true for natural $n\ge 3$, to be strict.} hence also
$$\g(\CG,o) \le \gs(\CG,o) \le \rt(\CG,o).$$
It implies that, thanks to Theorems \ref{bdpc} and \ref{bdpu}, it is sufficient to prove one of the inequalities $\gs(\CG)<\gr(\CG)$, $\rt(\CG)<\gr(\CG)$ in order to show that $\pc(\CG)<\pu(\CG)$.
\end{rem}
Below, I introduce some geometric notions concerning $\HHH$ used in this \partw. For definitions and descriptions of models of $\HHH$, see e.g.~Chapters I.2 and I.6 of \cite{BH}. For the theory of Coxeter groups and reflections groups, consult \cite{Dav} and \cite[Chap.~IV]{Mask}. 
\begin{df}\label{dfpolyh}\mbox{}
\begin{itemize}
\item By a half-space I mean by default a closed half-space.
\item By a \emd{convex polyhedron} in $\HHH$ I mean an intersection of finitely many half-spaces in $\HHH$ with non-empty interior.
\item For a plane supporting a convex polyhedron, i.e.~intersecting it so that it is left on one side of the plane, I call that intersection a \emd{face}, an \emd{edge} or a \emd{vertex} of the polyhedron if it is $2$-, $1$- or $0$-dimensional, respectively.
\item Two distinct faces $F_1$, $F_2$ of polyhedron are \emd{neighbours} (or: do \emd{neighbour}) iff they have an edge in common; we then write $F_1\sim F_2$.
\item The \emd{degree} of a face $F$ of a polyhedron, denoted by $\deg(F)$, is the number of its neighbours.
\item A convex polyhedron is \emd{Coxeter} iff the dihedral angle at each edge of it is of the form $\pi/m$, where $m\ge 2$ is natural.
\end{itemize}
\end{df}
\section{Description of the polyhedron and the graph}\label{secPiG}
The graphs that I am going to consider in this \partw{} are the Cayley graphs of the reflection groups of Coxeter polyhedra in $\HHH$. Below I introduce these notions:
\begin{df}
The \emd{reflection group} of a convex polyhedron $\Pi$ is the group of isometries of $\HHH$ generated by the hyperbolic reflections in planes of faces of $\Pi$.
\par A \emd{Coxeter matrix} over a finite set $S$ is a matrix $(m(s,t))_{s,t\in S}$ with elements in $\N\cup\{\infty\}$ such that
\begin{equation*}
m(s,t)=\left\{
\begin{array}{rl}
1&\textrm{ if }s=t\\
\ge 2&\textrm{ otherwise.}
\end{array}\right.
\end{equation*}
A \emd{Coxeter system} is a pair $(G,S)$, where $G$ is a group generated by $S$ with the presentation
\begin{equation*}
G=\langle S|s^2:s\in S; (st)^{m(s,t)}: s,t\in S, m(s,t)\neq\infty\rangle,
\end{equation*}
where $(m(s,t))_{s,t\in S}$ is a Coxeter matrix. We then call $G$ a \emd{Coxeter group}.
\end{df}
\begin{rem}
If $\Pi$ is a Coxeter polyhedron in $\HHH$, $S$ is the set of reflections in planes of its faces and $G=\langle S\rangle$ is its reflection group, then $(G,S)$ is a Coxeter system. The underlying Coxeter matrix $(m(s,t))_{s,t\in S}$ is connected with the dihedral angles of $\Pi$ as follows: for two neighbouring faces $F_1,F_2$ of $\Pi$ and the corresponding generators $s_1,s_2\in S$ (respectively), the dihedral angle between these faces equals $\pi/m(s_1,s_2)$. It is a special case of Poincar\'e's Polyhedron Theorem---Theorem IV.H.11 in \cite{Mask} (one has to use also the proposition in IV.I.6 there).
Moreover, from that theorem, $\Pi$ (or $\sint\Pi$, depending on the convention) is a fundamental polyhedron for $G$ (see \cite[Definition IV.F.2]{Mask}; roughly speaking, that is why the orbit $G\cdot\Pi$ is called a \emd{tiling} of $\HHH$ and $v\Pi,v\in G$---its \emd{tiles}) and because of that, the only stabiliser in $G$ of any interior point of $\Pi$ is the identity. Hence, fixing $o\in\sint\Pi$, one can identify $G$ with the orbit of $o$ by the obvious bijection. (Cf.~also \cite[Theorem 6.4.3]{Dav} for the case of compact $\Pi$.)
\end{rem}
\begin{assm}\label{PiG}
For the rest of this \work{}, let $\Pi$ be a Coxeter polyhedron in $\HHH$, let $S$ be the set of reflections in planes of its faces and let $G=\langle S\rangle$ be its reflection group. Let $\CG$ be the Cayley graph of the Coxeter system $(G,S)$, i.e.~the undirected (simple) graph with $V(\CG)=G$ such that an edge joins two vertices $v,w$ in $\CG$ iff $v=ws$ for some $s\in S$. I am going to think of $G=V(\CG)$ as of the orbit of a fixed point $o\in\sint\Pi$, so that $V(\Pi)\subseteq\HHH$, and $o$ is the identity element of $G$. (These assumptions are valid unless indicated otherwise.)
\end{assm}
\begin{rem}
The graph $\CG$ is dual to the tiling $G\cdot\Pi$ in the sense that each $v\in V(\CG)=G$ naturally corresponds with the tile $v\Pi$ (hence one may denote $v\Pi$ by $v\d$) and an edge $\{v,w\}\in E(\CG)$ indicates that the tiles $v\Pi$ and $w\Pi$ have a face in common. Indeed, the tiles $v^{-1}v\Pi=\Pi$ and $v^{-1}w\Pi$ share a face, say the face of $\Pi$ corresponding to $s$ for some $s\in S$, iff $s=v^{-1}w$).
\end{rem}
\begin{denot1}\label{ed}
I use the above remark to introduce the following notion: for $e=\{v,w\}\in E(\CG)$ (or $e=(v,w)\in\oE(\CG)$---see Notation \ref{oE}), let $e\d$ be the common face of $v\Pi$ and $w\Pi$.
Then, we say that $e$ traverses or passes through the face $e\d$. (Though, the geodesic segment between the endpoints of $e$ (in $\HHH$) does not need to literally cross the face $e\d$, so the introduced notion has a combinatorial nature.)
\end{denot1}
\section{A basic approach}\label{secbasic}
In this section I present a basic approach to proving the existence of a non-degenerate middle phase of Bernoulli bond and site percolation on some cocompact reflection groups of $\HHH$. The price for its relative easiness, in comparison to more developed versions I present later in this \partw{}, is that it works only for cocompact right-angled reflection groups with at least $18$ generators. Still, there is a quite elegant example of such group:
\begin{exmp}
Let $\Pi$ be a hyperbolic right-angled truncated icosahedron (i.e.~with the combinatorics of a classical football) in $\HHH$. Then its reflection group is a cocompact right-angled reflection group with $32$ generators.
\end{exmp}
\begin{thm*}[recalled Theorem \ref{3phbasic}]
If $\Pi$ is a compact right-angled polyhedron in $\HHH$ with $k\ge 18$ faces, then for $\CG$ the Cayley graph of the reflection group of $\Pi$ with the standard generating set, we have
$$\pc(\CG)<\pu(\CG)$$
for Bernoulli bond and site percolation on $\CG$.
\end{thm*}
Before I prove this theorem, I establish some general facts about reflection groups in $\HHH$, useful in the whole \partw{}.
\subsection{Gabber's lemma}
\label{subsecGab}
I estimate $\rt$ of the reflection group of a hyperbolic polyhedron using Gabber's lemma stated below.
\begin{df}\label{sprad}
For any (simple) graph $\CG$, by its \emd{spectral radius}, denoted by $\varrho(\CG)$, I mean the spectral radius of the simple random walk on $\CG$ starting at its fixed vertex $o$. (The \emd{simple random walk} on $\CG$ starting at $o$ is a Markov chain with the state space $V(\CG)$, where at each step the probability of a transition from a vertex $x$ to its neighbour equals $1/\deg(x)$---see \protect{\cite[Section I.1]{Woe}}). In other words
\begin{equation*}
\varrho(\CG) = \limsup_{n\to\infty}\sqrt[n]{p^{(n)}(o,o)},
\end{equation*}
where for any natural $n$, $p^{(n)}(o,o)$ denotes the probability that the simple random walk on $\CG$ starting at $o$ is back at $o$ after $n$ steps. Note that the spectral radius of $\CG$ does not depend on the choice of $o$ (as for $o,o'\in V(\CG)$ there is $C>0$ such that for $n\in\N$, $p^{(n+\mathrm{dist}(o,o'))}(o',o') \ge Cp^{(n)}(o,o)$ and \emph{vice-versa}, whence $p^{(n)}(o,o)$ and $p^{(n)}(o',o')$ has the same asymptotic behaviour).
\end{df}
\begin{rem}\label{remrrt}
It is easily seen that for any (simple) $k$-regular graph $\CG$ and for any $n\in\N$,
$$p^{(n)}(o,o)=C_n(\CG,o)/k^n,$$
so $\varrho(\CG)=\rt(\CG)/k$.
\end{rem}
\begin{denot}\label{oE}
For an undirected (simple) graph $\CG$, let $\oE(\CG)$ denote the set of edges of $\CG$ given orientations. (It can be formalised by $\oE(\CG)=\{(v,w):\{v,w\}\in E(\CG)\}$.) For $e\in\oE(\CG)$, let $\bar{e}$ denote the edge inverse to $e$ and let $e_+$,$e_-$ denote the end (head) and the origin (tail) of $e$, respectively.
\end{denot}
\begin{lem}[Gabber]
Let $\CG$ be an unoriented, infinite, locally finite graph and let a function $F:\oE(\CG)\to\mathbb{R}_+$ satisfy
\begin{equation}
F(e)=F(\bar{e})^{-1}
\end{equation}
for each edge $e\in\oE(\CG)$. If there exists a constant $C_F>0$ such that for each vertex $v$ of $\CG$,
\begin{equation}
\frac{1}{\deg(v)}\sum_{e_+=v} F(e)\le C_F,
\end{equation}
then
\begin{equation}
\varrho(\CG)\le C_F.
\end{equation}
\qed
\end{lem}
See \cite[Prop.~1]{Bart} for the proof. (They formulate this lemma for regular graphs, but their proof is valid in the above generality. Also, cf.~Lemma of Gabber on p.~2 of \cite{Nag}.)
\par In this \partw{} I am going to use the notion of $\rt$ rather than $\varrho$, as the former seems more convenient here. So I reformulate the above lemma, using Remark \ref{remrrt}:
\begin{cor}\label{Gabcor}
Let $\CG$ and $F$ be as in the above lemma and assume that $\CG$ is regular. Then, if a constant $C_F>0$ satisfies
\begin{equation}
\sum_{e_+=v} F(e)\le C_F
\end{equation}
for every vertex $v$ of $\CG$, then
\begin{equation}
\rt(\CG)\le C_F.
\end{equation}
\end{cor}
\subsection{Geometry of the polyhedron and the graph}\label{geomPiG}
\begin{df}\label{OG}\label{lgth}
Let $\CG$ be the Cayley graph of any Coxeter system $(G,S)$ and let $o$ be the identity element of $G$. I denote by $l$ the \emd{length function} on $(G,S)$, i.e.~for $v\in G=V(\CG)$, $l(v)$ is the graph-theoretic distance from $v$ to $o$ (or: the least length of a word over $S$ equal $v$ in $G$). Let $O(\CG)$ denote the \emd{standard orientation of edges} of $\CG$ arising from $l$:
\begin{equation}
O(\CG):=\{e\in\oE(\CG):l(e_+)>l(e_-)\}.
\end{equation}
\end{df}
\begin{rem}\label{l>l}
Note that, in the setting of Definition \ref{lgth}, we never have $l(e_+)=l(e_-)$ for $e\in\oE(\CG)$. In fact, otherwise we would obtain a word over $S$ of odd length, trivial in $G$, hence a product of conjugates of the Coxeter relations, which are of even lengths, a contradiction.
\par So, for $e\in E(\CG)$, exactly one of the two oriented edges corresponding to $e$ is in $O(\CG)$. This edge is called the \emd{orientation of $e$ defined by} $O(\CG)$, or \emd{$e$ oriented according to} $O(\CG)$.
\end{rem}
\begin{df}
From now on, whenever I mention or use the orientation of an unoriented edge of $\CG$, I mean, by default, the orientation defined by $O(\CG)$. Particularly, by an edge of $\CG$ \emd{passing to} (or \emd{from}) vertex $v\in V(\CG)$ I mean an edge $e$ from $E(\CG)$ with $e_+=v$ (respectively $e_-=v$) when oriented according to $O(\CG)$.
\end{df}
\begin{rem}
Assume Assumption \ref{PiG} (particularly, on $\CG$). Whenever $e\in \oE(\CG)$, the plane $P$ of $e\d$ separates the endvertices of $e$, and separates one of them from $o$---let us call it $v$ and the other endvertex---$u$. Then, it turns out that $(u,v)\in O(\CG)$. To see that,
take a path $p$ in $\CG$ from $o$ to $v$ of minimal length: $l(v)$. Here, consider $p$ as a sequence of vertices. Then, reflect in $P$
every fragment of $p$ lying outside $P$ (from the point of view of $o$), obtaining a path 
(which stays in one vertex for some steps, such steps contributing $0$ to the path length) from $o$ to $u$ with the length strictly less than $l(v)$. 
Hence a geodesic in $\CG$ from $o$ to $u$ has length $l(u)<l(v)$, so $(u,v)\in O(\CG)$.
\end{rem}
\begin{denot}
For $v\in V(\CG)$, in the setting of Definition \ref{OG}, let $r(v)$ denote the number of edges of $\CG$ passing to $v$, i.e.
\begin{equation}
r(v)=\#\{e\in O(\CG):e_+=v\}.
\end{equation}
Note that $r(v)>0$ for $v\neq o$ (and $r(o)=0$).
For natural $i$, let $q_i(v)$ be the number of edges passing from $v$ to vertices with $r(\cdot)=i$; formally
\begin{equation}
q_i(v)=\#\{e\in O(\CG):e_-=v,r(e_+)=i\}.
\end{equation}
Note that always $q_0(v)=0$.
\par I will prove in Proposition \ref{r<=3} that in the setting of Assumption \ref{PiG} we have $r(v)\le 3$ (whence $q_i(v)=0$ for $i>3$).
\end{denot}
The only assumption on $\Pi$ in this subsection is Assumption \ref{PiG}. Here I am going to use some geometric facts on the tiling of $\HHH$ for estimating $r(v)$ and for using Corollary \ref{Gabcor} (see the previous subsection):
\begin{prop}\label{ePi}
For any two faces of $\Pi$, either they neighbour or they lie in disjoint planes.
\end{prop}
\begin{prop}\label{vPi}
For any three faces of $\Pi$ whose planes have non-empty intersection, some vertex of $\Pi$ belongs to all those faces.
\end{prop}
Those facts are easy consequences of the theorem from \cite{Andr}. I formulate it below in a version for $\HHH$ and a finite-sided convex polyhedron (just as in Assumption \ref{PiG}) for simplicity. Before, I need some definitions:
\begin{df}\label{dffromAndr}
Consider the Klein unit ball model of $\HHH$ with its ideal boundary $\bdi\HHH$ being the unit sphere in $\mathbb{R}^3$. Then for $A\subseteq\HHH$, we denote by $\c{A}$ its closure in $\cHHH:=\HHH\cup\bdi\HHH$. Then, a single point in $\bdi\HHH$ is claimed to have dimension $-1$ (as opposed to a point in $\HHH$) and the empty set---to have dimension $-\infty$.
\end{df}
\begin{thm}[a version of the theorem from \cite{Andr}]
Let $\Pi$ be a convex polyhedron in $\HHH$ (as in Assumption \ref{PiG}) whose all dihedral angles are non-obtuse and let $(F_i)_{i\in I}$ be any family of its faces and for $i\in I$ let $P_i$ be the plane of $F_i$. Then
$$\dim\bigcap_{i\in I}\c{F_i} = \dim\bigcap_{i\in I}\c{P_i}.$$
\qed
\end{thm}
\begin{rem}
In order to conclude Propositions \ref{ePi} and \ref{vPi} from the above theorem, one has to observe that
\begin{equation}
\bigcap_{i\in I}F_i\neq\emptyset \iff \dim\bigcap_{i\in I}\c{F_i}\ge 0 \iff \dim\bigcap_{i\in I}\c{P_i}\ge 0 \iff \bigcap_{i\in I}P_i\neq\emptyset
\end{equation}
using Definition \ref{dffromAndr} and the convexity of $\bigcap_{i\in I}\c{F_i}$ and $\bigcap_{i\in I}\c{P_i}$ in the Klein model. That immediately gives Proposition \ref{vPi}. To conclude Proposition  \ref{ePi} as well, observe that if the planes of two faces of $\Pi$ intersect, then they have at least one common vertex. Then, they must have a common edge because otherwise the polyhedral angle of $\Pi$ at their common vertex would have more than three faces, which is impossible for a polyhedron with non-obtuse dihedral angles. So we obtain Proposition \ref{vPi}.
\end{rem}
\begin{rem} 
By default, I consider the planes of faces of $\Pi$ oriented outside the polyhedron, i.e. the angle between them is $\pi$ minus the angle between their normal vectors, which I consider always pointing outside the polyhedron.
\end{rem}
\begin{prop}\label{v:r>=3}
For any $v\in V(\CG)$ and any three edges passing to $v$, they correspond to three faces of the tile $v\Pi$ which have a vertex in common.
\end{prop}
\begin{proof}
Assume that some three edges pass to a vertex $v$. Let $H_1,H_2,H_3$ be the (closed) half-spaces containing $v\cdot\Pi$ whose boundaries contain the faces $F_1,F_2,F_3$, respectively, corresponding to those edges. Because none of $H_1,H_2,H_3$ contains $o$ and none of them contains another, their boundaries must intersect pairwise, as well as $F_1,F_2,F_3$ (by Proposition \ref{ePi}). Further, if those all boundaries had no common point, the intersection $A=H_1\cap H_2\cap H_3$ would be a bi-infinite triangular prism and $H_1$ would contain $\HHH\sm (H_2\cup H_3)$ together with $o$, which contradicts the assumptions. So the planes of $F_1,F_2,F_3$ have a point in common---call it $p$ and by Proposition \ref{vPi}, $F_1,F_2,F_3$ share $p$ as a common vertex.
\end{proof}
\begin{prop}\label{r<=3}
For any $v\in V(\CG)$, we have $r(v)\le 3$.
\end{prop}
\begin{proof}
Assume that, on the contrary, $r(v)\ge 4$. Then there are half-spaces $H_1,H_2,H_3,H_4$ corresponding to faces $F_1,F_2,F_3,F_4$ of $\Pi$ and separating $o$ from $\Pi$. Let $A=H_1\cap H_2\cap H_3$ (we know that it is a trihedral angle from the proof of Proposition \ref{v:r>=3}).
\par $\bdt H_4$ has to cross $\sint A$ in order to produce face $F_4$ and has to cross each of the edges of $A$ in exactly $1$ point other than $p$ (from Proposition \ref{v:r>=3} and the fact that $H_4$ contains all the edges of $\Pi$ adjacent to $p$).
Note that $p\in\sint H_4$.
It is clear that the open half-line $po\sm\overrightarrow{po}$ of the line $po$ lies in $A$, so it crosses $\bdt H_4$. Hence, $\overrightarrow{po}\subseteq H_4$ along with $o$, a contradiction.
\end{proof}
\begin{denot}
Now, I define the function $F$ for use of Corollary \ref{Gabcor}: for $e\in O(\CG)$, let
\begin{equation}F(e)=c_{r(e_+)},\quad F(\bar e)=\frac{1}{c_{r(e_+)}},\end{equation}
according to the condition from the corollary. Here $c_1,c_2,c_3>0$ are parameters (only three ones, as here always $0<r(e_+)\le 3$). I will write $(c_1,c_2,c_3)=\bar c$ for short. Let for $v\in V(\CG)$,
\begin{equation}f_v(\bar c)=\sum_{e_+=v} F(e).\end{equation}
\end{denot}
In this setting
\begin{align}
f_v(\bar c)
&=\sum_{e\in O(\CG):e_+=v} c_{r(v)} + \sum_{e\in O(\CG):e_-=v} \frac{1}{c_{r(e_+)}}= \label{f(r,q)1}\\
&=r(v)c_{r(v)} + \sum_{i=1}^3 \frac{q_i(v)}{c_i},\label{f(r,q)2}
\end{align}
because $q_i(v)=0$ for $i>3$ (due to Proposition \ref{r<=3}) and for $i=0$.
Recall that $k$ is the number of faces of $\Pi$.
\begin{lem}\label{rhobasic}
If $k\ge 6$, we have
$$\rt(\CG)\le 2\sqrt{3(k-3)}.$$
\end{lem}
\begin{rem}
The above bound has a better asymptotic behaviour that those in Lemmas \ref{rhogen} and \ref{rhoRACpt}, but the latter ones give the inequality $\pc<\pu$ for a bit more values of $k$ (see Section \ref{secrems}).
\end{rem}
\begin{proof}
I am going to choose values of $c_1,c_2,c_3$ giving a good upper bound for $\sup_{v\in V(\CG)}f_v(\bar c)$: let $c_1=c_2=c_3=\sqrt{\frac{k-3}{3}}$. Then from \eqref{f(r,q)1}
\begin{align*}
f_v(\bar c) &= r(v){\textstyle\sqrt{\frac{k-3}{3}}} + (k-r(v)){\textstyle\sqrt{\frac{3}{k-3}}}\le\\
&\le 3{\textstyle\sqrt{\frac{k-3}{3}}} + (k-3){\textstyle\sqrt{\frac{3}{k-3}}}=2\sqrt{3(k-3)},
\end{align*}
because here $\sqrt{\frac{k-3}{3}}\ge\sqrt{\frac{3}{k-3}}$ and $r(v)\le 3$.
\end{proof}
\begin{df}\label{nerve}
For a Coxeter system $(G,S)$, we call a subset $T\subseteq S$ \emd{spherical} if the subgroup $\langle T\rangle$ is finite. The \emd{nerve} of $(G,S)$ is the abstract simplicial complex (or its geometric realisation if indicated so) whose simplices are all non-empty spherical subsets of $S$.
(For definitions of abstract simplicial complex and its geometric realisation, see \cite[Section A.2]{Dav} or Chapter I.7, especially the appendix of Chapter I.7, of \cite{BH}. A definition of nerve is also present in \cite[Section 7.1.]{Dav}.) 
\end{df}
\begin{clm}\label{no4inL}
In the setting of Assumption \ref{PiG}, no subset of $S$ of cardinality $4$ is spherical (i.e.~the nerve of $(G,S)$ contains no $3$-simplex).
\end{clm}
\begin{proof}
First, note that any $3$ generators constituting a spherical subset of $S$, correspond to reflection planes with non-empty intersection (otherwise those planes would be the planes of some bi-infinite triangular prism and reflections in them would generate an infinite group). Hence, from Proposition \ref{vPi} every $3$ faces corresponding to such $3$ generators share a vertex.
\par Now, assume, contrary to the claim, that there is a spherical subset of $S$ of cardinality $4$. Then, from the above, every $3$ of the $4$ faces of $\Pi$ corresponding to those $4$ generators, have a common vertex. On the other hand, there is no common vertex of all those $4$ faces because otherwise we would have a polyhedral angle with more than $3$ faces and non-obtuse dihedral angles, which is impossible. One can easily see that then $\Pi$ must be a compact tetrahedron, but then the group generated by those $4$ faces would be infinite, a contradiction.
\end{proof}
\subsection{Growth of right-angled cocompact groups in $\HHH$}
\begin{clm}
Let $G$ be the (Coxeter) reflection group of a right-angled compact polyhedron $\Pi$ in $\HHH$ with the standard generating set $S$ and let $L$ be the nerve of $(G,S)$ (in the sense of a geometric realisation---see Definition \ref{nerve}). Then $L$ is a flag triangulation of $\SS$. (A simplicial complex is \emd{flag} iff any finite set of its vertices which are pairwise connected by edges spans a simplex.)
\end{clm}
\begin{proof}
Due to Claim \ref{no4inL}, there is no $3$-simplex (nor higher-dimensional ones) in $L$. Now, each pair of faces of $\Pi$ corresponds to a spherical pair of generators iff those faces neighbour (for the ``only if'' part, use Proposition \ref{ePi} and the fact that reflections in two disjoint planes generate an infinite group). Similarly, each $3$ faces correspond to a spherical subset of $S$ iff they share a vertex---see the proof of Claim \ref{no4inL}. Note that the degrees of vertices of $\Pi$ are all equal to $3$ because the only possibility for a polyhedral angle with right dihedral angles is a trihedral angle (we then call $\Pi$ simple). So, $L$ is a complex dual to the polygonal complex of faces of $\Pi$ (meaning that vertices of $L$ correspond to faces of $\Pi$, edges of $L$---to edges of $\Pi$ and triangles of $L$---to vertices of $\Pi$), hence, it is a triangulation of $\SS$. It remains to show that it is flag: if it were not, then we would have three faces pairwise neighbouring, but not sharing a vertex, which would contradict Andreev's theorem (see \cite[Thm.~6.10.2 (ii)]{Dav}), as $\Pi$ is right-angled and simple.
\end{proof}
\begin{df}\label{grser}
Let $G$ be a group with finite generating set $S$. Then the \emd{growth series} of $G$ with respect to $S$ is the formal power series $W$ defined by
$$W(z)=\sum_{n=0}^\infty\#S(n)z^n = \sum_{g\in G}z^{l(g)},$$
where for $n\in\N$, $S(n)=B(n)\sm B(n-1)$ is the (graph-theoretic) sphere in the Cayley graph of $(G,S)$, centred at some fixed vertex, of radius $n$, and $l$ is the length function on $(G,S)$ (cf.~Definitions \ref{dfgr} and \ref{lgth}).
\end{df}
Growth rate of a group with finite set of generators is exactly the reciprocal of radius of convergence of the growth series of the group. That and the above claim lead to the following theorem.
\begin{thm}[for the proof, see \protect{\cite[Example 17.4.3.]{Dav}} with the exercise there]
\label{gr}
For $G$ the (Coxeter) reflection group of a right-angled compact polyhedron $\Pi$ with $k$ faces in $\HHH$, with the standard generating set $S$, its growth rate
$$\gr(G,S)=\frac{k-4+\sqrt{(k-4)^2-4}}{2}.$$
\qed
\end{thm}
\subsection{The proof for basic approach}
I recall the theorem to be proved:
\begin{thm*}[recalled Theorem \ref{3phbasic}]
If $\Pi$ is a compact right-angled polyhedron in $\HHH$ with $k\ge 18$ faces, then for $\CG$ the Cayley graph of reflection group of $\Pi$ with the standard generating set, we have
$$\pc(\CG)<\pu(\CG)$$
for Bernoulli bond and site percolation on $\CG$.
\end{thm*}
\begin{proof}
First, note that, due to Remark \ref{g<=gs<=rt}, it is sufficient to show that $\rt(\CG)<\gr(\CG)$ to prove the theorem.
\par Let us put $b_1(k):=2\sqrt{3(k-3)}$ (the upper bound for $\rt(\CG)$) and $b_2(k):=\frac{1}{2}(k-4+\sqrt{(k-4)^2-4})$ (the formula for $\gr(\CG)$). It is sufficient to prove that for real $k\ge 18$,
\begin{equation}b_1(k)<b_2(k).\end{equation}
That will follow once shown for $k=18$, provided that inequality
\begin{equation}
\dk b_1(k) \le \dk b_2(k)
\end{equation}
is shown for $k\ge 18$.
\par For $k=18$, we have
$$b_1(k)=2\sqrt{45}<7+\sqrt{48}=b_2(k)$$
and for $k\ge 18$, we differentiate:
\begin{equation*}
\dk b_1(k)=\sqrt{\frac{3}{k-3}}\le 1
\end{equation*}
and
\begin{equation*}
\dk b_2(k) = \frac{1}{2} + \frac{1}{2}\frac{k-4}{\sqrt{(k-4)^2-4}}\ge 1 \ge \dk b_1(k), \label{gr'>=1},
\end{equation*}
which finishes the proof.
\end{proof}
\section{The case of a general polyhedron}\label{secgen}
Recall Assumption \ref{PiG}.
\begin{thm*}[recalled Theorem \ref{3phgen}]
For $\CG$ the Cayley graph of the reflection group of a Coxeter polyhedron in $\HHH$ (in the sense of Def.\ \ref{dfpolyh}) with $k\ge13$ faces, with the standard generating set, we have
$$\pc(\CG)<\pu(\CG)$$
for Bernoulli bond and site percolation on $\CG$.
\end{thm*}
The tools which improve the result in Theorem \ref{3phbasic} to the above one (and also to Theorem \ref{3phRACpt}), are the upper bound for $\gs(\CG)$, along with a more appropriate upper bound for $\rt(\CG)$, and a lower bound for $\gr(\CG)$, stated below. (Along with the upper bound for $\gs(\CG)$, I establish the equality in the second part of the theorem below, but I do not use it in this \work.)
\begin{thm}\label{thmgs}
Let $\CG$ be an arbitrary regular graph of degree $k\ge 2$ (not necessarily simple) with distinguished vertex $o$ and let $\rt=\rt(\CG,o)$, $\gs=\gs(\CG,o)$. Then
\begin{equation}
\gs \le \frac{\rt+\sqrt{\rt^2-4(k-1)}}{2}.
\end{equation}
If, in addition, $\rt(\CG)>2\sqrt{k-1}$ (e.g.~when $\CG$ is vertex-transitive and simple 
and is not a tree), then the estimate becomes an identity:
\begin{equation}
\gs = \frac{\rt+\sqrt{\rt^2-4(k-1)}}{2}.
\end{equation}
\end{thm}
The proof of the above theorem is deferred to \appsec\ \ref{appxgs}.
\begin{thm}\label{gr>=}
For $\CG$ the Cayley graph of the reflection group of a Coxeter polyhedron with $k$ faces in $\HHH$ with the standard generating set, if we assume that $k\ge 6$, then the growth rate of $\CG$
$$\gr(\CG)\ge\frac{k-4+\sqrt{(k-4)^2-4}}{2}.$$
\end{thm}
(The assumptions here are just as in Assumption \ref{PiG}.) The proof is presented in \appsec\ \ref{appxgr}.
\begin{lem}\label{rhogen}
For $k$ and $\CG$ as in Assumption \ref{PiG}, we have
$$\rt(\CG)\le\frac{k+17}{3}.$$
\end{lem}
\begin{proof}
In order to bound the sum \eqref{f(r,q)2} (and to use the Corollary \ref{Gabcor}), I put $c_1=c_2=3,c_3=2$ and I am going to bound $q_3(v)$.
\begin{lem}\label{q_3gen}
For $o\neq v\in V(\CG)$, we have
$$q_3(v)\le
\left\{\begin{array}{r@{\text{ if }}l}
0& r(v)=1\\
2& r(v)=2\\
3& r(v)=3
\end{array}\right.$$
\end{lem}
Before proving the lemma, some notions concerning oriented edges need to be introduced:
\begin{denot}
For $e\in\oE(\CG)$, let $|e|\in E(\CG)$ denote $e$ as unoriented edge and let us denote by $H(e)$ the half-space bounded by the plane of $e\d$, containing $e_-$.
\par For $e\in E(\CG)$ (or $e\in \oE(\CG)$), let $R_e$ be the reflection in the plane of $e\d$. Whenever I write $R_e(e')$ for some $e'=(v,w)\in\oE(\CG)$, I mean the oriented edge $(R_e(v),R_e(w))$.
\end{denot}
\begin{proof}
Let $e\in O(\CG)$ be an arbitrary edge passing from $v$ and $N$ be the set of unoriented edges passing to $v$ (when oriented according to $O(\CG)$) traversing faces of $\Pi$ neighbouring $e\d$:
\begin{equation}
N=\{|f|:f\in O(\CG), f_+=v, f\d\sim e\d\}.
\end{equation}
Call the vertex $e_+$ by $v'$ and let $N'=R_e(N)$ and note that $v'\Pi=R_e(v\Pi)$.
\begin{clm}\label{eN'}
All the edges passing to $v'$ belong (without orientations) to $\{|e|\}\cup N'$.
\end{clm}
\begin{proof}
The proof is by contraposition: for arbitrary $e'\in \oE(\CG)$ passing from $v'$ such that $|e'|\notin\{|e|\}\cup N'$, I show that $e'\in O(\CG)$. So, let $e'\neq \bar e$ be arbitrary edge from $\oE(\CG)$ passing from $v'$ (this orientation is assumed just for convenience), such that $|e'|\notin\{|e|\}\cup N'$. Then,
\begin{itemize}
\item if $e'\d$ neighbours $e\d$, then $R_e(e')\d\sim e\d$ as well and $|R_e(e')|\notin N$, so $|R_e(e')|$ passes from $v$, i.e.~$R_e(e')\in O(\CG)$, so
\begin{equation*}
H(e')\supseteq H(R_e(e'))\cap H(e)\ni o,
\end{equation*}
because $v'\Pi$ has only non-obtuse dihedral angles, so $e'\in O(\CG)$;
\item if $e'\d$ does not neighbour $e\d$, then from Proposition \ref{ePi} the planes containing them are disjoint, so $H(e')\supseteq H(e)\ni o$ and $e'\in O(\CG)$.
\end{itemize}
That shows that the edges passing to $v'$ all belong to $\{|e|\}\cup N'$ (without orientations).
\end{proof}
Now, assume that $r(v')=3$. By the claim, two of the edges passing to $v'$ are in $N'$ and the third is $e$ and by Proposition \ref{r<=3} the corresponding faces of $v'\Pi$ share a common vertex. So do their reflections in $R_e$, which are $e\d$ and two of the faces corresponding to edges passing to $v$.
\par That means that:
\begin{itemize}
\item if $r(v)=1$, then there is no $e$ as above (i.e.~with $r(e_+)=3$) and $q_3(v)=0$;
\item if $r(v)=2$, then the two faces corresponding to edges passing to $v$ have to lie on intersecting planes, hence share one edge (from Proposition \ref{ePi}) and there are at most two possibilities for such a common vertex with $e\d$ as above (and hence for $e$), so $q_3(v)\le 2$;
\item if $r(v)=3$, then the faces corresponding to edges passing to $v$ share a common vertex and share at most three other vertices pairwise---the latter are the only possibilities for a common vertex with $e\d$ as above, so $q_3(v)\le 3$.
\end{itemize}
That finishes the proof of Lemma \ref{q_3gen}
\end{proof}
\begin{rem}
Consider vertex $o$. Each neighbour $v$ of $o$ has $r(v)=1$ (because going closer to $o$ from $v$ we must return to $o$). Hence, $q_1(o)=k$, $q_2(o)=q_3(o)=0$.
\end{rem}
Now, recall that $\bar c = (3,3,2)$. For $o\neq v\in V(\CG)$, because $q_1(v)+q_2(v)=k-r(v)-q_3(v)$, we have
\begin{align}
\\
\textrm{if $r(v)=1$, then} &\quad f_v(\bar c)=3 + \frac{q_1(v)+q_2(v)}{3} + \frac{0}{2}=\frac{k+8}{3},\notag\\
\textrm{if $r(v)=2$, then} &\quad f_v(\bar c)=2\cdot 3 + \frac{q_1(v)+q_2(v)}{3} + \frac{q_3(v)}{2}\le 6+\frac{k-4}{3}+\frac{2}{2}=\frac{k+17}{3},\notag\\
\textrm{if $r(v)=3$, then} &\quad f_v(\bar c)=3\cdot 2 + \frac{q_1(v)+q_2(v)}{3} + \frac{q_3(v)}{2}\le 6+\frac{k-6}{3}+\frac{3}{2}=\frac{k+16\frac{1}{2}}{3},\notag
\end{align}
and for $v=o$, $f_o(\bar c)=\frac{k}{3}$, so we put $C_f=\frac{k+17}{3}$ in Corollary \ref{Gabcor} and obtain the Lemma \ref{rhogen}.
\end{proof}
\begin{proof}[Proof of Theorem \ref{3phgen}]
First of all, note that, due to Remark \ref{g<=gs<=rt}, it is sufficient to show that $\gs(\CG)<\gr(\CG)$ in order to prove the theorem.
\par The calculations are analogous to those in proof of Theorem \ref{3phbasic}. Theorems \ref{thmgs}, \ref{gr>=} and Lemma \ref{rhogen} give us the following bounds, which I denote by $b_1(k)$, $b_2(k)$, respectively:
\begin{align}
\gs(\CG) \le \frac{\frac{k+17}{3}+\sqrt{(\frac{k+17}{3})^2-4(k-1)}}{2} &= b_1(k)\\
\gr(\CG) \ge \frac{k-4+\sqrt{(k-4)^2-4}}{2} &= b_2(k),
\end{align}
hence it is sufficient to prove the inequality $b_1(k)<b_2(k)$ for $k\ge 13$.
Let us check it for $k=13$:
\begin{equation*}
\rt(\CG)\le 10,\quad\textrm{so}\quad
b_1(k) = 5+\sqrt{13} < \frac{9+\sqrt{77}}{2} = b_2(k).
\end{equation*}
and, again, check the inequality
\begin{equation}
\dk b_1(k) \le \dk b_2(k)
\end{equation}
for real $k\ge 13$. The right-hand side derivative is $\ge1$ by \eqref{gr'>=1} and the left-hand side:
\begin{equation*}
\dk\left(\frac{k+17+\sqrt{k^2-2k+325}}{6}\right) = \frac{1}{6}\left(1+\frac{2k-1}{2\sqrt{k^2-2k+325}}\right) \le \frac{1}{6}\left(1+\frac{2k-1}{2k-2}\right) < 1,
\end{equation*}
which finishes the proof.
\end{proof}
\section{The compact right-angled case}\label{secRACpt}
In this section I assume that $\Pi$ is right-angled and compact, hence $G$ is a right-angled Coxeter group acting cocompactly on $\HHH$.
\begin{exmp}\label{dodecah}
One of the simplest examples of such $\Pi$ is the right-angled regular dodecahedron. Its orbit under the action of $G$ is a regular tiling of $\HHH$.
\end{exmp}
\begin{thm*}[recalled Theorem \ref{3phRACpt}]
If $\Pi$ is a compact right-angled polyhedron in $\HHH$, then for $\CG$ the Cayley graph of the reflection group of $\Pi$ with the standard generating set, we have
$$\pc(\CG)<\pu(\CG)$$
for Bernoulli bond and site percolation on $\CG$.
\end{thm*}
The proof is analogous to the proof of Theorem \ref{3phgen}, but I don't use Theorem \ref{gr>=} (rather Theorem \ref{gr}) and I use an additional fact:
\begin{lem}\label{Delta}
In the setting of the above theorem, $\Delta\le\frac{k-1}{2}$.
\end{lem} 
\begin{proof}
Take any face $F$ of $\Pi$ with $\Delta$ sides. Each edge of $F$ belongs to unique face $F'_i$ other than $F$, where $i=1,\ldots,\Delta$ numerates these edges. There are also $\Delta$ edges of $\Pi$ perpendicular to $F$, incident to its vertices. Outside $F$, at the ends of those edges there are attached faces $F''_i,i=1,\ldots,\Delta$ perpendicular to those edges, respectively.
\par Now:
\begin{itemize}
\item the faces $F,F'_1,\ldots,F'_\Delta$ are pairwise distinct because no two edges of $F$ lie on a common line;
\item the faces $F''_i,i=1,\ldots,\Delta$ are pairwise distinct because each of the planes containing them determines uniquely the closest point in the plane of $F$ (disjoint from them) and those points---the vertices of $F$---are pairwise distinct;
\item any face $F''_i,i=1,\ldots,\Delta$ is distinct from any face $F,F'_1,\ldots,F'_\Delta$ because the former are disjoint from $F$ and the latter are not.
\end{itemize}
So all the faces $F,F'_1,\ldots,F'_\Delta,F''_1\ldots,F''_\Delta$ are pairwise distinct, which shows that $k\ge2\Delta+1$ and finishes the proof of the lemma.
\end{proof}
\begin{rem}\label{k>=12}
In the setting of Theorem \ref{3phRACpt}, $k\ge 12$. To see that, note first that $k\ge 11$ because a face of $\Pi$, which is right-angled, must have at least $5$ sides, so in Lemma \ref{Delta} $\Delta\ge 5$. Now, suppose that $k=11$. Then from Lemma \ref{Delta} $\Delta=5$, so all faces of $\Pi$ are pentagons. Since $\Pi$ is compact, we obtain twice the number of its edges counting all sides of each of the faces. That gives $5k$, so $k$ must be even, a contradiction.
\end{rem}
\begin{lem}\label{rhoRACpt}
In the setting of Theorem \ref{3phRACpt}, we have
$$\rt(\CG)\le \frac{k}{2}+3\frac{1}{10}.$$
\end{lem}
\begin{proof}
Note that every vertex of $v\Pi$ is adjacent to exactly $3$ faces of $v\Pi$ because a polyhedral angle with all dihedral angles right must be trihedral. For $o\neq v\in V(\CG)$, consider the faces of $v\Pi$ traversed by edges passing to $v$. Then, from Lemma \ref{Delta}, the number of other faces of $v\Pi$ neighbouring them is at most
\begin{equation}
\left\{
\begin{array}{l@{\text{ if }}l}
\Delta\le\frac{k-1}{2}& r(v)=1\\
2+2(\Delta-3)=2\Delta-4\le k-5& r(v)=2\\
\min(3+3(\Delta-4),k-r(v))\le k-3& r(v)=3\\
\end{array}\right.
\end{equation}
and so is $q_2(v)+q_3(v)$ (because of Claim \ref{eN'}---as in proof of Lemma \ref{q_3gen}).
Now, let $c_1=5,c_2=2,c_3=1$. For $o\neq v\in V(\CG)$, basing on all that and on Lemma \ref{q_3gen} itself, I estimate (similarly to proof of Lemma \ref{rhogen}; note that below, taking first the smallest possible value for $q_1(v)=k-r(v)-q_2(v)-q_3(v)$, then the smallest possible value for $q_2(v)$, indeed gives the upper bounds because $c_1>c_2>c_3$):
\begin{align}
\\
\textrm{if $r(v)=1$, then} &\quad f_v(\bar c)=5 + \frac{q_1(v)}{5} +  \frac{q_2(v)}{2} + 0 \le 5 + \frac{\frac{k-1}{2}}{5} +  \frac{\frac{k-1}{2}}{2} = \frac{7k+93}{20},\notag\\
\textrm{if $r(v)=2$, then} &\quad f_v(\bar c)=2\cdot 2 + \frac{q_1(v)}{5} +  \frac{q_2(v)}{2} + q_3(v) \le 4+\frac{3}{5}+\frac{k-7}{2}+2 = \frac{10k+62}{20},\notag\\
\textrm{if $r(v)=3$, then} &\quad f_v(\bar c)=3\cdot 1 + \frac{q_1(v)}{5} +  \frac{q_2(v)}{2} + q_3(v)\le 3+\frac{k-6}{2}+3 = \frac{10k+60}{20}\notag
\end{align}
and for $v=o$, $f_o(\bar c)=\frac{k}{5}=\frac{4k}{20}$, so we put $C_f=\frac{10k+62}{20}$ in Corollary \ref{Gabcor}, as it is the largest of the bounds above (because $k\ge 12$, due to Remark~\ref{k>=12}), and we obtain the lemma.
\end{proof}
\begin{proof}[Proof of the Theorem \ref{3phRACpt}]
First, note that, due to Remark \ref{g<=gs<=rt}, it is sufficient to show that $\gs(\CG)<\gr(\CG)$ to prove the theorem.
\par Because the conclusion is shown for $k>12$ in general case in Theorem \ref{3phgen}, it is sufficient to show it for $k=12$ (because of Remark \ref{k>=12}).
So I calculate:
\begin{equation*}
\rt(\CG) \le 9\frac{1}{10},\quad\textrm{so}\quad
\gs(\CG) \le \frac{91+\sqrt{3881}}{20} < 4+\sqrt{15} = \gr(\CG).
\end{equation*}
which shows the desired inequality.
\end{proof}
\section{Effectiveness of different approaches}\label{secrems}
Recall Assumption \ref{PiG}.
\par In the following table I give the conditions (lower bounds) for $k$, necessary and sufficient for obtaining the inequality $\pc(\CG)<\pu(\CG)$ by means of Lemma \ref{rhobasic}, \ref{rhogen} and \ref{rhoRACpt}, respectively, and using the value $\gs(\CG)$ or only $\rt(\CG)$, respectively.
\begin{center}
\begin{tabular}{|l||c|c|}
\hline
&Using $\rt$&Using $\gs$\\\hline\hline
Using lem.~\ref{rhobasic}&$k\ge 18$&$k\ge 15$\\\hline
Using lem.~\ref{rhogen}&$k\ge 15$&$k\ge 13$\\\hline
Using lem.~\ref{rhoRACpt}&$k\ge 15$&$k\ge 12$\\\hline
\end{tabular}
\end{center}
One can see that in the above cases the use of $\gs$ improves the bound on $k$ by $2$ or $3$. A similar effect has using Lemma \ref{rhogen} instead of \ref{rhobasic} (both valid for the general case from Assumption \ref{PiG}). Using Lemma \ref{rhoRACpt} (proved here only for compact right-angled $\Pi$) plays a role only for $k=12$ when using $\gs$ (which covers the case of Example \ref{dodecah}).
\par It is worth noting that the methods used in this \partw{} do not give inequality $\pc(\CG)<\pu(\CG)$ for some regular tilings of $\HH$, e.g.~with right-angled regular pentagons.
\section{Estimate for $\gs$---an enhancement of spectral radius}\label{appxgs}
Below, I am going to prove the relations between $\gs$ and $\rt$ in the Theorem \ref{thmgs}:
\begin{thm*}[recalled Theorem \ref{thmgs}]
Let $\CG$ be an arbitrary regular graph of degree $k\ge 2$ (not necessarily simple) with distinguished vertex $o$ and let $\rt=\rt(\CG,o)$, $\gs=\gs(\CG,o)$. Then
\begin{equation}
\gs \le \frac{\rt+\sqrt{\rt^2-4(k-1)}}{2}.
\end{equation}
If, in addition, $\rt(\CG)>2\sqrt{k-1}$ (e.g. when $\CG$ is vertex-transitive and simple and is not a tree), then the estimate becomes an identity:
\begin{equation}
\gs = \frac{\rt+\sqrt{\rt^2-4(k-1)}}{2}.
\end{equation}
\end{thm*}
\begin{proof}
Recall Notations \ref{Cas}.
The following proposition is basic for obtaining the estimate for $\gs$.
\begin{prop}
Let $\CG$ be any (undirected) regular graph of degree $k\ge 2$ (not necessarily simple). Let us choose its vertex $o$ for being its origin. Consider the values $C_n(\CG,o)$ and $a^*_n(\CG,o)$---let us call them $C_n$, $a^*_n$, respectively, for short. Then
\begin{equation*}
C_n = \sum_{d=0}^n a^*_d c^{T_k}(n,d),
\end{equation*}
where for any natural $n$, $c^{T_k}(n,d)$ is the number of paths of length $n$ in $T_k$ joining two points with distance $d$ between them (note that $c^{T_k}(n,d)$ is well-defined).
\end{prop}
\begin{proof}
Let us consider the universal cover
 of $\CG$, which is the infinite $k$-regular tree $T_k$ (with origin being a vertex $\ot$ chosen to cover $o$). Then every path in $\CG$ starting at $o$ lifts to a unique path starting at $\ot$ in $T_k$ which covers it. In particular, every cycle in $\CG$ starting at $o$ lifts to a unique path in $T_k$ joining $\ot$ to some vertex $\ot'$ covering $o$. So, for any natural $n$, if $d(\ot,\ot')$ is the distance between $\ot$ and $\ot'$, then
\begin{equation}
C_n = \quad\voidindop{\sum}{\ot'\in T_k\text{ covering }o}\quad c^{T_k}(n,d(\ot,\ot')) = \sum_{d=0}^n\#\{\ot'\in T_k\text{ covering $o$}: d(\ot,\ot')=d\}c^{T_k}(n,d),\label{eqprfCas}
\end{equation}
Now, the number of $\ot'\in T_k$ at distance $d$ from $\ot$, covering $o$ is exactly the number of geodesic segments of length $d$ from $\ot$ to such vertices $\ot'$. But a cycle $P$ in $\CG$ lifts to a geodesic in $T_k$ if and only if $P$ does not admit any backtracks. Hence
$$C_n=\sum_{d=0}^n a^*_d c^{T_k}(n,d).$$
\end{proof}
\begin{denot1}
For any numerical sequence $x=(x_n)_{n=0}^\infty$, denote by $\F(x)$ its generating function, i.e. power series (of $z\in\mathbb{C}$)
$$\F(x)(z)=\sum_{n=0}^\infty x_n z^n.$$
\end{denot1}
\begin{rem}
In this \appsec, I consider every sum of a power series as a sum (a number, if it is convergent, or $\pm\infty$, if it diverges to $\pm\infty$), rather than value of some holomorphic continuation of it, unless indicated otherwise.
\end{rem}
We calculate $\F(C)(r)$ for $r\ge 0$:
\begin{equation}
\F(C)(r) = \sum_{n=0}^\infty C_n r^n = \sum_{n=0}^\infty \sum_{d=0}^n a^*_d c^{T_k}(n,d) r^n = \sum_{d=0}^\infty a^*_d \sum_{n=d}^\infty c^{T_k}(n,d) r^n.
\end{equation}
Here I am going to use a formula concerning the simple random walk on $T_k$. As $c^{T_k}(n,d)=0$ for $n<d$, and $c^{T_k}(n,d)/k^n$ is the probability of passing from $\ot$ to some fixed $\ot'$ at distance $d$ from $\ot$ in the simple random walk on $T_k$ in $n$ steps, we have for $z\in\mathbb{C}$
\begin{equation}
\sum_{n=d}^\infty c^{T_k}(n,d) z^n = \sum_{n=0}^\infty (c^{T_k}(n,d)/k^n)(kz)^n = G(d,kz),
\end{equation}
where $G(d,\cdot)$ is the Green function for the simple random walk on $T_k$ and a pair of its vertices at distance $d$. 
(Here I treat the Green function as a power series. For the definition, see \cite[1.6]{Woe}.)
From Lemma 1.24 from \cite{Woe} we have
\begin{align}
G(d,kz) &= \frac{2(k-1)}{k-2+\sqrt{k^2-4(k-1)(kz)^2}}\left(\frac{k-\sqrt{k^2-4(k-1)(kz)^2}}{2(k-1)kz}\right)^d=\notag\\
&= \underbrace{\frac{2(k-1)}{k-2+k\sqrt{1-4(k-1)z^2}}}_{A(z)}\Biggl(\underbrace{\frac{1-\sqrt{1-4(k-1)z^2}}{2(k-1)z}}_{f(z)}\Biggr)^d,\label{Green}
\end{align}
(where we always take the standard branch of square root\footnote{Note that here, under the square roots, if $z$ is close to $0$, then we have values close to $1$}, with $\sqrt{1}=1$). Let us introduce notations $A(z),f(z)$, as indicated in the above formula. For $z=0$, formally, there is a problem with $f(z)$ defined by such formula, but
\begin{equation}
f(z) = \frac{2z}{1+\sqrt{1-4(k-1)z^2}}
\end{equation}
(for $z$ such that it exists), so if we put $f(0)=0$, then formula \eqref{Green} is satisfied because $G(d,0)=0^d$. Then the right-hand side of \eqref{Green} is holomorphic on the $0$-centred open ball $B(1/(2\sqrt{k-1}))$, so the equality holds on that ball. Hence, for $r\in [0;1/(2\sqrt{k-1}))$,
\begin{equation}
\F(C)(r) = A(r) \sum_{d=0}^\infty a^*_d f(r)^d = A(r) F(a^*)(f(r))
\end{equation}
Note that here $0<A(r)<\infty$, so for $r\in [0;1/(2\sqrt{k-1}))$,
\begin{equation}\label{FCfinFAfin}
F(C)(r)<\infty \iff F(a^*)(f(r))<\infty.
\end{equation}
\begin{clm}
Put $f(1/(2\sqrt{k-1}))=1/\sqrt{k-1}$.\footnote{Actually, it follows from \eqref{Green}, although $1/(2\sqrt{k-1})$ is a singularity of $f$ as a holomorphic function, so I am avoiding a doubt.} Then $f:[0;1/(2\sqrt{k-1})] \to [0;1/\sqrt{k-1}]$ is strictly increasing and onto.
\end{clm}
\begin{proof}
The strict monotonicity is obvious, as well as the continuity. Since $f(0)=0$ and $f(1/(2\sqrt{k-1}))=1/\sqrt{k-1}$, being onto follows from the Darboux property.
\end{proof}
\begin{rem}
Further in this \appsec, I restrict $f$ to $[0;1/(2\sqrt{k-1})]$ by default.
\end{rem}
\begin{denot1}\label{radconvrg}
For power series $F(z)=\sum_{n=0}^\infty x_n z^n$, by $\r(F)$ or $\r(x)$ I mean its radius of convergence (equal to $1/\mathrm{gr}((x_n)_n)$)
and by $\R(F)$ or $\R(x)$ I mean the set $\{r\ge 0: F(r)<\infty\}$.
\end{denot1}
\begin{clm}
$f(\r(C))=\min\(\r(a^*),1/\sqrt{k-1}\)$. In particular, $f(\r(C)) \le \r(a^*)$.
\end{clm}
\begin{rem}
The left-hand side above makes sense because $\r(C)=1/\rt \le 1/(2\sqrt{k-1})$ (this bound can be obtained by noting that for $n\in\N$, $C_n(\CG,o)\ge c^{T_k}(n,0)=C_n(T_k,\ot)$---from \eqref{eqprfCas}---hence $\rt(\CG)\ge\rt(T_k)$ and $\rt(T_k)=2/\sqrt{k-1}$ by \cite[Lem.~1.24]{Woe}).
\end{rem}
 \begin{proof}[Proof of the claim]
%
%
From \eqref{FCfinFAfin} we have
\begin{align}
[0;1/(2\sqrt{k-1}))\cap \R(F(C)) &= [0;1/(2\sqrt{k-1}))\cap f^{-1}(\R(F(a^*))),
\intertext{so, taking the images in $f$, we have}
\begin{split}
f[[0;1/(2\sqrt{k-1}))\cap \R(F(C))] &= \Bigl([0;1/\sqrt{k-1})\cap \R(F(a^*))\Bigr)\cap \im(f) =\\
&= [0;1/\sqrt{k-1})\cap \R(F(a^*)).
\end{split}
\end{align}
Taking the suprema of those sets gives
$$f(\min(1/(2\sqrt{k-1}),\r(C))) = \min(1/\sqrt{k-1},\r(a^*)),$$
which is
$$f(\r(C)) = \min(1/\sqrt{k-1},\r(a^*))$$
because $\r(C)\le 1/(2\sqrt{k-1})$.
\end{proof}

Now,
\begin{equation*}
f(\r(C)) = f(1/\rt) = \frac{2}{\rt+\sqrt{\rt^2-4(k-1)}},
\end{equation*}
hence
\begin{equation*}
\gs = \frac{1}{\r(a^*)} \le \frac{1}{f(\r(C))} = \frac{\rt+\sqrt{\rt^2-4(k-1)}}{2},
\end{equation*}
which proves the first part of Theorem \ref{thmgs}. For the second part, assume that $\rt>2\sqrt{k-1}$. Then
\begin{equation}
\min(\r(a^*),1/\sqrt{k-1}) = f(\r(C)) < 1/\sqrt{k-1},
\end{equation}
so $\r(a^*)<1/\sqrt{k-1}$, hence
\begin{equation}
f(\r(C)) = \r(a^*),
\end{equation}
which, as above, gives the equality
$$\gs = \frac{\rt+\sqrt{\rt^2-4(k-1)}}{2}.$$
\end{proof}
\section{Estimate for growth rate of a Coxeter group in $\HHH$}\label{appxgr}
In this \appsec\ I prove Theorem \ref{gr>=}. Recall the Assumption \ref{PiG} and notions from Subsection \ref{geomPiG}.
\begin{thm*}[recalled Theorem \ref{gr>=}]
For $k$ and $\CG$ as in Assumption \ref{PiG}, if we assume that $k\ge 6$, then the growth rate of $\CG$
$$\gr(\CG)\ge\frac{k-4+\sqrt{(k-4)^2-4}}{2}.$$
\end{thm*}
\begin{proof}
Let $W$ be the growth series of $G$ with respect to $S$ (see Definition~\ref{grser}).
I use the following formula of Steinberg for $1/W(z)$.
\begin{thm}[Steinberg, \protect{\cite[1.28]{Stein}}, see also \protect{\cite[Sect.~17.1]{Dav} or \cite[Thm.~1]{Kolp}}]\label{Steinb}
$$\frac{1}{W(z)}=\sum_{T\in\mathcal{F}} \frac{(-1)^{\#T}}{W_T(z^{-1})}$$
(as formal Laurent series of $z$), where $\mathcal{F}=\{T\subseteq S:\langle T\rangle\text{ is finite}\}$ (i.e.~the family of spherical subsets of $S$) and $W_T$ is the growth series of the Coxeter group $\langle T\rangle$ with respect to generating set $T$.
\\\qed
\end{thm}
\begin{rem}
The above theorem is formulated in terms of formal Laurent series, as in \cite{Stein} (see notations 1.24 there). (Note that it makes sense because every formal Laurent series admits formal Laurent series reciprocal to it, and for $T\in\mathcal{F}$, $W_T$ are polynomials.) On the other hand, it can be also viewed as a meromorphic function on $\mathbb{C}$ because due to \cite[1.26]{Stein}, it is a rational function. So for the rest of this \appsec, I treat all growth series as meromorphic functions.
Also, I am going to use the convention: $1/0=\infty$, $1/\infty=0$.
\end{rem}
For a power series $\sum_{n=0}^\infty x_n z^n$, let $\mathrm{r}(x)$ be its radius of convergence (as in Definition \ref{radconvrg}).
\begin{lem}[\protect{\cite[par.~7.21]{Titch}}]\label{poscoeff}
Any power series $\sum_{n=0}^\infty x_n z^n$ with $x_n\ge 0$ has a singularity at $\mathrm{r}(x)$.
\\\qed
\end{lem}
\begin{rem}\label{root1/W}
The radius of convergence of $W$, which is the least positive pole of $W$, due to the above lemma, equals the radius of convergence of
$$\sum_{n=0}^\infty\#B(n)z^n = \sum_{n=0}^\infty \sum_{m=0}^n \#S(m)z^n = W(z)/(1-z),$$
which is $1/\gr(\CG)$. Hence, it is the least positive root of the function $1/W$.
\end{rem}
\begin{denot}\label{L}
Let $L$ be the nerve of $(G,S)$ (as an abstract simplicial complex---see Definition \ref{nerve}). Note that $\mathcal{F}$ is the set of all simplices in $L$ plus the empty set, so for $\s\in L$, $W_\s$ should be understood with $\s\subseteq S$. For a simplicial complex $C$ (abstract or geometric one), I denote by $V(C)$, $E(C)$ its sets of vertices and edges, respectively.
Note also that $L$ does not need to be connected.
\end{denot}
In order that the rest of the proof worked, I have to show now the conclusion of the theorem for three exceptional cases, and after that exclude them from consideration (Assumption \ref{L!=s+vs}). Those cases are:
\begin{itemize}
\item the nerve $L$ contains only vertices (i.e.~no edges);
\item there is only $1$ edge in $L$;
\item there is only $1$ triangle in $L$ and no edges outside that triangle
\end{itemize}
and i consider them at once. Namely, in each of those cases, there is a set $S'\subseteq S$ of $k-3$ isolated vertices of $L$, which generates a subgroup $G'<G$ isomorphic to $\mathbb{Z}_2^{*(k-3)}$ (each of the free factors of $\mathbb{Z}_2^{*(k-3)}$ is generated by some $s\in S'$). Hence, the Cayley graph $\CG'$ of $(G',S')$, which is the infinite $(k-3)$-regular tree (with growth rate $k-4$), embeds in $\CG$, so
\begin{equation}
\gr(\CG) \ge \gr(\CG') = k-4 \ge \frac{k-4+\sqrt{(k-4)^2-4}}{2},
\end{equation}
which finishes the proof for the three cases above. Now, I exclude those cases:
\begin{assm}\label{L!=s+vs}
For the rest of the proof, let $L$ contain more that $1$ edge and not have exactly $1$ triangle containing all the edges of $L$. (In other words: $L$ doesn't amount to a collection of isolated vertices and single $0$-, $1$- or $2$-simplex.)
\end{assm}
Using Claim \ref{no4inL} and Theorem \ref{Steinb}, I calculate a formula for $1/W(t)$:
\begin{equation}\label{1/W3sums}
\frac{1}{W(t)} = 1 - \quad\lvoidindop{\sum}{v\text{---a vertex of }L} \frac{1}{W_{\{v\}}(t^{-1})} + \quad\lvoidindop{\sum}{e\text{---a side of }L} \frac{1}{W_e(t^{-1})} - \qquad\lvoidindop{\sum}{f\text{---a $2$-simplex of }L} \frac{1}{W_f(t^{-1})}.
\end{equation}
\begin{rem}
Now, the main idea of the proof is to use the ,,right-angled compact'' counterpart of $W$. Namely, imagine a right-angled Coxeter group $G\rb$ with generating set $S\rb$ of $k$ elements and with nerve $L\rb$ which is a flag triangulation of $\SS$ (as in Theorem \ref{gr}; I use a convention of putting $\cdot\rb$ on elements concerning the ``right-angled compact version'' of $G$). Then $L\rb$ has $k$ vertices; let $f_1\rb$, $f_2\rb$ be the numbers of its edges and triangles, respectively. Then, they are uniquely determined by $k$, using Euler formula for such triangulation:
\begin{align}
k-f_1\rb+f_2\rb=2.
\end{align}
We have also $2f_1\rb=3f_2\rb$, so
\begin{align}
2k - 3f_2\rb + 2f_2\rb = 4,\notag\\
f_2\rb = 2(k - 2),\text{\quad and}\label{f_2rb}\\
f_1\rb = \frac{3}{2}f_2\rb = 3(k - 2).\label{f_1rb}
\end{align}
So, if $W\rb$ is the growth series of $(G\rb,S\rb)$, then, using formula \eqref{1/W3sums}, we have
\begin{equation}
\frac{1}{W\rb(t)} = 1 - \frac{k}{t^{-1}+1} + \frac{f_1\rb}{(t^{-1}+1)^2} - \frac{f_2\rb}{(t^{-1}+1)^3},\label{1/Wrbsum}
\end{equation}
as for $\mathbb{Z}_2^n$ with generating set consisting of the generators of the factors (copies of $\mathbb{Z}_2$), its growth series is $(1+z)^n$ (for $n=1$, it is obvious, for other $n$, see \cite[17.1.13]{Dav}); so from \eqref{f_1rb} and \eqref{f_2rb}
\begin{equation}
\frac{1}{W\rb(t)} = \frac{t-1}{(t+1)^3}(-t^2 + (k-4)t-1).\label{1/Wrbprod}
\end{equation}
As $k\ge 6$, all the roots of $1/W\rb(t)$ are non-negative and the least one is
\begin{equation}
\frac{k-4-\sqrt{(k-4)^2-4}}{2}.
\end{equation}
In this proof, I am not going to use existence of such group $G\rb$. What I use is only the fact that the right-hand side of the inequality of the Theorem \ref{gr>=} (which coincides with formula for the growth of $(G\rb,S\rb)$) is indeed reciprocal of that least positive root of $1/W\rb(t)$ (similarly to Remark \ref{root1/W}). So I define just the function $1/W\rb(t)$ by \eqref{1/Wrbsum} (in terms of $f_1\rb,f_2\rb$ defined by \eqref{f_1rb} and \eqref{f_2rb}) and I want to bound the least positive root of $1/W$ by that of $1/W\rb$ (from above), which is less or equal to $1$. Note that $1/W(0),1/W\rb(0)=1>0$ and $1/W$, $1/W\rb$ are continuous on $[0;1]$ (which is easily seen from Theorem~\ref{Steinb}, formulae \eqref{1/Wrbsum}, \eqref{1/Wrbprod} and for $1/W$ at $0$---from the fact that $\r(W)>0$), so, to get that bound, it suffices to prove that
\begin{equation}\label{ineqWWrb}
1/W(t)\le 1/W\rb(t)
\end{equation}
for $t\in(0;1]$. I do it in Claims \ref{clWsum} and \ref{clsumWrb}.
\end{rem}
\begin{prop}
The nerve $L$ (in the sense of the geometric realisation) embeds into sphere $\SS$.
\end{prop}
\begin{rem}
A combinatorial version of the idea of this fact in a bit different setting is present in \cite[Example 7.1.4]{Dav}.
\end{rem}
\begin{proof}
First, consider $\HHH$ in the Klein unit ball model (see e.g.~\cite[Chapter I.6]{BH}). The ideal boundary $\bdi\HHH$ of $\HHH$ is the unit sphere in this model and naturally compactify the unit ball $\HHH$ resulting in $\cHHH$---the closed unit disc. Let for $A\subseteq\HHH$, $\bdi A$ denote $\cl{A}\sm A$, where the closure $\cl{A}$ is taken in $\cHHH$. In this context, note that $\B=\bdi\Pi\cup\bdt\Pi$ is homeomorphic to $\SS$. Indeed, here $\Pi$ is a convex set (as a subset of $\mathbb{R}^3$), in particular, it is star-convex  with regard to some point in $\int\Pi$.
\par Now, i construct an embedding $\Phi$ of $L$ into $\B$ (which completes the proof). Let us fix an interior point $c_f$ or $c_e$, respectively, of each face $f$ and of each edge $e$ of $\Pi$. Then, for $s\in S=V(L)$, I put $\Phi(s)=c_f$, where $f$ is the face of $\Pi$ corresponding to $s$. Next, I embed the edges of $L$: for $e=\{s,t\}\in E(L)$, the planes of faces corresponding to $s$ and $t$ have non-empty intersection (in order that the corresponding reflections generated a finite group), so by Proposition \ref{ePi} those faces are neighbours---let $e\d$ denote their common edge\footnote{Not to be confused with $e\d$ from Notation \ref{ed}.}. I join $\Phi(s)$ with $c_{e\d}$ and $\Phi(t)$ with $c_{e\d}$ by geodesic segments (lying in the faces corresponding to $s$, $t$, respectively). The union of those two segments is a path in $\B$ joining $\Phi(s)$ with $\Phi(t)$---let $\Phi(e)$ be that path. Note that $\Phi(e_1)$ and $\Phi(e_2)$ are disjoint off the endpoints for edges $e_1\neq e_2$ of $L$, so now $\Phi$ is an embedding of $1$-skeleton of $L$ into $\B$. Now, each $2$-simplex $\s$ of $L$ corresponds to three generators $s_i\in S$ ($i=1,2,3$). Let $e_i, i=1,2,3$ be edges of $\s$. Then points $\Phi(s_i)$ are pairwise joined by paths $\Phi(e_i)$. Because $\langle s_1,s_2,s_3\rangle$ is finite, the planes of faces corresponding to $s_1,s_2,s_3$ have non-empty intersection, due to Corollary II.2.8 from \cite{BH} ($\HHH$ is a $\mathrm{CAT}(0)$ space). So, from Proposition \ref{vPi}, those faces share a vertex $p$. From Jordan-Schoenflies theorem, $\bigcup_{i=1}^3\Phi(e_i)$ disconnects $\B$ into two open discs (up to homeomorphism) with boundary $\bigcup_{i=1}^3\Phi(e_i)$. So I take the closure of the component of $\B\sm\bigcup_{i=1}^3\Phi(e_i)$ containing $p$---call it $D$. Obviously, it is an embedding of $\s$ in $\B$ (with sides $\Phi(e_i)$ and vertices $\Phi(s_i)$). Put $\Phi(\s)=D$. Then, indeed, $\Phi$ is an embedding of $L$ because for different $2$-simplices $\s_1,\s_2$ of $L$, the interiors (taken in $\B$) of $\Phi(\s_1)$ and $\Phi(\s_2)$ are disjoint.
\end{proof}
\begin{df}\label{L1}
Let us denote by $L_1$ the embedding of $1$-skeleton of $L$ in $\SS$ from the above proposition. I call closure of a component of its complement in $\SS$ its \emd{face}. I am going to identify $1$-skeletons of $L$ and of its embedding in $\SS$. Also, I declare a face of $L_1$ to belong to $L$ (or to be a $2$-simplex of $L$) iff it equals the embedding of some $2$-simplex of $L$.
For $f$---a face of $L_1$, I denote by $\deg(f)$ the \emd{number of sides} of $f$, i.e.~number of edges lying in $f$, but with edges crossing interior of $f$ counted twice.
\end{df}
\begin{denot}\label{f's}
Let $f_0,f_1,f_2$ be the numbers of vertices, edges and faces of $L_1$, respectively, $\ft2$---of $3$-sided faces of $L_1$ and $\fL2$---the number of 2-simplices of $L$. For $e$---an edge of $L$, let $m(e)$ be the order of $s_1s_2$ in $G$, where $e=\{s_1,s_2\}$ (so that the dihedral angle of $\Pi$ at the edge corresponding to $e$ is $\pi/m(e)$). (Note that $f_0=k$.)
\end{denot}
\begin{rem}\label{WvWe}
Note that for a vertex $v$ of $L$,
\begin{equation}\label{Wv}
W_{\{v\}}(z)=z+1
\end{equation}
and, by an easy exercise, for an edge $e$ of $L$,
\begin{equation}
W_e(z)=(z+1)(z^{m(e)-1}+\cdots+z+1).
\end{equation}
\end{rem}
\begin{denot}
To ease our work with such polynomials, for $n\in\N$, I define polynomials
\begin{equation*}
[n](z)=\underbrace{z^{n-1}+\cdots+z+1}_{n\text{ summands}}
\end{equation*}
and for $n_1,\ldots,n_m\in\N$,
\begin{equation*}
[n_1,\ldots,n_m](z)=\prod_{i=1}^m[n_i](z).
\end{equation*}
I will drop ``$(z)$'' if the argument is obvious.
\end{denot}
For $t\in(0;1]$, using this notation, Remark \ref{WvWe} and \eqref{1/W3sums}, we have
\begin{align}
\frac{1}{W(t)}&=1 - \frac{f_0}{t^{-1}+1} + \sum_{e\text{---a side of }L} \frac{1}{[2,m(e)](t^{-1})} - \sum_{f\text{---a $2$-simplex of }L} \frac{1}{W_f(t^{-1})}
\intertext{and, reordering the sums into a sum taken over the faces of $L_1$}
&=1 - \frac{f_0}{t^{-1}+1} + \sum_{f\text{---a face of }L_1} \underbrace{\left(\sum_{e\text{---a side of }f}\frac{1}{2}\frac{1}{[2,m(e)](t^{-1})} - \frac{\ind_{f\in L}}{W_f(t^{-1})}\right)}_{A(f)}.\label{A(f)}
\end{align}
It will be convenient to work with the expression $A(f)$ defined above (for $f$ a face of $L_1$).
\begin{rem}\label{remsidessum}
The variable $e$ of the last summation above runs through the sides of $f$, which means that some edge may be counted twice there.
\end{rem}
\begin{clm}\label{clWsum}
For $t\in(0;1]$,
\begin{equation*}
\frac{1}{W(t)} \le 1 - \frac{f_0}{t^{-1}+1} + \frac{f_1}{(t^{-1}+1)^2} - \frac{\ft2}{(t^{-1}+1)^3}.
\end{equation*}
\end{clm}
\begin{proof}
It suffices to prove the bound
\begin{equation}
A(f) \le \underbrace{\frac{\deg(f)}{2(t^{-1}+1)^2} - \frac{\ind_{f\text{---3-sided}}}{(t^{-1}+1)^3}}_{B(f)} \label{B(f)}
\end{equation}
for $f$---a face of $L_1$, to take a sum over all such $f$'s, and to make calculation analogous to deriving \eqref{A(f)} from \eqref{1/W3sums} (with similar rendering of sums) for the right-hand side of the claim (written using summations on edges and $3$-sided faces of $L_1$), to complete the proof.
So I estimate $A(f)$ for $f$ a face of $L_1$, considering three cases:
\begin{case}$f$ is not 3-sided\end{case}
Then
\begin{equation}
A(f) \le \sum_{e\text{---a side of }f}\frac{1}{2[2,2](t^{-1})} - 0 = B(f).
\end{equation}
I used here the fact that $[n]\le[m]$ for $n\le m$, on $[0;\infty)$.
\begin{case}$f$ is 3-sided, but outside $L$\end{case}
Note that because of Assumption \ref{L!=s+vs}, $\overline{\SS\sm f}$ is not a triangle of $L$. Hence, vertices of $f$ generate an infinite subgroup of $G$. It means that the dihedral angles between the faces of $\Pi$ corresponding to the vertices of $f$, sum up to a number not greater than $\pi$ (see e.g.~\cite[Exercise 6.8.10]{Dav} combined with \cite[Theorem 6.8.12]{Dav} for an explanation). Thus,
\begin{equation}
\sum_{e\text{---a side of }f}\frac{1}{m(e)}\le 1.
\end{equation}
Note also that for an edge $e$ and $t^{-1}>1$,
\begin{align}
[2,m(e)](t^{-1}) &= ((t^{-2})^{\frac{m(e)}{2}}-1)(t^{-1}+1)/(t^{-1}-1) \ge \notag\\
&\ge \voidindunderbrace{{\textstyle\frac{m(e)}{2}}(t^{-2}-1)}{\text{value of tangent to $z^{\frac{m(e)}{2}}-1$ at $z=1$ taken at $t^{-2}$}}(t^{-1}+1)/(t^{-1}-1) = {\textstyle\frac{m(e)}{2}}(t^{-1}+1)^2.
\end{align}
This is also true for $t=1$. Hence,
\begin{equation}
\begin{split}
A(f) &\le \sum_{e\text{---a side of }f}\frac{1}{m(e)(t^{-1}+1)^2} - 0 \le\\
&\le \frac{1}{(t^{-1}+1)^2} < \frac{3}{2}\frac{1}{(t^{-1}+1)^2} - \frac{1}{2}\frac{1}{(t^{-1}+1)^3} = B(f),
\end{split}
\end{equation}
as $t^{-1}+1 > 1$.
\begin{case}$f$ is a triangle from $L$\end{case}
Because here $f\subseteq S$ is spherical, the Steinberg formula (Theorem \ref{Steinb}) yields
\begin{align}
\frac{1}{W_f(t)} &= 1 - \frac{3}{t^{-1}+1} + \sum_{e\text{---a side of }f} \frac{1}{[2,m(e)](t^{-1})} - \frac{1}{W_f(t^{-1})} =\\
&= 1 - \frac{3}{t^{-1}+1} + 2A(f) + \frac{1}{W_f(t^{-1})}.\label{Steinb f in L}
\end{align}
Let $m(f)$ be the length of the longest element of $\langle f\rangle$ (i.e.~$m(f)=\max_{g\in\langle f\rangle}l(g)$, using Definition \ref{lgth}).
By Lemma 17.1.1.~
in \cite{Dav}, we have $W_f(t)=t^{m(f)}W_f(t^{-1})$, so
\begin{align}
\frac{t^{-m(f)}-1}{W_f(t^{-1})} = 1 - \frac{3}{t^{-1}+1} + 2A(f).\label{A(f)(W_f)final}
\end{align}
Consider ``right-angled counterparts'' of $\langle f\rangle$ and $W_f$, which are $\mathbb{Z}_2^3$ and its growth series $W_f\rb$, respectively, the latter given by
\begin{equation}
\frac{1}{W_f\rb(t)} = 1 - \frac{3}{t^{-1}+1} + \frac{3}{(t^{-1}+1)^2} - \underbrace{\frac{1}{(t^{-1}+1)^3}}_{1/W_f\rb(t^{-1})}
\end{equation}
(as in first equality of \eqref{Steinb f in L}; see also explanation of \eqref{1/Wrbsum}). Note that computations analogous to \eqref{Steinb f in L} through \eqref{A(f)(W_f)final} can be made for $1/W_f\rb$ and $B(f)=\frac{3}{2(t^{-1}+1)^2} - \frac{1}{(t^{-1}+1)^3}$ in place of $1/W_f$ and $A(f)$, giving
\begin{equation}
\frac{t^{-3}-1}{(t^{-1}+1)^3} = \frac{t^{-3}-1}{W_f\rb(t^{-1})} = 1 - \frac{3}{t^{-1}+1} + 2B(f).
\end{equation}
So the inequality $A(f)\le B(f)$ is equivalent to
\begin{align}
\frac{t^{-m(f)}-1}{W_f(t^{-1})} &\le \frac{t^{-3}-1}{(t^{-1}+1)^3},
\intertext{which is obvious for $t=1$, and for $t<1$, it is equivalent to}
[m(f),2,2,2](t^{-1}) &\le W_f(t^{-1})[3](t^{-1}).\label{ineqW_fbeftable}
\end{align}
To prove this, I will need the following fact:
\begin{prop}
Let $a,b$ be natural numbers such that $a\le b+1$. Then for any natural $d\le a$,
$$[a-d,b+d](t)\le[a,b](t)$$
for any $t\ge 0$.
\end{prop}
\begin{proof}
Let $t\ge 0$. First, I prove the conclusion for $d=1$:
\begin{align*}
[a,b](t)-[a-1,b+1](t) &= [a,b] - ([a]-t^{a-1})([b]+t^b) = t^{a-1}[b] - [a]t^b + t^{a+b-1} =\\
&= ([a+b-1]-[a-1]) - ([a+b]-[b]) + t^{a+b-1} =\\
&= -t^{a+b-1} - [a-1] + [b] + t^{a+b-1} = [b] - [a-1] \ge 0.
\end{align*}
The general case follows by induction on $d$.
\end{proof}
Below, I use the Table 1.~from \cite{KelPer}, giving convenient formulae for $W_f$, which is the growth series of the reflection group of a Coxeter triangle on $\SS$, i.e.~of $\mathrm{G}_2^{m(f)-1}\times\mathbb{Z}_2$, where $m(f)\ge 3$, $\mathrm{A}_3$, $\mathrm{B}_3$ or $\mathrm{H}_3$ (with the standard sets of generators). I consider those four cases below (for $t<1$), using the above proposition and the fact that $m(f)=\deg(W_f)$. Here, while $t^{-1}$ is still argument of the polynomials, I drop it for brevity.
\begin{center}
\begin{tabular}{|c|c|l|}
\hline
$\langle f\rangle\cong$ & $W_f$
 & \multicolumn{1}{|c|}{Proof of \eqref{ineqW_fbeftable}}\\ \hline
$\mathrm{G}_2^{m(f)-1}\times\mathbb{Z}_2$ & $[2,2,m(f)-1]$ & $[2,2,2,m(f)]\le[2,2,3,m(f)-1]=[3]W_f$\rule{0pt}{3ex}\\
$\mathrm{A}_3$ & $[2,3,4]$ & $[2,2,2,m(f)]=[2,2,2,6]\le[2,2,3,5]\le[2,3,3,4]=[3]W_f$\rule{0pt}{3ex}\\
$\mathrm{B}_3$ & $[2,4,6]$ & $[2,2,2,m(f)]=[2,2,2,9]\le[2,2,4,7]\le[2,3,4,6]=[3]W_f$\rule{0pt}{3ex}\\
$\mathrm{H}_3$ & $[2,6,10]$ & $[2,2,2,m(f)]=[2,2,2,15]\le[2,2,7,10]\le[2,3,6,10]=[3]W_f$\rule[-2ex]{0pt}{5ex}\\ \hline
\end{tabular}
\end{center}
\par That finishes the proof of inequality \eqref{B(f)} in case of $f$ a triangle from $L$, hence completes the proof of the claim.
\end{proof}
\begin{clm}\label{clsumWrb}
For $t\in(0;1]$,
\begin{equation*}
1 - \frac{f_0}{t^{-1}+1} + \frac{f_1}{(t^{-1}+1)^2} - \frac{\ft2}{(t^{-1}+1)^3} \le \frac{1}{W\rb(t)}
\end{equation*}
(left-hand side above is exactly the right-hand side in Claim \ref{clWsum}).
\end{clm}
\begin{proof}
From \eqref{1/Wrbsum}, to prove the claim, it suffices to show that
\begin{equation}
\frac{f_1}{(t^{-1}+1)^2} - \frac{\ft2}{(t^{-1}+1)^3} \le \frac{f_1\rb}{(t^{-1}+1)^2} - \frac{f_2\rb}{(t^{-1}+1)^3}.
\end{equation}
Euler formula for $L_1$ gives
\begin{equation}
f_1\le f_0+f_2-2\label{ineqEulL_1}
\end{equation}
(recall that $L_1$ may be disconnected), so the desired inequality above is implied by
\begin{align}
\frac{f_0+f_2-2}{(t^{-1}+1)^2} - \frac{\ft2}{(t^{-1}+1)^3} &\le \frac{3(f_0-2)}{(t^{-1}+1)^2} - \frac{2(f_0-2)}{(t^{-1}+1)^3},\\
\frac{f_2}{(t^{-1}+1)^2} - \frac{\ft2}{(t^{-1}+1)^3} &\le \frac{2(f_0-2)t^{-1}}{(t^{-1}+1)^3}.\label{clsumWrb.ineqf_2}
\end{align}
Now, counting sides in every face of $L_1$ gives
\begin{equation}
2f_1 \;\ge\; 3\ft2 + 4(f_2-\ft2) \;=\; 4f_2-\ft2,
\end{equation}
so from \eqref{ineqEulL_1}
\begin{align}
2(f_0+f_2-2) \;&\ge\; 4f_2-\ft2,\notag\\
2f_0-2f_2-4 \;&\ge\; -\ft2,\label{ineqft2}
\end{align}
so inequality \eqref{clsumWrb.ineqf_2} holds, provided that
\begin{align}
\frac{f_2}{(t^{-1}+1)^2} + \frac{2f_0-2f_2-4}{(t^{-1}+1)^3} \;&\le\; \left. \frac{2(f_0-2)t^{-1}}{(t^{-1}+1)^3}\quad \right|\cdot(t^{-1}+1)^3\\
f_2(t^{-1}-1) \;&\le\; 2(f_0 - 2)(t^{-1}-1),\notag\\
0 \;&\le\; (2f_0 - 4 - f_2)(t^{-1}-1),\notag
\end{align}
which is true because $t^{-1}\ge 1$ and, due to \eqref{ineqft2}, $2f_0-4-f_2 \ge f_2-\ft2 \ge 0$.
\end{proof}
Claims \ref{clWsum} and \ref{clsumWrb} give the inequality \eqref{ineqWWrb} and hence complete the proof of the theorem.
\end{proof}

\newcommand{\G}{G} 
\renewcommand{\S}{{\mathbb{S}}}
\newcommand{\cHd}{{\widehat{\mathbb{H}}^d}}
\newcommand{\hc}[1]{\accentset{\frown}{#1}}
\newcommand{\hcHd}{{\hc{\mathbb{H}}^d}}
\newcommand{\hbd}{\eth}
\newcommand{\pbd}{\pi} 
\renewcommand{\R}{{\mathbb{R}}}
\renewcommand{\L}{\mathcal{L}}
\newcommand{\NP}{{\mathcal{N}}}
\newcommand{\p}{{p_{1/2}}}
\newcommand{\hR}{^{(h,R)}}
\newcommand{\ho}{o\hR}
\newcommand{\IOd}{(0;1]\times O(d)}
\newcommand{\gt}{\tilde{g}}
\renewcommand{\d}{\delta}
\newcommand{\conn}{\leftrightarrow}
\newcommand{\then}{\Longrightarrow}

\newcommand{\Bo}{B^1_{r_0}}
\newcommand{\bo}{\mathbf{o}}
\newcommand{\vh}{v_\mathrm{h}}
\newcommand{\tends}[1]{\xrightarrow[#1]{}}
\newcommand{\RhR}{\mathscr{R}\hR}

\newcommand{\bigcupincr}{\bigcup}

\newlength{\bigcapKXwidth}
\settowidth{\bigcapKXwidth}{$\displaystyle{\bigcap_{K\subseteq X}}$}
\newcommand{\bigcapKX}{\makebox[\bigcapKXwidth]{$\displaystyle{\bigcap_{\substack{K\subseteq X\\ K\text{ -- compact}}}}$}}

\chapter{One-point boundaries of ends of infinite clusters}\label{ch1ptbd}
\section{Introduction}

In this chapter I am concerned with boundaries of ends of the percolation clusters on graphs ``naturally'' embedded in $\Hd$ with $d\ge 2$. As it is mentioned in the introduction to the \work, I started to investigate them hoping that it would be useful for discovering an additional phase transition between $\pc$ and $\pu$. This phase transition is defined as the threshold $\p$ between two ranges of values of the \Bp\ parameter $p$: those for which all the infinite clusters have only one-point boundaries of ends, and the remaining values of $p$. So,  the question is if $\pc<\p<\pu$ e.g.\ for some natural tiling graphs in $\Hd$ for $d\ge 3$. I define the boundaries of ends of a cluster in $\Hd$ as follows:
\begin{denot}
For any topological space $X$, by $\sint_X$ and $\overline{\makebox[1em]{$\cdot$}}^X$ I mean operations of taking interior and closure, respectively, in the space $X$. (I use such notations especially when $X$ is a subspace of another topological space.)
\end{denot}
\begin{df}\label{dfbd}
Let $X$ be a completely regular Hausdorff ($\mathrm T_{3\frac{1}{2}}$), locally compact topological space. Then:
\begin{itemize}
\item An \emd{end} of a subset $C\subseteq X$ is a function $e$ from the family of all compact subsets of $X$ to the family of subsets of $C$ such that:
\begin{itemize}
 \item for any compact $K\subseteq X$ the set $e(K)$ is one of the component of $C\setminus K$;
 \item for $K\subseteq K'\subseteq X$ -- both compact -- we have
 $$e(K)\supseteq e(K').$$
\end{itemize}
\end{itemize}
Now let $\cX$ be an arbitrary compactification of $X$. Then
\begin{itemize}
\item The \emd{boundary} of $C\subseteq X$ is the following:
$$\bdi C= \clcX{C}\setminus X.$$
\item Finally the \emd{boundary of an end} $e$ of $C\subseteq X$ is
$$\bdi e=\bigcapKX\bdi e(K).$$
\end{itemize}
I also put $\c C = \clcX{C}$.\uwagaj{Może zrezygnować z ogólnego sformułowania def. i od razu definiować dla $\Hd$?} Whenever I mean a boundary in the usual sense (taken in $\Hd$ by default), I denote it by $\bdt$ to distinguish it from $\bdi$.
\par I use these notions in the context of the hyperbolic space $\Hd$, where the underlying compactification is the compactification $\cHd$ of $\Hd$ by its set of points at infinity\footnote{For $\Hd$, it is the same as its Gromov boundary---see \cite[Section III.H.3]{BH}.} (see \cite[Definition II.8.1]{BH}). The role of $C$ above will be played by percolation clusters in $\Hd$.
\par Let also $\bdi\Hd$ denote $\cHd\sm\Hd$, the set of points at infinity. If $\Hd$ is considered in its Poincar\'e disc model\footnote{It is called also Poincar\'e ball model.}, $\bdi\Hd$ is naturally identified with the boundary sphere of the Poincar\'e disc.
\end{df}
\begin{rem}\label{smarthat}
In this chapter, whenever I consider a subset of $\Hd$ denoted by a symbol of the form e.g.\ $C_x^y(z)$, I use the notation  $\c C_x^y(z)$ for its closure in $\cHd$ instead of $\c{C_x^y(z)}$, for aesthetic reasons.
\end{rem}
In this chapter I give a sufficient condition for $p$-\Bbp\ to admit infinite clusters with only one-point boundaries of ends, for a large class of transitive graphs embedded in $\Hd$. Namely, that sufficient condition is ``$p<p_0$'', where $p_0$ is a threshold defined in Definition \ref{dfp0}. The key part of the proof is an adaptation of the proof of Theorem (5.4) from \cite{Grim}, which in turn is based on \cite{Men}.
\par In the next section I formulate the assumptions on the graph and the main theorem.

\subsection{The graph and the sufficient condition}

\begin{assm}\label{assmG}
Throughout this \partw{} I assume that $\G$ is a connected (simple) graph embedded in $\Hd$, such that:
\begin{itemize}
\item its edges are geodesic segments;
\item the sets of vertices and edges are locally finite;
\item it is {transitive} (under the action of some group of isometries of $\Hd$---see the definition below).
\end{itemize}
Let us also pick a vertex $o$ (for ``origin'') of $\G$ and fix it once and for all.
\end{assm}
\begin{df}
Throughout this chapter, for any graph $\G$ embedded in arbitrary metric space, I call this embedded graph \emd{transitive under isometries} if some group of isometries of the space acts on $\G$ by graph automorphisms transitively on its set of vertices.
\par By a \emd{simple} graph I mean a graph without loops and multiple edges.
\end{df}
\begin{rem}
Local finiteness in the above assumption means that every compact subset of $\Hd$ meets only finitely many vertices and edges of the embedded graph. Note that by these assumptions, $V(\G)$ is countable, $\G$ has finite degree and is a closed subset of $\Hd$.
\end{rem}

\begin{df}\label{dfp0}
For $v\in V(\G)$, by $C(v)$ I mean the percolation cluster of $v$ in $\G$.
Let $\NP(\G)$ (for ``null''), or $\NP$ for short, be defined by
\begin{equation}\label{assmnull}
\NP(\G) = \{p\in[0;1]:(\forall x\in\bdi\Hd)(\Pr_p(x\in\bdi C(o))=0)\}
\end{equation}
and put
\begin{equation}
p_0=p_0(\G)=\sup\NP(\G).
\end{equation}
\end{df}
\begin{rem}
In words, $\NP$ is the set of parameters $p$ of \Bbp{} on $G$ such that~no point of $\bdi\Hd$ lies in boundary of the cluster of $o$ with positive \pbb.
Note that $\NP$ is an interval (I do not know whether it is right-open or right-closed) because the events $\{x\in\bdi C(o)\}$ for $x\in\bdi\Hd$ are all increasing (see Definition \ref{incr}), so $\Pr_p(x\in\bdi C(o))$ is a non-decreasing function of $p$ (see \cite[Thm. (2.1)]{Grim}). That allows us to think of $p_0$ as the point of a phase transition.
\par I am going to make a few more remarks concerning the above definition and how $p_0$ may be related to the other percolation thresholds in Section \ref{secremsp0}.
\end{rem}

Now, I formulate the main theorem:
\begin{thm}\label{mainthm}
Let $G$ satisfy the Assumption \ref{assmG}. Then, for any $0\le p<p_0$, a.s. every cluster in $p$-\Bbp\ on $\G$ has only one-point boundaries of ends.
\end{thm}
The key ingredient of the proof of this theorem is Lemma \ref{corMen}, which is a corollary of Theorem \ref{lemMen}. Both of the latter are quite interesting in their own right. They are presented (along with a proof of Lemma \ref{corMen}) in separate Section \ref{secexpdec}. The elaborate proof of Theorem \ref{lemMen}, rewritten from the proof of Theorem (5.4) in \cite{Grim}, is deferred to Section \ref{prflemMen}. The proof of this theorem itself, is presented in Section \ref{secprfmainthm}.

\subsection{Remarks on the sufficient condition}\label{secremsp0}

In this section I give some remarks on the threshold $p_0$ and on the events $\{x\in\bdi C(o)\}$ (used to define $\NP$).
\begin{df}
For $A,B\subseteq\Hd$, let $A\conn B$ be the event that there is an open path in the percolation process (given by the context) intersecting both $A$ and $B$. I say also that such path \emd{joins} $A$ and $B$. If any of the sets is of the form $\{x\}$, I write $x$ instead of $\{x\}$ in that formula and those phrases.
\end{df}
\begin{rem}\label{nullmeasb}
For $x\in\bdi\Hd$, the configuration property $\{x\in\bdi C(o)\}$ is indeed a (measurable) random event.
Even more: the set
\begin{equation}
A = \{(x,\omega)\in \bdi\Hd\times 2^{E(\G)}:x\in\bdi(C(o))(\omega)\}
\end{equation}
is measurable in the product $\bdi\Hd\times 2^{E(\G)}$ (where the underlying $\s$-field on $\bdi\Hd$ is the $\s$-field of Borel sets). To prove it, let us introduce a countable family $(H_n)_{n\in\N}$ of half-spaces such that~the family of open discs
$$\{\sint_{\bdi\Hd}\bdi H_n:n\in\N\}$$
is a base of the topology on $\bdi\Hd$. Then, let us rewrite the condition defining $A$:
\begin{eqnarray}
x\in \bdi C(o) &\iff& (\forall n)(x\in\sint_{\bdi\Hd}\bdi H_n \then C(o)\cap H_n\neq\emptyset) \iff\\
&\iff& (\forall n)\bigl(\neg(x\in\sint_{\bdi\Hd}\bdi H_n) \lor\\
&&\lor\, \bigl(x\in\sint_{\bdi\Hd}\bdi H_n \land (\exists v\in V(\G)\cap H_n)(o\conn v)\bigr)\bigr),\notag
\end{eqnarray}
which is a measurable condition, as the sets
$$\{(x,\omega)\in \bdi\Hd\times 2^{E(\G)}: x\in\sint_{\bdi\Hd}\bdi H_n\} = \sint_{\bdi\Hd}\bdi H_n\times 2^{E(\G)}$$
and
$$\{(x,\omega)\in \bdi\Hd\times 2^{E(\G)}: o\conn v\textrm{ in }\omega\}$$
are measurable for $n\in\N$, $v\in V(\G)$.
\par For $x\in\bdi\Hd$, the measurability of the event $\{x\in\bdi C(o)\}$ follows the same way if we treat it as the $x$-section of $A$:
\begin{equation}
\{x\in\bdi C(o)\} = \{\omega:(x,\omega)\in A\}.
\end{equation}

\end{rem}

\begin{rem}\label{ineqp0}
The threshold $p_0$ is bounded as follows:
$$\pc\le p_0 \le\pu.$$
The inequality $\pc\le p_0$ is obvious and the inequality $p_0\le\pu$ can be shown as follows: if $p$ is such that $\Pr_p$-a.s.~there is a unique infinite cluster in $\G$, then with some \pbb\ $a>0$, $o$ belongs to the infinite cluster and by BK-inequality (see Theorem \ref{BKineq}), for any $v\in V(\G)$,
$$\Pr_p(o\conn v)\ge a^2.$$

Take $x\in\bdi\G$. Choose a decreasing (in the sense of set inclusion) sequence $(H_n)_n$ of half-spaces such that $\bigcap_{n=1}^\infty \sint_{\bdi\Hd}\bdi H_n=\{x\}$. Since $x\in\bdi\G$, we have $V(\G)\cap H_n \neq \emptyset$ for all $n$. Therefore
\begin{align}
\Pr_p(x\in\bdi C(o)) &= \Pr_p\left(\bigcap_{n\in\N}\{(\exists v\in V(\G)\cap H_n)(o\conn v)\}\right) =\\
&= \lim_{n\to\infty} \Pr_p((\exists v\in V(\G)\cap H_n)(o\conn v)) \ge a^2.
\end{align}
Hence, $p\notin\NP$, so $p\ge p_0$, as desired.
\par The main theorem (Theorem \ref{mainthm}) is interesting when $\pc<p_0$. As for now, I do not know, what is the class of embedded graphs $\G$ (even among those arising from Coxeter reflection groups as in \ref{PiG}) satisfying $\pc(\G)<p_0(\G)$. I suspect that $p_0=\pu$ for graphs as in \ref{PiG} in the cocompact case (see Remark \ref{meas0->nullae}; in such case most often we would have $p_0>\pc$). On the other hand, there are examples where $p_0<\pu$ (see Example \ref{p0<pu} below). Still, I do not know if it is possible that $\pc=p_0<\pu$.
\end{rem}
\begin{exmp}\label{p0<pu}
Let $\Pi$ be an unbounded polyhedron with $6$ faces in $\HHH$ whose five faces are cyclically perpendicular and the sixth one is disjoint from them. Then, in the setting of Assumption \ref{PiG}, the corresponding Coxeter group $G$ is isomorphic to the free product of $\mathbb Z_2$ and the right-angled Coxeter group $G_5$ with the nerve being a simple $5$-cycle. Let $\CG$ and $\CG_5$ be the Cayley graphs of $G$ and $G_5$, respectively. Then, $\CG$ has infinitely many ends, so from \cite[Exercise 7.12(b)]{LP} $\pu(\CG)=1$. Next, if $p>\pu(\CG_5)$, then with positive \pbb\ $\bdi C(o)$ contains the whole circle $\bdi(G_5\cdot o)$. (It is implied by Thm.~4.1 and Lem.~4.3 from \cite{BS96}.) Hence, $p_0\le\pu(\CG_5)<\pu(\CG)$, as $\pu(\CG_5)<1$ by \cite[Thm.~10]{BB}. Moreover, the conclusion of the main theorem (Theorem \ref{mainthm}) fails for any $p>\pu(\CG_5)$.
\end{exmp}
\begin{rem}\label{meas0->nullae}
This remark is hoped to explain a little my suspicion (stated in Remark \ref{ineqp0}) that for the Cayley graph of a cocompact Coxeter reflection group in $\Hd$, we have $p_0=\pu$. Namely, for $p<\pu$, I suspect that a property of the $p$-\Bbp\ in that setting quite similar to $\Pr_p(x\in\bdi C(o))=0$ (considered in \eqref{assmnull}) is exhibited:
\begin{equation}
\Pr_p\textrm{-a.s.~}|\bdi C(o)|=0.
\end{equation}
Here $|\cdot|$ can be the Lebesgue measure on $\bdi\Hd=\S^{d-1}$, or the Poisson measure on $\bdi\Hd$ arising from the simple random walk on $\G$ starting at $o$. (For a definition of a simple random walk and an explanation of Poisson boundary, see \cite{Woe}, Section 1.C and Section 24., p.~260, respectively.) If one proves it, then the probability vanishing in \eqref{assmnull} follows for $|\cdot|$-a.e. point $x\in\bdi\Hd$ by an easy exercise (using a simple version of Fubini's theorem for for the product measure $\Pr_p\times|\cdot|$). In addition, because the induced action of such cocompact group on $\bdi\Hd$ has only dense orbits (see e.g.~\cite[Proposition 4.2]{KapBen}), I rather suspect that in such situation as above, $\Pr_p(x\in\bdi C(o))=0$ holds for all $x\in\bdi\Hd$.
\end{rem}

\newcommand{\maled}{d}
\section{Definitions: percolation on a fragment of $\mathbb{H}^{\lowercase{d}}$}

Here I am going to introduce some notions and notations used in Theorem \ref{lemMen} and Lemma \ref{corMen} and in the proof of the main theorem.
\begin{denot1}
First of all I remark that in this \work\ $0$ is a natural number. I denote the set of all positive natural numbers by $\N_+$.
\end{denot1}
\begin{df}\label{dfhHd}
For the rest of this \partw, consider $\Hd$ in its fixed half-space Poincar\'e model (being the upper half-space $\R^{d-1}\times(0;\infty)$) in which the point $o$ (the distinguished vertex of $\G$) is represented by $(0,\ldots,0,1)$. (It will play the role of origin of both $\Hd$ and $\G$.)
\par The half-space model of $\Hd$ and its relation to the Poincar\'e ball model are explained in \cite[Chapter I.6, p.~90]{BH}. Note that the inversion of $\R^d$ mapping the Poincar\'e ball model $\mathbb{B}^d$ to our fixed half-space model sends one point of the sphere $\bdt\mathbb{B}^d$ to infinity. In the context of the half-space model, I treat that ``infinity'' as an abstract point (outside $\R^d$) compactifying $\R^d$. I call it the \emd{point at infinity} and denote it by $\infty$.
\par Let $\hcHd$ be the closure of $\Hd$ in $\R^d$ and $\hbd\Hd=\hcHd\sm\Hd$ 
(so here $\hcHd=\R^{d-1}\times[0;\infty)$ and $\hbd\Hd=\R^{d-1}\times\{0\}$). I identify $\hcHd$ with $\cHd\sm\{\infty\}$ and $\hbd\Hd$ with $\bdi\Hd\sm\{\infty\}$ in a natural way. Also, for any closed $A\subseteq\Hd$, let $\hc{A}=\cl{A}^\hcHd$ and $\hbd A= \hc{A}\sm A$. (Here, for complex notation for a subset of $\Hd$ (of the form e.g.\ $A_x^y(z)$), I use the same notational convention for $\,\hc\cdot\,$ as for $\,\c\cdot\,$---see Remark \ref{smarthat}.)
\par Although sometimes I use the linear and Euclidean structure of $\R^d$ in $\Hd$, the default geometry on $\Hd$ is the hyperbolic one, unless indicated otherwise. On the other hand, by the \emd{Euclidean metric of the disc model} I mean the metric on $\cHd$ induced by the embedding of $\cHd$ in $\R^d$ (as a unit disc) arising from the Poincar\'e disc model of $\Hd$. Nevertheless, I am going to treat that metric as a metric on the set $\hcHd\cup\{\infty\}=\cHd$, never really considering $\Hd$ in the disc model.
\end{df}

\begin{df}
For $k>0$ and $x\in\R^{d-1}\times\{0\}$, by $y\mapsto k\cdot y$ and $y\mapsto y+x$ (or $k\cdot$, $\cdot+x$, respectively, for short) I mean always just a scaling and a translation of $\R^d$, respectively, often as isometries of $\Hd$. (Note that restricted to $\Hd$ they are indeed hyperbolic isometries.)
\end{df}

\begin{denot}
Let $\Isom(\Hd)$ denote the isometry group of $\Hd$.
\par For any $h\in(0;1]$ and $R\in O(d)$ (the orthogonal linear group of $\R^d$) the pair $(h,R)$ determines uniquely an isometry of $\Hd$ denoted by $\Phi\hR$ such that~$\Phi\hR(o)=(0,\ldots,0,h)$ and $D\Phi\hR(o)=hR$ (as an ordinary derivative of a function $\R^{d-1}\times(0,\infty)\to\R^d$).
\par Let $\G\hR$ denote $\Phi\hR[\G]$. Similarly, for any $\Phi\in\Isom(\Hd)$ let $\G^\Phi=\Phi[\G]$. Further, in the same fashion, let $o\hR=\Phi\hR(o)$ (which is $h\cdot o$) and $o^\Phi=\Phi(o)$.
\end{denot}
\begin{df}
For any $p\in[0;1]$, whenever I consider $p$-\Bbp{} on $\G^\Phi$ for $\Phi\in\Isom(\Hd)$, I just take $\Phi[\omega]$, where $\omega$ denotes the random configuration in $p$-\Bbp\ on $G$.
\end{df}
\begin{rem}\label{remcoupl}
One can say that this is a way of coupling of the \Bbp\ processes on $\G^\Phi$ for $\Phi\in\Isom(\Hd)$.
\par Formally, the notion of ``$p$-\Bbp\ on $\G^\Phi$'' is not well-defined because for different isometries $\Phi_1$, $\Phi_2$ of $\Hd$ such that $\G^{\Phi_1}=\G^{\Phi_2}$, still the processes $\Phi_1[\omega]$ and $\Phi_2[\omega]$ are different. Thus, I am going to use the convention that the isometry $\Phi$ used to determine the process $\Phi[\omega]$ is the same as used in the notation $\G^\Phi$ determining the underlying graph.
\end{rem}
\begin{denot}
Let $L^h=\R^{d-1}\times(0;h]\subseteq\Hd$ and denote $L=L^1$. (In other words, $L^h$ is the complement of some open horoball in $\Hd$, which viewed in the Poincar\'e disc model $\mathbb{B}^d$ is tangent to $\bdi\mathbb{B}^d$ at the point corresponding to $\infty$.)
\end{denot}

\begin{df}
Consider any closed set $A\subseteq\Hd$ intersecting each geodesic line only in finitely many intervals and half-lines of that line (every set from the algebra of sets generated by convex sets satisfies this condition, e.g.~$A=L^h$). Then, by $\G^\Phi\cap A$ I mean an embedded graph in $A$ with the set of vertices consisting of $V(\G^\Phi)\cap A$ and the points of intersection of the edges of $\G^\Phi$ with $\bdt A$\uwaga{ale może zmienić tę def., aby nie zajmować się tymi dodatk. wierzch.? $\leadsto$ Sprawdzić, czy w wersji z odrzuceniem krawędzi $\nsubseteq A$ wszystko działa i ew. poprawić wszystko pod tym kątem.} and with the edges being all the non-degenerate components of intersections of edges of $\G^\Phi$ with $A$. The percolation process on $\G^\Phi\cap A$ I consider in this \partw\ is, by default, the process $\Phi[\omega]\cap A$. The same convention as in Remark \ref{remcoupl} is used for these processes.
\end{df}

\begin{rem}
To prove the main theorem, I use the process $\Phi\hR[\omega]\cap L^H$ for $p\in[0;1]$ and for different $H$. In some sense, it is $p$-\Bbp{} on $\G\hR\cap L^H$: on one hand, this process is defined in terms of the independent random states of the edges of $\G\hR$, but on the other hand, some different edges of the graph $\G\hR\cap L^H$ are obtained from the same edge of $\G\hR$, so their states are stochastically dependent. Nevertheless, I am going to use some facts about \Bp\ for the percolation process on $\G\hR\cap L^H$. In such situation, I consider the edges of $\G\hR$ intersecting $L^H$ instead of their fragments obtained in the intersection with $L^H$.
\end{rem}

\section{Exponential decay of the cluster size distribution}\label{secexpdec}

I am going to treat the percolation process $\Phi\hR[\omega]\cap L^H$ roughly as a Bernoulli percolation process on the standard lattice $\mathbb{Z}^{d-1}$ (given graph structure by joining every pair of vertices from $\mathbb{Z}^{d-1}$ with distance $1$ by an egde). It is motivated by the fact that $\mathbb{Z}^{d-1}$ with the graph metric is quasi-isometric to $\hbd\Hd$ or $L^H$ with the Euclidean metric. (Two metric spaces are \emd{quasi-isometric} if, loosely speakig, there are mappings in both directions between them which are bi-Lipschitz up to an additive constant. For strict definition, see \cite[Definition I.8.14]{BH}; cf.\ also Exercise 8.16(1) there.)

In the setting of $\mathbb{Z}^{d-1}$, we have a theorem on exponential decay of the cluster size distribution, below the critical threshold of percolation:{\uwaga{jakie $d$? Chyba $d\ge2$?}}
\begin{thm}[\protect{\cite[Theorem (5.4)]{Grim}}]\label{thm5.4}
For any $p<\pc(\mathbb{Z}^d)$ there exists $\psi(p)>0$ such that in $p$-\Bbp{} on $\mathbb{Z}^d$
$$\Pr_p(\textrm{the origin $(0,\ldots,0)$ is connected to the sphere of radius }n)<e^{-\psi(p)n}\quad\textrm{for all }n,$$
where the spheres are considered in the graph metric on $\mathbb{Z}^d$.
\\\qed
\end{thm}
The idea (of a bit more general theorem) comes from \cite{Men}, where a sketch of proof is given, and a detailed proof of the above statement is present in \cite{Grim}.

I adapt the idea of this theorem to the percolation process on $\G\hR\cap L$ in Theorem \ref{lemMen} and Lemma \ref{corMen}, appropriately rewriting the proof in \cite{Grim}, which is going to be the key part of the proof of main theorem.
In order to consider such counterpart of the above theorem, I define a kind of tail of all the distributions of the cluster size in $\G\hR\cap L$ for $(h,R)\in\IOd$ as follows:

\begin{df}\label{gpr}
Let $\pbd$ be the Euclidean orthogonal projection from $\Hd$ onto $\hbd\Hd$ and for any $x,y\in\Hd$,
$$d_\hbd(x,y)=\|\pbd(x)-\pbd(y)\|_\infty,$$
where $\|\cdot\|_\infty$ is the maximum (i.e.\ $l^\infty$) norm on $\hbd\Hd=\R^{d-1}\times\{0\}$.
Then, for $r>0$ and $x\in\Hd$, let
$$B_r(x)=\{y\in\Hd:d_\hbd(x,y)\le r\}\quad\textrm{and}\quad S_r(x)=\bdt B_r(x)$$
and for $h>0$, put
$$B_r^h(x)=B_r(x)\cap L^h.$$
If $x=o$ (or, more generally, if $\pbd(x)=\pbd(o)$), then I omit ``$(x)$''.
At last, for $p\in[0;1]$ and $r>0$, let
$$g_p(r)=\sup_{(h,R)\in\IOd}\Pr_p( \ho\conn S_r\textrm{ in }\G\hR\cap L).$$
\end{df}
\begin{rem}
In the Euclidean geometry, $B_r(x)$ and $B_r^h(x)$ are just cuboids of dimensions $r\times\ldots\times r\times\infty$ (unbounded in the direction of $d$-th axis) and $r\times\ldots\times r\times h$, respectively (up to removal of the face lying in $\hbd\Hd$).
\end{rem}
The condition ``$p<p_c(\mathbb{Z}^d)$'' in Theorem \ref{thm5.4} is going to be replaced by ``$p<p_0$'', which is natural because of the remark below. Before making it, I introduce notation concerning the percolation clusters:
\begin{denot}\label{dfMhRL}
For $\Phi\in\Isom(\Hd)$ and $v\in V(\G^\Phi)$ and a set $A\subseteq\Hd$ from the algebra generated by the convex sets, let $C^\Phi(v)$ and $C_A^\Phi(v)$ be the clusters of $v$ in $\G^\Phi$ and $\G^\Phi\cap A$, respectively, in the percolation configuration. Similarly, for $(h,R)\in\IOd$ and $\Phi=\Phi\hR$, I use notations $C\hR(v)$ and $C\hR_A(v)$, respectively.
\par If $v=\Phi(o)$, I omit ``$(v)$'' for short.
\end{denot}
\begin{rem}\label{rembdd}
If $p\in\NP$, then for any $\Phi\in\Isom(\Hd)$, the cluster $C^\Phi$ is $\Pr_p$-a.s.~bounded in the Euclidean metric. The reason is as follows. Take any $p\in\NP$ and $\Phi\in\Isom(\Hd)$. Then, for any $x\in\bdi\Hd$, we have $x\notin \c C(o)$ $\Pr_p$-a.s.~as well as $x\notin \c C^\Phi$ $\Pr_p$-a.s.
If we choose $x=\infty$ (for our half-space model of $\Hd$), then $\c C^\Phi$ is $\Pr_p$-a.s.~a compact set in $\hcHd$, so $C^\Phi$ is bounded in the Euclidean metric.
\end{rem}

Now, I formulate the theorem which is the counterpart of Theorem \ref{thm5.4}. Its proof (based on that of \cite[Theorem (5.4)]{Grim}) is delayed to Section \ref{prflemMen}.
\begin{thm}[exponential decay of $g_p(\cdot)$]\label{lemMen}\label{LEMMEN}
Let a graph $G$ embedded in $\Hd$ be connected, locally finite, transitive under isometries and let its edges be geodesic segments (as in Assumption \ref{assmG}). Then, for any $p<p_0$, there exists $\psi=\psi(p)>0$ such that~for any $r>0$,
$$g_p(r)\le e^{-\psi r}.$$
\end{thm}

The next lemma is a stronger version of the above one, where we take the union of all the clusters meeting some $\Bo$ instead of the cluster of $\ho$ in $\G\hR\cap L$. In other words, here the role of $\ho$ played in Theorem \ref{lemMen} is taken over by its thickened version $\Bo\cap V(\G^\Phi)$ for any $\Phi\in\Isom(\Hd)$. That leads to the following notation:
\begin{denot1}
I denote
$$\bo_\Phi=\Bo\cap V(\G^\Phi)$$
for $\Phi\in\Isom(\Hd)$.
\end{denot1}

\begin{df}
For any $C\subseteq\Hd$, I define its \emd{size} by
$$r(C)=\sup_{x\in C}d_\hbd(o,x).$$
\end{df}

\begin{lem}\label{corMen}
Let a graph $G$ embedded in $\Hd$ be connected, locally finite, transitive under isometries and let its edges be geodesic segments (as in Assumption \ref{assmG}). Then, for any $p$ such that~the conclusion of Theorem \ref{lemMen} holds (in particular, for $p<p_0$) and for any $r_0>0$, there exist $\al=\al(p,r_0),\phi=\phi(p,r_0)>0$ such that~for any $r\ge 0$,
\begin{equation}\label{ineqcorMen}
\sup_{\Phi\in\Isom(\Hd)} \Pr_p(r(\bigcup_{v\in\bo_\Phi} C_L^\Phi(v))\ge r) \le \al e^{-\phi r}.
\end{equation}
\end{lem}
\begin{proof}
First, note that it is sufficient to prove the inequality
\begin{equation}
\label{ineqcorMen<->}
\sup_{\Phi\in\Isom(\Hd)} \Pr_p(\bo_\Phi \conn S_r\textrm{ in }\G^\Phi\cap L) \le \al e^{-\phi r}
\end{equation}
for $r$ greater than some fixed $r_1>0$ in place of \eqref{ineqcorMen}. Indeed, suppose there exist $\al,\phi>0$ such that \eqref{ineqcorMen<->} holds for all $r>r_1$. We then have:
\begin{itemize}
\item for any $r>r_1$ and $\e\in(0;\min(r-r_1,1))$,
\begin{align}
\sup_{\Phi\in\Isom(\Hd)} \Pr_p(r(\bigcup_{v\in\bo} C_L^\Phi(v))\ge r) &\le \sup_{\Phi\in\Isom(\Hd)} \Pr_p(\bo_\Phi \conn S_{r-\e}\textrm{ in }\G^\Phi\cap L) \le\\
&\le \al e^{-\phi(r-\e)} \le (\al e^{\phi})e^{-\phi r},
\end{align}
\item for $r\le r_1$, the left-hand side of \eqref{ineqcorMen} is less than or equal to $1 \le e^{\phi r_1}e^{-\phi r}$.
\end{itemize}
So then we will get the lemma for any $r\ge 0$ with $\max(e^{\phi r_1},\al e^{\phi})$ put in place of $\al$.

\par Now, I prove \eqref{ineqcorMen<->} (I pick $r_1$ as above later): let $r>r_0$ and $\Phi\in\Isom(\Hd)$. The task is to pick appropriate values of $\al$ and $\phi$ independently of $r$ and $\Phi$.
\begin{df}
Put $\bo=\bo_\Phi$. For $x\in\Hd\subseteq\R^d$, let $h(x)$ denote the $d$-th coordinate of $x$ (or: Euclidean distance from $x$ to $\hbd\Hd$), which I call \emd{height} of $x$.
\end{df}
Assume for a while that $\bo\conn S_r$ in $\G^\Phi\cap L$ (note that this event may have probability $0$, e.g.~when $\bo=\emptyset$).
Consider all open paths 
in $\G^\Phi\cap L$ joining $\bo$ to $S_r$ and consider all the vertices of $G\Phi$ visited by those paths, lying in $B_r^1$. There is a non-zero finite number of vertices of maximal height among them because $\G^\Phi$ is locally finite. Choose one of these vertices at random and call it $\vh$. This $\Hd$-valued random variable is defined whenever $\bo\conn S_r$ in $\G^\Phi\cap L$.)
\begin{obs}
There exists $H\ge 1$ such that a.s. if $\vh$ is defined, then $\vh \conn S_\frac{r-r_0}{2}(\vh)$ in $\G^\Phi\cap L^{Hh(\vh)}$.
\end{obs}
\begin{proof}
 Assume that $\vh$ is defined and take a path $P$ joining $\bo$ to $S_r$ passing through $\vh$. Hyperbolic lengths of edges in $\G^\Phi$ are bounded from above (by the transitivity of $\G^\Phi$ under isometries).
That implies that for any edge of $\G^\Phi$ the ratio between the heights of any two of its points is also bounded from above by some constant $H\ge 1$ (it is going to be the $H$ in the observation). (The reason for that are the two following basic properties of the half-space model of $\Hd$:
\begin{itemize}
\item The heights of points of any fixed hyperbolic ball (of finite radius) are bounded from above and from below by some positive constants.
\item Any hyperbolic ball can be mapped onto any other hyperbolic ball of the same radius by a translation by vector from $\R^{d-1}\times\{0\}$ composed with a linear scaling of $\R^d$.
\end{itemize}
That implies that the path $P\subseteq L^{Hh(\vh)}$.
\par Now, because $P$ contains some points $x\in\Bo$ and $y\in S_r$ and, by triangle inequality, $d_\hbd(x,y)\ge r-r_0$, it follows that $d_\hbd(\vh,x)$ or $d_\hbd(\vh,y)$ is at least $\frac{r-r_0}{2}$ (again by triangle inequality). Hence, $P$ intersects $S_\frac{r-r_0}{2}(\vh)$, which finishes the proof.
\end{proof}
Based on that observation, I estimate:
\begin{align}
\Pr_p(\bo\conn S_r\textrm{ in }\G^\Phi\cap L) \label{Pbo<->Sr}\tag{$\ast$}
&= \quad\voidindop{\sum}{v\in V(\G^\Phi)\cap B_r^1}\quad \Pr_p(\vh\textrm{ is defined and }\vh=v) \le\\
&\le \quad\voidindop{\sum}{v\in V(\G^\Phi)\cap B_r^1}\quad \Pr_p(v\conn S_\frac{r-r_0}{2}(v)\textrm{ in }\G^\Phi\cap L^{Hh(v)}) =\\
&= \quad\voidindop{\sum}{v\in V(\G^\Phi)\cap B_r^1}\quad \Pr_p\Bigl(\textstyle\frac{1}{H}\cdot o\conn S_\frac{r-r_0}{2Hh(v)}\left(\frac{1}{H}\cdot o\right)\textrm{ in }
\frac{1}{Hh(v)}(\G^{\Phi} - \pbd(v))\cap L^1\Bigr),
\end{align}
by mapping the situation via the (hyperbolic) isometry $\frac{1}{Hh(v)}(\cdot - \pbd(v))$ for each $v$.
Note that because for $v\in V(\G^\Phi)\cap B_r^1$, $\frac{1}{H}\cdot o$ indeed is a vertex of $\frac{1}{Hh(v)}(\G^{\Phi} - \pbd(v))$, by the transitivity of $\G$ under isometries, we can replace the isometry $\frac{1}{Hh(v)}\(\cdot-\pbd(v)\)$ with an isometry giving the same image of $\G$ and mapping $o$ to $\frac{1}{H}\cdot o$, hence of the form $\Phi^{(1/H,R)}$. That, combined with the assumption on $p$ (the conclusion of Theorem \ref{lemMen}), gives
\begin{equation}\label{Pbo<->Sr<exp}
\eqref{Pbo<->Sr} \le \sum_{v\in V(\G^\Phi)\cap B_r^1} g_p\(\frac{r-r_0}{2Hh(v)}\)
\le \sum_{v\in V(\G^\Phi)\cap B_r^1} e^{-\psi\frac{r-r_0}{2Hh(v)}},
\end{equation}
where $\psi$ is as in Theorem \ref{lemMen}.
\par Because $B_r^1 = [-r;r]^{d-1}\times(0;1]$, one can cover it by $\left\lceil\frac{r}{r_0}\right\rceil^{d-1}$ translations of $\Bo$ by vectors from $\R^{d-1}\times\{0\}$. So, let $\{\Bo(x_i):i=1,\ldots,\left\lceil\frac{r}{r_0}\right\rceil^{d-1}\}$ be such covering. Moreover, each $\Bo(x_i)$ can be tesselated by infinitely many isometric (in the hyperbolic sense) copies of $K=\Bo\sm L^{\frac{1}{2}}$, more precisely, by: a translation of $K$, $2^{d-1}$ translations of $\frac{1}{2}K$,  $(2^{d-1})^2$ translations of $\frac{1}{2^2}K$, etc., all along $\R^{d-1}\times\{0\}$. Let $U=\sup_{\phi\in\Isom(\Hd)}\#(V(\G^\Phi)\cap\phi[K])$ ($U<\infty$ by local finiteness of $\G$). Then, splitting the sum from \eqref{Pbo<->Sr<exp} according to those tesselations,
\begin{align}
\eqref{Pbo<->Sr} &\le \sum_{i=1}^{\left\lceil\frac{r}{r_0}\right\rceil^{d-1}} \sum_{v\in V(\G^\Phi)\cap B_{r_0}^1(x_i)} e^{-\psi\frac{r-r_0}{2Hh(v)}} \le\\
&\le \left\lceil\frac{r}{r_0}\right\rceil^{d-1} \sum_{k=0}^\infty (2^{d-1})^k U\sup_{h\in[\frac{1}{2^{k+1}};\frac{1}{2^k}]} e^{-\psi\frac{r-r_0}{2Hh}} \le\\
&\le U\left\lceil\frac{r}{r_0}\right\rceil^{d-1} \sum_{k=0}^\infty (2^{d-1})^k e^{-\frac{\psi}{H}2^{k-1}(r-r_0)} =\label{smallinfsum}\\
&= U\left\lceil\frac{r}{r_0}\right\rceil^{d-1} \sum_{k=0}^\infty e^{\ln 2\cdot k(d-1) - \frac{\psi}{H}2^{k-1}(r-r_0)}.
\end{align}
Now, I am going to show that the above bound is finite and tends to $0$ at exponential rate with $r\to\infty$. First, I claim that there exists $k_0\in\N$ such that
\begin{equation}\label{k_0cond}
(\forall k\ge k_0)(\forall r\ge 2r_0) \left(\ln2\cdot k(d-1) - 2^{k-1} \frac{\psi}{H}(r-r_0) \le -kr\right).
\end{equation}
Indeed, for sufficiently large $k$ we have $2^{k-1}\frac{\psi}{H}-k>0$, so for $r\ge 2r_0$
\begin{equation}
\left(2^{k-1}\frac{\psi}{H}-k\right)r \ge \left(2^{k-1}\frac{\psi}{H}-k\right)\cdot2r_0
\end{equation}
and
\begin{equation}
2^{k-1}\frac{\psi}{H}(r-r_0) - kr \ge 2^{k-1}\frac{\psi}{H}r_0 - 2kr_0 \ge k(d-1)\ln2
\end{equation}
for sufficiently large $k$. So, let $k_0$ satisfy \eqref{k_0cond}. Then, for $r\ge 2r_0$,
\begin{align}
\eqref{Pbo<->Sr} &\le U\left\lceil\frac{r}{r_0}\right\rceil^{d-1} \(\sum_{k=0}^{k_0-1} (2^{d-1})^k e^{-2^{k-1} \frac{\psi}{H}(r-r_0)} + \sum_{k=k_0}^\infty e^{-kr}\) \le\\
&\le U\left\lceil\frac{r}{r_0}\right\rceil^{d-1} \biggl(k_0(2^{d-1})^{k_0-1} e^{-\frac{\psi}{2H}(r-r_0)} + e^{-k_0r}\underbrace{\frac{1}{1-e^{-r}}}_{\le \frac{1}{1-e^{-2r_0}}}\biggr) \le\\
&\le U\left\lceil\frac{r}{r_0}\right\rceil^{d-1} (De^{-Er})
\end{align}
for some constants $D,E>0$. If we choose $r_1\ge 2r_0$ such that
$$(\forall r\ge r_1)\left( \left\lceil\frac{r}{r_0}\right\rceil^{d-1} \le e^\frac{Er}{2}\right)$$
(which is possible), then
$$\eqref{Pbo<->Sr}\le UDe^\frac{-Er}{2}\quad\textrm{for }r\ge r_1,$$
which finishes the proof of the lemma.
\end{proof}

\section{Scaling---proof of the main theorem}\label{secprfmainthm}

Now I complete the proof of the main theorem:
\begin{thm*}[recalled Theorem \ref{mainthm}]
Let $G$ satisfy the Assumption \ref{assmG}. Then, for any $0\le p<p_0$, a.s. every cluster in $p$-\Bbp\ on $\G$ has only one-point boundaries of ends.
\end{thm*}
\begin{proof}[Proof of Theorem \ref{mainthm}]
Fix $p\in[0;p_0)$ and suppose towards a contradiction that with some positive probability there is some cluster with some end with a non-one-point boundary. Note that by Remark \ref{rembdd} and by the transitivity of $\G$ under isometries, for any $v\in V(\G)$ a.s.~$C(v)$ is bounded in the Euclidean metric, so, a.s.~all the percolation clusters in $\G$ are bounded in the Euclidean metric. Then, for some $\d>0$ and $r>0$, there exists with \pbb\ $a>0$ a cluster bounded in the Euclidean metric, with boundary of some end having Euclidean diameter greater than or equal to $\d$ and intersecting the open disc $\sint_{\hbd\Hd}\hbd B_r$.
Let $C$ and $e$ be such cluster and its end, respectively. Let for $A\subseteq\Hd$, the \emd{projection diameter} of $A$ be the Euclidean diameter of $\pbd(A)$. Then for $h>0$
\begin{itemize}
\item the set $\cl{C\sm L^h}$ is compact;
\item $e(\cl{C\sm L^h})$ is a cluster in the percolation configuration on $\G\cap L^h)$;
\item $e(\cl{C\sm L^h})$ has projection diameter at least $\d$ and intersects $B_r\cap V(\G)$.
\end{itemize}
All the above implies that for any $k\in\N$,
$$\Pr_p(\exists\textrm{ a cluster in $\G\cap L^\frac{1}{2^k}$ of projection diameter $\ge\d$ intersecting $B_r\cap V(\G)$})\ge a,$$
so, by scaling by $2^k$ in $\R^d$ (which is a hyperbolic isometry), we obtain
$$\Pr_p(\exists\textrm{ a cluster in $\G^{2^k\cdot}\cap L$ of projection diameter $\ge2^k\d$ intersecting $B_{2^kr}\cap V(\G^{2^k\cdot})$})\ge a$$
(where $\G^{2^k\cdot}$ is the image of $\G$ under the scaling). $B_{2^kr}\cap L$ is a sum of $(2^k)^{d-1}$ isometric copies of $B_r\cap L$, so the left-hand side of above inequality is bounded from above by
\begin{multline}
(2^k)^{d-1}\sup_{\Phi\in\Isom(\Hd)} \Pr_p(\exists\textrm{ a cluster in $\G^\Phi\cap L$ of proj.\ diam.\ $\ge2^k\d$ intersecting $B_{r}\cap V(\G^\Phi)$}) \le\\
\le (2^k)^{d-1}\sup_{\Phi\in\Isom(\Hd)} \Pr_p\left(r\left(\bigcup_{v\in B_r^1\cap V(\G^\Phi)} C_L^\Phi(v)\right)\ge \frac{2^k\d}{2}\right)
\end{multline}
(because the size of a cluster is at least half its projection diameter),
so by Lemma \ref{corMen}, for any $k\in\N$,
$$a\le (2^k)^{d-1} \al e^{-\phi\d 2^{k-1}},$$
where $\al,\phi>0$ are constants (as well as $\d$, $a$ and $r$). But the right-hand side of this inequality tends to $0$ with $k\to \infty$, a contradiction.
\end{proof}

\section{Proof of the exponential decay}\label{prflemMen}
In this section, I prove Theorem \ref{lemMen}:
\begin{thm*}[recalled Theorem \ref{lemMen}]
Let a graph $G$ embedded in $\Hd$ be connected, locally finite, transitive under isometries and let its edges be geodesic segments (as in Assumption \ref{assmG}). Then, for any $p<p_0$, there exists $\psi=\psi(p)>0$ such that~for any $r>0$,
$$g_p(r)\le e^{-\psi r}.$$
\end{thm*}
\begin{proof}
As mentioned earlier, this proof is an adaptation of the proof of Theorem (5.4) in \cite{Grim} based on the work \cite{Men}. Its structure and most of its notation are also borrowed from \cite{Grim}, so it is quite easy to compare both the proofs. (The differences are technical and they are summarised in Remark \ref{diffprfs}.) The longest part of this proof is devoted to show functional inequality \eqref{fctineq(gt)1} and it goes roughly linearly. Then follows Lemma \ref{gt<=sqrt}, whose proof, using that functional inequality, is deferred to the end of this section. Roughly speaking, that lemma provides a mild asymptotic estimate for $g_p$ (more precisely: for $\gt_p$ defined below), which is then sharpened to that desired in Theorem \ref{lemMen}, using repeatedly inequality \eqref{fctineq(gt)1}.
\par At some point, I would like to use random variables with left-continuous distribution function\footnote{By the \emd{left-continuous distribution function} of a probability distribution (measure) $\mu$ on $\R$, I mean the function $\R\ni x\mapsto\mu((-\infty,x))$.} $1-g_p$. Because $1-g_p$ does not need to be left-continuous, I replace $g_p$, when needed, by its left-continuous version $\gt_p$ defined as follows:
\begin{df}
Put $\gt_p(r)=\lim_{\rho\to r^-} g_p(\rho)$ for $r>0$.
\end{df}
As one of the cornerstones of this proof, I am going to prove the following functional inequality for $\gt_\cdot(\cdot)$: for any $\al,\b$ s.t.~$0\le\al<\b\le 1$ and for $r>0$,
\begin{equation}\label{fctineq(gt)1}
\gt_\al(r) \le \gt_\b(r)\exp\(-(\b-\al)\(\frac{r}{a+\int_0^r \gt_\b(m)\,\ud m} - 1\)\),
\end{equation}
where $a$ is a positive constant depending only on $G$. Note that it implies Theorem \ref{lemMen} provided that the integral in the denominator is a bounded function of $r$.
\par I am going to approach this inequality, considering the following events depending only on a finite fragment of the percolation configuration and proving functional inequality \eqref{fctineq(f)} (see below). Cf.\ Remark \ref{whyLd}.
\begin{df}
Fix arbitrary $(h,R)\in\IOd$. I am going to use events $A^\d(r)$ defined as follows: let $p\in[0,1]$, $r>0$ and $\d\in(0;h]$, and define $L_\d=\R^{d-1}\times[\d;1]\subseteq\Hd$ (not to be confused with $L^\d$). Let the event
$$A^\d(r)=\{\ho\conn S_r\textrm{ in }\G\hR\cap L_\d\}$$
and let
$$f_p^\d(r)=\Pr_p(A^\d(r)).$$
\end{df}
Now, I am going to show that the functions defined above satisfy a functional inequality:
\begin{equation}\label{fctineq(f)}
f_\al^\d(r)\le f_\b^\d(r)\exp\(-(\b-\al)\(\frac{r}{a+\int_0^r \gt_\b(m)\,\ud m} - 1\)\)
\end{equation}
for any $0\le\al<\b\le 1$, $r>0$\uwaga? and for $\d\in(0;h)$.
Having obtained this, I will pass to some limits and to supremum over $(h,R)$, obtaining the inequality \eqref{fctineq(gt)1}.
\par Note that, if there is no path joining $\ho$ to $S_r$ in $\G\hR\cap L_\d$ at all, then for any $p\in[0;1]$, $f_p^\d(r)=0$ and the inequality \eqref{fctineq(f)} is obvious. The same happens when $\al=0$. Because in the proof of that inequality I need $f_p^\d(r)>0$ and $\al>0$, now I make the following assumption (without loss of generality):
\begin{assm}\label{assmapath}
I assume that there is a path joining $\ho$ to $S_r$ in $\G\hR\cap L_\d$ and that $\al>0$. (Then for $p>0$, $f_p^\d(r)>0)$.)
\end{assm}
The first step in proving inequality \eqref{fctineq(f)} in using Russo's formula for the events $A^\d(r)$. Before I formulate it, I provide a couple of definition needed there.
\begin{df}
Next, for a random event $A$ in the percolation on $\G\hR$, call an edge \emd{pivotal} for a given configuration iff changing the state of that edge (and preserving other edges' states) causes $A$ to change its state as well (from occurring to not occurring or \emph{vice-versa}). Then, let $N(A)$ be the (random) number of all edges pivotal for $A$.
\end{df}
\begin{df}\label{incr}
We say that an event $A$ (being a set of configurations) is \emd{increasing} iff for any configurations $\omega\subseteq\omega'$, if $\omega\in A$, then $\omega'\in A$.
\end{df}
\begin{thm}[Russo's formula]
Consider \Bbp\ on any graph $G$ and let $A$ be an increasing event defined in terms of the states of only finitely many edges of $G$. Then
$$\frac{\ud}{\ud p}\Pr_p(A)=\Est_p(N(A)).$$
\qed
\end{thm}
This formula is proved as Theorem (2.25) in \cite{Grim} for $G$ being the classical lattice $\mathbb{Z}^d$, but the proof applies for any graph $G$.
\par Let $p>0$. The events $A^\d(r)$ depend on the states of only finitely many edges of $\G\hR$ (namely, those intersecting $L_\d\cap B_r$), so I am able to use Russo's formula for them, obtaining
$$\frac{\ud}{\ud p} f_p^\d(r) = \Est_p(N(A^\d(r))).$$
Now, for $e\in E(\G\hR)$, the event $\{e\textrm{ is pivotal for }A^\d(r)\}$ is independent of the state of $e$ (which is easily seen; it is the rule for any event), so
\begin{align}
\Pr_p(A^\d(r)\land e\textrm{ is pivotal for }A^\d(r)) &= \Pr_p(e\textrm{ is open and pivotal for }A^\d(r)) =\\
&=\textrm{(because $A^\d(r)$ is increasing)}\\
&= p\Pr_p(e\textrm{ is pivotal for }A^\d(r)),
\end{align}
hence
\begin{align}
\frac{\ud}{\ud p} f_p^\d(r) &= \sum_{e\in E(\G\hR)} \Pr_p(e\textrm{ is pivotal for }A^\d(r)) =\\
&= \frac 1p \sum_{e\in E(\G\hR)} \Pr_p(A^\d(r)\land e\textrm{ is pivotal for }A^\d(r)) =\\
&=\frac{f_p^\d(r)}{p} \sum_{e\in E(\G\hR)} \Pr_p(e\textrm{ is pivotal for }A^\d(r)|A^\d(r)) =\\
&= \frac{f_p^\d(r)}{p} \Est_p(N(A^\d(r))|A^\d(r)),
\end{align}
which can be written as
$$\frac{\ud}{\ud p} \ln(f_p^\d(r)) = \frac{1}{p} \Est_p(N(A^\d(r))|A^\d(r)).$$
For any $0<\al<\b\le 1$, integrating over $[\al,\b]$ and exponentiating the above equality gives
$$\frac{f_\al^\d(r)}{f_\b^\d(r)} = \exp\(-\int_\al^\b \frac 1p \Est_p(N(A^\d(r))|A^\d(r))\,\ud p\),$$
which implies
\begin{align}
f_\al^\d(r) 
\le f_\b^\d(r) \exp\(-\int_\al^\b \Est_p(N(A^\d(r))|A^\d(r))\,\ud p\).\label{corineqRusso}
\end{align}
\uwaga{Czy ten $\updownarrow$ tekst nie jest zbyt luźny/potoczysty?}\\
At this point, our aim is to bound $\Est_p(N(A^\d(r))|A^\d(r))$ from below.
\begin{df}
Let $\eta$ denote the percolation configuration in $\G\hR\cap L_\d$, i.e.
$$\eta = \Phi\hR[\omega]\cap L_\d.$$
\end{df}
Fix any $r>0$ and $\d\in(0;h]$ and assume for a while that $A^\d(r)$ occurs. Let us make a picture of the cluster of $\ho$ in $\G\hR\cap L_\d$ in the context of the pivotal edges for $A^\d(r)$ (the same picture as in \cite{Grim} and \cite{Men}). If $e\in E(\G\hR)$ is pivotal for $A^\d(r)$, then if we change the percolation configuration by closing $e$, we cause the cluster of $C\hR_{L_\d}(\ho)$ to be disjoint from $S_r$. So, in our situation, all the pivotal edges lie on any open path in $\G\hR\cap L_\d$ joining $\ho$ to $S_r$ and they are visited by the path in the same order and direction (regardless of the choice of the path).
\begin{df}\label{dfe_i}
Let $N=N(A^\d(r))$, and let $e_1,\ldots,e_N$ be this ordering, and denote by $x_i,y_i$ the endvertices of $e_i$, $x_i$ being the one closer to $\ho$ along a path as above. Also, let $y_0=\ho$.
\end{df}
Note that because for $i=1,\ldots,N$, there is no edge separating $y_{i-1}$ from $x_i$ in the open cluster in $\G\hR\cap L_\d$, by Menger's theorem (see e.g.~\cite[Thm.~3.3.1, Cor.~3.3.5(ii)]{Diest}), there exist two edge-disjoint open paths in that cluster joining $y_{i-1}$ to $x_i$. (One can say, following the discoverer of this proof idea, that that open cluster resembles a chain of sausages.)
\begin{df}
Now, for $i=1,\ldots,N$, let $\rho_i=d_\hbd(y_{i-1},x_i)$ (this way of defining $\rho_i$, of which one can think as ``projection length'' of the $i$-th ``sausage'', is an adaptation of that from \cite{Grim}).
\end{df}
\par Now I drop the assumption that $A^\d(r)$ occurs. The next lemma is used to compare $(\rho_1,\ldots,\rho_N)$ to some renewal process with inter-renewal times of roughly the same distribution as the size of $C_L\hR(\ho)$.
\begin{df}
Let $a$ denote the maximal projection distance (in the sense of $d_\hbd$) between the endpoints of a single edge $\G\hR$ crossing $L$.
\end{df}
\uwaga{Chyba powyżej bardziej poukładać wszystko w definicje itd. Może ulepszyć kolejność powyższych fragm.}
\begin{lem}[cf.~\protect{\cite[Lem.~(5.12)]{Grim}}]\label{saus}
Let $k\in\N_+$ and let $r_1,\ldots,r_k\ge 0$ be such that~$\sum_{i=1}^k r_i \le r - (k-1)a$. Then for $0<p<1$,
$$\Pr_p(\rho_k<r_k,\ \rho_i=r_i\textrm{ for }i<k|A^\d(r)) \ge (1-g_p(r_k))\Pr_p(\rho_i=r_i\textrm{ for }i<k|A^\d(r)).$$
\end{lem}
\begin{rem}
I use the convention that for $i\in\N$ such that~$i>N(A^\d(r))$ (i.e.~$e_i$, $\rho_i$ are undefined), $\rho_i=+\infty$ (being greater than any real number). On the other hand, whenever I mention $e_i$, $i\le N(A^\d(r))$.
\end{rem}
\begin{proof}[Proof of the lemma]
This proof mimics that of \cite[Lem.~(5.12)]{Grim}. Let $k\ge 2$ (I delay the case of $k=1$ to the end of the proof).
\begin{df}
Let for $e\in E(\G\hR\cap L^\d)$\uwaga{wydaje mi się, że tu chcę to traktować jak prawdziwy graf (z wierzchołkami powstałymi przez przekrój)}, $D_e$ be the connected component of $\ho$ in $\eta\sm\{e\}$. Let $B_e$ denote the event that the following conditions are satisfied:
\begin{itemize}
\item $e$ is open;
\item exactly one endvertex of $e$ lies in $D_e$---call it $x(e)$ and the other---$y(e)$;
\item $D_e$ is disjoint from $S_r$;
\item there are $k-1$ pivotal edges for the event $\{\ho\conn y(e)\textrm{ in }\eta\}$ (i.e. the edges each of which separates $\ho$ from $y(e)$ in $D_e\cup\{e\}$)---call them $e_1'=\{x_1',y_1'\},\ldots,e_{k-1}'=\{x_{k-1}',y_{k-1}'\}=e$, $x_i'$ being closer to $\ho$ than $y_i'$, in the order from $\ho$ to $y(e)$ (as in the Definition \ref{dfe_i});
\item $d_\hbd(y_{i-1}',x_i')=r_i$ for $i<k$, where $y_0'=\ho$.
\end{itemize}
Let $B=\bigcup_{e\in E(\G\hR\cap L^\d)} B_e$. When $B_e$ occurs, I say that $D_e\cup\{e\}$ with $y(e)$ marked, as a graph with distinguished vertex, is a \emd{witness for $B$}.
\end{df}
\par Note that it may happen that there are more than one such witnesses (which means that $B_e$ occurs for many different $e$). On the other hand, when $A^\d(r)$ occurs, then $B_e$ occurs for only one edge $e$, namely $e=e_{k-1}$ (in other words, $B\cap A^\d(r)=\bigcupdot_{e\in E(\G\hR\cap L^\d)} (B_e\cap A^\d(r))$), and there is only one witness for $B$. Hence,
$$\Pr_p(A^\d(r)\cap B) = \sum_{\Gamma} \Pr_p(\Gamma\textrm{ a witness for }B)\Pr_p(A^\d(r)|\Gamma\textrm{ a witness for }B),$$
where the sum is always over all $\Gamma$ being finite subgraphs of $\G\hR\cap L_\d$ with distinguished vertices such that~$\Pr_p(\Gamma\textrm{ a witness for }B)>0$.
\par For $\Gamma$ a graph with distinguished vertex, let $y(\Gamma)$ denote that vertex. Under the condition that $\Gamma$ is a witness for $B$, $A^\d(r)$ is equivalent the event that $y(\Gamma)$ is joined to $S_r$ by an open path in $\eta$ which is disjoint from $V(\Gamma)\sm \{y(\Gamma)\}$. I shortly write the latter event $\{y(\Gamma)\conn S_r\textrm{ in $\eta$ off }\Gamma\}$.
Now, the event $\{\Gamma\textrm{ a witness for }B\}$ depends only on the states of edges incident to vertices from $V(\Gamma)\sm\{y(\Gamma)\}$, so it is independent of the event $\{y(\Gamma)\conn S_r\textrm{ in $\eta$ off }\Gamma\}$. Hence,
\begin{equation}\label{P(AB)}
\Pr_p(A^\d(r)\cap B) = \sum_{\Gamma} \Pr_p(\Gamma\textrm{ a witness for }B)\Pr_p(y(\Gamma)\conn S_r\textrm{ in $\eta$ off }\Gamma).
\end{equation}
A similar reasoning, performed below, gives us the estimate of $\Pr_p(\{\rho_k\ge r_k\}\cap A^\d(r)\cap B)$. Here I use also the following fact: conditioned on the event $\{\Gamma\textrm{ a witness for }B\}$, the event $A^\d(r)\cap\{\rho_k\ge r_k\}$ is equivalent to each of the following:
\begin{align*}
&(A^\d(r)\land e_k\textrm{ does not exist}) \lor (A^\d(r)\land e_k\textrm{ exists }\land\rho_k\ge r_k) \iff\\
\iff &(\exists\textrm{ two edge-disjoint paths joining $y(\Gamma)$ to $S_r$ in $\eta$ off }\Gamma) \lor\\
\lor &(\exists\textrm{ two edge-disjoint paths in $\eta$ off $\Gamma$, joining $y(\Gamma)$ to $S_r$ and to $S_{r_k}(y(\Gamma))$, resp.}) \iff\\
\iff &(\exists\textrm{ two edge-disjoint paths in $\eta$ off $\Gamma$, joining $y(\Gamma)$ to $S_r$ and to $S_{r_k}(y(\Gamma))$, resp.}),
\end{align*}
because $S_{r_k}(y(\Gamma))\subseteq B_r$ from the assumption on $\sum_{i=1}^k r_i$.\uwaga{Tu uwaga: ,,obviously, $e_k$ (if exists) intersects $\sint B_r$'' -- czy potrzebna? Czy odnosi się ona do warunku na $\sum_{i=1}^k r_i$?}
So I estimate
\begin{multline}\label{P(rhoAB)}
\Pr_p(\{\rho_k\ge r_k\}\cap A^\d(r)\cap B) = \sum_{\Gamma} \Pr_p(\Gamma\textrm{ a witness for }B) \Pr_p(\{\rho_k\ge r_k\}\cap A^\d(r)|\Gamma\textrm{ a witness for }B) =\\
= \sum_{\Gamma} \Pr_p(\Gamma\textrm{ a witness for }B) \Pr_p((y(\Gamma)\conn S_r\textrm{ in $\eta$ off }\Gamma) \circ (y(\Gamma)\conn S_{r_k}(y(\Gamma))\textrm{ in $\eta$ off }\Gamma)),
\end{multline}
where the operation ``$\circ$'' is defined below:
\begin{df}
For increasing events $A$ and $B$ in a percolation on any graph $G$, the event $A\circ B$ means that ``$A$ and $B$ occur on disjoint sets of edges''. Formally,
$$A\circ B = \{\omega_A\cupdot\omega_B: \omega_A,\omega_B\subseteq E(G) \land \omega_A\in A\land\omega_B\in B\},$$
that is, $A\circ B$ is the set of configurations containing two disjoint set of open edges ($\omega_A,\omega_B$ above) which guarantee occurring of the events $A$ and $B$, respectively.
\end{df}
Now, I am going to use the following BK inequality (proved in \cite{Grim}):
\begin{thm}[BK inequality, \protect{\cite[Theorems (2.12) and (2.15)]{Grim}}]\label{BKineq}
For any graph $G$ and increasing events $A$ and $B$ depending on the states of only finitely many edges in $p$-\Bbp\ on $G$, we have
$$\Pr_p(A\circ B) \le \Pr_p(A)\Pr_p(B).$$
\qed
\end{thm}
I use this inequality for the last term (as the events involved are increasing (see def.~\ref{incr}) and defined in terms of only the edges from $E(\G\hR\cap(L_\d\cap B_r))$), obtaining
\begin{eqnarray*}
\Pr_p(\{\rho_k\ge r_k\}\cap A^\d(r)\cap B) &\le &\sum_{\Gamma} \Pr_p(\Gamma\textrm{ a witness for }B) \cdot \Pr_p(y(\Gamma)\conn S_r\textrm{ in $\eta$ off }\Gamma) \cdot\\
&& {}\cdot \Pr_p(y(\Gamma)\conn S_{r_k}(y(\Gamma))\textrm{ in $\eta$ off }\Gamma) \le\\
&\le &\(\sum_{\Gamma} \Pr_p(\Gamma\textrm{ a witness for }B) \Pr_p(y(\Gamma)\conn S_r\textrm{ in $\eta$ off }\Gamma)\)g_p(r_k) =\\
&= &\Pr_p(A^\d(r)\cap B) g_p(r_k)
\end{eqnarray*}
(by \eqref{P(AB)}). Dividing by $\Pr_p(A^\d(r))$ (which is positive by Assumption \ref{assmapath}) gives
\begin{align}\label{P(|A)}
\Pr_p(\{\rho_k\ge r_k\}\cap B|A^\d(r)) &\le \Pr_p(B|A^\d(r)) g_p(r_k) \quad|\quad\Pr_p(B|A^\d(r)) - \cdot\\
\Pr_p(\{\rho_k< r_k\}\cap B|A^\d(r)) &\ge \Pr_p(B|A^\d(r)) (1-g_p(r_k)).
\end{align}
Note that, conditioned on $A^\d(r)$, $B$ is equivalent to the event $\{\rho_i=r_i\textrm{ for }i<k\}$, so the above amounts to
\begin{equation}
\Pr_p(\rho_k<r_k, \rho_i=r_i\textrm{ for }i<k |A^\d(r)) \ge \Pr_p(\rho_i=r_i\textrm{ for }i<k|A^\d(r)) (1-g_p(r_k)),
\end{equation}
which is the desired conclusion.
\par Now, consider the case of $k=1$. In this case, similarly to \eqref{P(rhoAB)} and thanks to the assumption $r_1\le r$,
\begin{align}
\Pr_p(\{\rho_1\ge r_1\}\cap A^\d(r)) = \Pr_p((\ho\conn S_{r_1}\textrm{ in }\eta)\circ(\ho\conn S_r\textrm{ in }\eta)) \le
g_p(r_1) \Pr_p(A^\d(r)).
\end{align}
Further, similarly to \eqref{P(|A)},
\begin{equation}
\Pr_p(\rho_1<r_1|A^\d(r))\ge 1 - g_p(r_1),
\end{equation}
which is the lemma's conclusion for $k=1$.
\end{proof}

Now, I want to do some probabilistic reasoning using random variables with the left-continuous distribution function $1-\gt_p$.
The function $1-\gt_p$ is non-decreasing (because for $(h,R)\in\IOd$, $\Pr_p(\ho \conn S_r\textrm{ in }\G\hR\cap L)$ is non-increasing with respect to $r$, so $g_p$ and $\gt_p$ are non-increasing as well), left-continuous, with values in $[0;1]$ and such that~$1-\gt_p(0)=0$, so it is the left-continuous distribution function of a random variable with values in $[0;\infty]$.
\begin{denot}
Let $M_1, M_2,\ldots$ be an infinite sequence of independent random variables all distributed according to $1-\gt_p$ and all independent of the whole percolation process. Because their distribution depends on $p$, I will also denote them by $M_1^{(p)},M_2^{(p)},\ldots$. (Here, an abuse of notation is going to happen, as I am still writing $\Pr_p$ for the whole probability measure used also for defining the variables $M_1, M_2,\ldots$.)
\end{denot}
We can now state the following corollary of Lemma \ref{saus}:
\begin{cor}\label{corsaus}
For any $r>0$, positive integer $k$ and $0<p<1$,
$$\Pr_p(\rho_1+\cdots+\rho_k < r-(k-1)a |A^\d(r)) \ge \Pr_p(M_1+\cdots+M_k < r-(k-1)a ).$$
\end{cor}
\begin{proof}
I compose the proof of the intermediate inequalities:
\begin{align*}
&\Pr_p(\rho_1+\cdots+\rho_k < r-(k-1)a |A^\d(r)) \ge\\
\ge &\Pr_p(\rho_1+\cdots+\rho_{k-1}+M_k < r-(k-1)a |A^\d(r)) \ge\cdots\\
\cdots\ge &\Pr_p(\rho_1+M_2+\cdots+M_k < r-(k-1)a |A^\d(r)) \ge\\
\ge &\Pr_p(M_1+\cdots+M_k < r-(k-1)a |A^\d(r)) = \Pr_p(M_1+\cdots+M_k < r-(k-1)a)
\end{align*}
using the step:
\begin{align}
&\Pr_p(\rho_1+\cdots+\rho_j + M_{j+1}+\cdots+M_k < r-(k-1)a |A^\d(r)) \ge\\
\ge &\Pr_p(\rho_1+\cdots+\rho_{j-1} + M_j+\cdots+M_k < r-(k-1)a |A^\d(r)).
\end{align}
for $j=k,k-1,\ldots,2,1$. Now I prove this step: let $j\in\{1,2,\ldots k\}$.
\begin{df}
Put
$$\RhR = \{d_\hbd(x,y): x,y\in V(\G\hR)\}$$
(I need $\RhR$ as a countable set of all possible values of $\rho_i$ for $i=1,\ldots,N$).
\end{df}
I express the considered probability as an integral, thinking of the whole probability space as Cartesian product of the space on which the percolation processes are defined and the space used for defining $M_1,M_2,\ldots$, and using a version of Fubini theorem for events:
\begin{align*}
&\Pr_p(\rho_1+\cdots+\rho_j + M_{j+1}+\cdots+M_k < r-(k-1)a |A^\d(r)) =\\
= &\int \Pr_p(\rho_1+\cdots+\rho_j + S_M < r-(k-1)a |A^\d(r))\,\ud\L_{j+1}^k(S_M) =\\
\intertext{(here $\L_{j+1}^k$ denotes the distribution of the random variable $M_{j+1}+\cdots+M_k$)}
= &\int \rvoidindop{\sum}{(r_1,\ldots,r_{j-1})}\quad
 \Pr_p\left(\left.\rho_i=r_i\textrm{ for }i<j\land \rho_j < r-(k-1)a - \sum_{i=1}^{j-1} r_i - S_M \right|A^\d(r)\right)\,\ud\L_{j+1}^k(S_M) \ge\\
\intertext{(where the sum is taken over all $(r_1,\ldots,r_{j-1})\in(\RhR)^{j-1}:r_1+\cdots+r_{j-1}<r-(k-1)a-S_M$)}
\ge &\int \rvoidindop{\sum}{(r_1,\ldots,r_{j-1})}\quad
 \left(1-\gt_p\left(r-(k-1)a - \sum_{i=1}^{j-1} r_i - S_M\right)\right) \Pr_p(\rho_i=r_i\textrm{ for }i<j |A^\d(r))\,\ud\L_{j+1}^k(S_M) =\\
\intertext{(from Lemma \ref{saus} and because $g_p\le\gt_p$)}
= &\int \rvoidindop{\sum}{(r_1,\ldots,r_{j-1})}\quad
 \Pr_p \left(\left.M_j < r-(k-1)a - \sum_{i=1}^{j-1} r_i - S_M \land \rho_i=r_i\textrm{ for }i<j \right|A^\d(r)\right)\,\ud\L_{j+1}^k(S_M) =\\
= &\int \Pr_p(\rho_1+\cdots+\rho_{j-1} + M_j+S_M < r-(k-1)a |A^\d(r))\,\ud\L_{j+1}^k(S_M) =\\
= &\Pr_p(\rho_1+\cdots+\rho_{j-1} + M_j+M_{j+1}+\cdots+M_k < r-(k-1)a |A^\d(r)).
\end{align*}
That completes the proof.
\end{proof}

\begin{lem}[cf.~\protect{\cite[Lem.~(5.17)]{Grim}}]\label{lemEN>=frac}
For $0<p<1$, $r>0$,
$$\Est_p(N(A^\d(r))|A^\d(r)) \ge \frac{r}{a+\int_0^r \gt_p(m)\,\ud m}-1.$$
\end{lem}
\begin{proof}
For any $k\in\N_+$, if $\rho_1+\cdots+\rho_k < r-(k-1)a$, then $e_1,\ldots,e_k$ exist and $N(A^\d(r))\ge k$. So, from the corollary above,
\begin{align}\label{P(N>=k|A)>=}
\Pr_p(N(A^\d(r))\ge k|A^\d(r)) \ge \Pr_p(\sum_{i-1}^k \rho_i < r-(k-1)a) \ge \Pr_p (\sum_{i-1}^k M_i < r-(k-1)a).
\end{align}
Now, I use a calculation which relates $a+\int_0^r \gt_p(m)\,\ud m$ to the distribution of $M_1$. Namely, I replace the variables $M_i$ by
$$M_i' = a+\min(M_i,r)$$
for $i=1,2,\ldots$ (a kind of truncated version of $M_i$). In this setting,
\begin{align}
\sum_{i=1}^k M_i < r-(k-1)a \iff
\sum_{i=1}^k \min(M_i,r) < r-(k-1)a \iff
\sum_{i=1}^k M_i' < r+a,
\end{align}
so from \eqref{P(N>=k|A)>=},
\begin{align}
\Est_p(N(A^\d(r))|A^\d(r)) &= \sum_{k=1}^\infty \Pr_p(N(A^\d(r))\ge k |A^\d(r)) \ge\\
&\ge \sum_{k=1}^\infty \Pr_p(\sum_{i=1}^k M_i' < r+a) = \sum_{k=1}^\infty \Pr_p(K\ge k+1) =\\
&= \Est(K)-1,
\end{align}
where
$$K = \min\{k:M_1'+\cdots+M_k'\ge r+a\}.$$
Let for $k\in\N$,
$$S_k=M_1'+\cdots+M_k'.$$
By Wald's equation (see e.g.\ \cite[p.~396]{GrimSti}) for the random variable $S_K$,
$$r+a\le \Est(S_K)=\Est(K)\Est(M_1').$$
In order that Wald's equation were valid for $S_K$, the random variable $K$ has to satisfy $\Est(M_i'|K\ge i) = \Est(M_i')$ for $i\in\N_+$. But we have
$$K\ge i \iff M_1'+\cdots+M_{i-1}'<r+a,$$
so $M_i'$ is independent of the event $\{K\ge i\}$ for $i\in\N_+$, which allows us to use Wald's equation. (In fact, $K$ is a so-called stopping time for the sequence $(M_i')_{i=1}^\infty$.) Hence,
\begin{align}
\Est_p(N(A^\d(r))|A^\d(r)) &\ge \Est(K)-1 \ge \frac{r+a}{\Est(M_1')}-1 =\\
&= \frac{r+a}{a+\int_0^\infty \Pr_p(\min(M_1,r)\ge m)\,\ud m} -1 \ge \frac{r}{a+\int_0^r \gt_p(m)\,\ud m}-1,
\end{align}
which finishes the proof.
\end{proof}

Now, combining that with inequality \eqref{corineqRusso} for $0<\al<\b\le 1$, we have
\begin{align}
f_\al^\d(r) &\le f_\b^\d(r) \exp\(-\int_\al^\b \Est_p(N(A^\d(r))|A^\d(r))\,\ud p\) \le\\
&\le f_\b^\d(r) \exp\(-\int_\al^\b \(\frac{r}{a+\int_0^r \gt_p(m)\,\ud m}-1\)\,\ud p\) \le\\
&\le f_\b^\d(r) \exp\(-(\b-\al) \(\frac{r}{a+\int_0^r \gt_\b(m)\,\ud m}-1\)\)\label{fctineq(f)beflims}
\end{align}
(because $\gt_p\le \gt_\b$ for $p\le\b$), which completes the proof of inequality \ref{fctineq(f)}. (Let us now drop Assumption \ref{assmapath}.)

\par Now, note that for any $r>0$ and $p\in[0;1]$, the event $A^\d(r)$ increases as $\d$ decreases. Thus, taking the limit with $\d\to 0$, we have
$$\lim_{\d\to 0^+} f_p^\d(r) = \Pr_p(\bigcup_{\d> 0} A^\d(r)) = \Pr_p(\ho\conn S_r\textrm{ in }\G\hR\cap L).$$
So for any $r>0$ and $0\le\al<\b\le 1$, using this for inequality \ref{fctineq(f)} gives
\begin{multline}
\Pr_\al(\ho\conn S_r\textrm{ in }\G\hR\cap L) \le\\
\le \Pr_\b(\ho\conn S_r\textrm{ in }\G\hR\cap L) \exp\(-(\b-\al) \(\frac{r}{a+\int_0^r \gt_\b(m)\,\ud m}-1\)\).
\end{multline}
Further, I take the supremum over $(h,R)\in\IOd$, obtaining
$$g_\al(r) \le g_\b(r) \exp\(-(\b-\al) \(\frac{r}{a+\int_0^{r} \gt_\b(m)\,\ud m}-1\)\).$$
At last, taking the limits with $r$ from the left, I get the functional inequality \eqref{fctineq(gt)1} involving only $\gt_\cdot(\cdot)$:
\begin{equation}\label{fctineq(gt)2}
\gt_\al(r) \le \gt_\b(r) \exp\(-(\b-\al) \(\frac{r}{a+\int_0^{r} \gt_\b(m)\,\ud m}-1\)\).
\end{equation}
(Note that the exponent remains unchanged all the time from \eqref{fctineq(f)beflims} till now.)
\par Recall that once we have
$$\int_0^{\infty} \gt_\b(m)\,\ud m = \Est(M_1^{(\b)})<\infty,$$
then we obtain Theorem \ref{lemMen} for $\gt_\al(r)$, for $\al<\b$. This bound is going to be established by showing the rapid decay of $\gt_p$, using repeatedly \eqref{fctineq(gt)2}.  The next lemma is the first step of this procedure.
\begin{lem}[cf.~\protect{\cite[Lem.~(5.24)]{Grim}}]\label{gt<=sqrt}
For any $p<p_0$, there exists $\d(p)$ such that
$$\gt_p(r)\le \d(p)\cdot\frac{1}{\sqrt{r}}\quad\textrm{for }r>0.$$
\end{lem}
I defer proving the above lemma to the end of this section.
\par Obtaining Theorem \ref{lemMen} (being proved) from Lemma \ref{gt<=sqrt} is relatively easy. First, we deduce that for $r>0$ and $p<p_0$,
$$\int_0^{r} \gt_p(m)\,\ud m \le 2\d(p)\sqrt{r},$$
so if $r\ge a^2$, then
$$a+\int_0^{r} \gt_p(m)\,\ud m \le (2\d(p)+1)\sqrt{r}.$$
Then, using \eqref{fctineq(gt)2}, for $0\le\al<\b<p_0$, we have
\begin{align}
\int_{a^2}^{\infty} \gt_\al(r)\,\ud r &\le \int_{a^2}^{\infty} \exp\(-(\b-\al) \(\frac{r}{a+\int_0^{r} \gt_\b(m)\,\ud m}-1\)\) \,\ud r \le\\
&\le e \int_{a^2}^{\infty} \exp\biggl(-\underbrace{\frac{\b-\al}{2\d(\b)+1}}_{=C>0} \sqrt{r}\biggl) \,\ud r =\\
&= e \int_{a}^{\infty} e^{-Cx}\cdot 2x \,\ud x,
\end{align}
so
$$
\Est(M_1^{(\al)}) = \int_0^\infty \gt_\al(r)\,\ud r \le a^2 + e\int_{a}^{\infty} e^{-Cx}\cdot 2x \,\ud x < \infty,
$$
as desired. Finally, we use the finiteness of $\Est(M_1^{(\al)})$ as promised: for $r>0$ and $0\le\al<p_0$, if we take $\al<\b<p_0$, then, using \eqref{fctineq(gt)2} again,
\begin{align}
g_\al(r) \le \gt_\al(r) \le \exp\(-(\b-\al)\(\frac{r}{a+\Est(M_1^{(\b)})} -1\)\) \le e^{-\phi(\al,\b)r+\g(\al,\b)},
\end{align}
for some constants $\phi(\al,\b),\g(\al,\b)>0$.
\uwaga{Chyba za mało mówię, od czego zależą stałe -- skontrolować.}
\par Now I perform a standard estimation, aiming to rule out the additive constant $\g(\al,\b)$. For any $0<\psi_1<\phi(\al,\b)$, there exists $r_0>0$ such that for $r\ge r_0$,
$$-\phi(\al,\b)r+\g(\al,\b) \le -\psi_1 r,$$
so
$$g_\al(r)\le e^{-\psi_1 r}.$$
On the other hand, for any $r>0$, $g_p(r)$ is no greater than the probability of opening at least on edge adjacent to $o$, so $g_\al(r) \le 1-(1-\al)^{\deg(o)}<1$, where $\deg(o)$ is the degree of $o$ in the graph $G$. Hence,
$$g_\al(r)\le e^{-\psi_2(\al)r}$$
for $r\le r_0$, for some sufficiently small $\psi_2(\al)>0$. Taking $\psi = \min(\psi_1,\psi_2(\al))$ gives
$$g_\al(r) \le e^{-\psi r}$$
for any $r>0$, completing the proof of Theorem \ref{lemMen}.
\end{proof}

\par Now I am going to prove Lemma \ref{gt<=sqrt}.
\begin{proof}[Proof of Lemma \ref{gt<=sqrt}]
Assume without loss of generality that $\gt_p(r)>0$ for $r>0$.
I am going to construct sequences $(p_i)_{i=1}^\infty$ and $(r_i)_{i=1}^\infty$ such that
$$p_0>p_1>p_2>\cdots>p,\quad 0<r_1\le r_2\le\cdots$$
and such that the sequence $(\gt_{p_i}(r_i))_{i=1}^\infty$ decays rapidly. The construction is by recursion: for $i\ge 1$, having constructed $p_1,\ldots,p_i$ and $r_1,\ldots,r_i$, we put
\begin{equation}\label{recur}
r_{i+1} = r_i/g_i\quad\textrm{and}\quad p_{i+1} = p_i - 3g_i(1-\ln g_i),
\end{equation}
where $g_i = \gt_{p_i}(r_i)$. (Note that indeed, $r_{i+1}\le r_i$ and $p_{i+1}<p_i$.) The above formula may give an incorrect value of $p_{i+1}$, i.e.~not satisfying $p_{i+1}>p$ (this condition is needed because we want to bound values of $\gt_p$). In order to prevent that, we choose appropriate  values of $p_1,r_1$, using the following fact to bound the difference $p_1 - p_i$ by a small number independent of $i$. 
\begin{prop}
If we define sequence $(x_i)_{i=1}^\infty$ by $x_{i+1}=x_i^2$ for $i\ge 1$ (i.e.~$x_i=x_1^{2^{i-1}}$) with $0<x_1<1$, then
\begin{equation}\label{s(x)}
s(x_1) := \sum_{i=1}^\infty 3x_i(1-\ln x_i)
\end{equation}
is finite and $s(x_1)\tends{x_1\to 0}0$.
\end{prop}
(The idea of the proof of this fact is similar to that of estimating the sum in \eqref{smallinfsum}.)
To make use of it, we are going to bound $g_i$ by $x_i$ for any $i$ and to make $g_1$ small enough. It is done thanks to the two claims below, respectively.
\begin{clm}
If $p_1,\ldots,p_i>p$ and $r_1,\ldots,r_i>0$ are defined by \eqref{recur} with $r_1\ge a$, we have
$$g_{j+1}\le g_j^2$$
for $j=1,\ldots,i-1$.
\end{clm}
\begin{proof}
Let $j\in\{1,\ldots,i-1\}$. From \eqref{fctineq(gt)2},
\begin{align}
g_{j+1} &\le \gt_{p_j}(r_{j+1})\exp\(-(p_j-p_{j+1})\(\frac{r_{j+1}}{a+\int_0^{r_{j+1}} \gt_{p_j}(m)\,\ud m} -1\)\) \le\\
&\le g_j \exp\(1 - (p_j-p_{j+1}) \frac{r_{j+1}}{a+\int_0^{r_{j+1}} \gt_{p_j}(m)\,\ud m} \).
\end{align}
Inverse of the fraction above is estimated as follows
\begin{align}
\frac{1}{r_{j+1}} \(a+\int_0^{r_{j+1}} \gt_{p_j}(m)\,\ud m\) &\le \frac{a}{r_{j+1}} + \frac{r_j}{r_{j+1}} + \frac{1}{r_{j+1}} \int_{r_j}^{r_{j+1}} \gt_{p_j}(m)\,\ud m \le\\
&\le \frac{a}{r_{j+1}} + g_j + \frac{r_{j+1}-r_j}{r_{j+1}} \gt_{p_j}(r_j) \le
\intertext{(using $r_{j+1} = r_j/g_j$ and the monotonicity of $\gt_{p_j}(\cdot)$)}
&\le \frac{a}{r_{j+1}} + 2g_j.
\end{align}
Now, by the assumption, $r_j\ge r_1\ge a$, so $r_{j+1} = r_j/g_j \ge a/g_j$ and
$$\frac{a}{r_{j+1}} + 2g_j \le 3g_j.$$
That gives
$$
g_{j+1} \le g_j \exp\(1 - \frac{p_j-p_{j+1}}{3g_j}\) = g_j^2
$$
by the definition of $p_{j+1}$.
\end{proof}
\begin{denot1}
Put
$$M\hR=r(C\hR)$$
for $(h,R)\in\IOd$.
\end{denot1}
Note that, by Remark \ref{rembdd}, for any $p\in\NP$, $\Pr_p$-a.s.~$ M\hR<\infty$.
\begin{clm}\label{Mtight}
For any $p\in\NP$,
$$\gt_p(r)\tends{r\to\infty}0.$$
\end{clm}

\begin{proof}
First, note that it is sufficient to prove
\begin{equation}\label{convMtight}
\sup_{R\in O(d)} \Pr_p( M^{(1,R)}\ge r)\tends{r\to\infty}0,
\end{equation}
because
\begin{align}
g_p(r) &= \sup_{(h,R)\in\IOd} \Pr_p(\ho\conn S_r\textrm{ in }\G\hR\cap L) \le\\
&\le \sup_{(h,R)\in\IOd} \Pr_p(\ho\conn S_{hr}\textrm{ in }\G\hR) =\quad\textrm{(because $hr\le r$)}\\
&= \sup_{R\in O(d)} \Pr_p(o\conn S_r\textrm{ in }\G^{(1,R)}) \le\quad\textrm{(by scaling the situation)}\\
&\le \sup_{R\in O(d)} \Pr_p(M^{(1,R)}\ge r),
\end{align}
so $g_p(r)\tends{r\to\infty}0$ and, equivalently, $\gt_p(r)\tends{r\to\infty}0$ will be implied.
To prove \eqref{convMtight}, I use upper semi-continuity of the function $O(d)\ni R\mapsto \Pr_p( M^{(1,R)}\ge r)$ for any $p\in\NP$ and $r>0$. Let us fix such $p$ and $r$ and let $(R_n)_n$ be a sequence of elements of $O(d)$ convergent to some $R$. Assume without loss of generality that the cluster $C^{(1,R)}$ is bounded in the Euclidean metric and, throughout this proof, condition on it all the events by default. I am going to show that
\begin{equation}\label{incllimsupM}
\limsup_{n\to\infty}\{M^{(1,R_n)}\ge r\}\subseteq\{ M^{(1,R)}\ge r\}.
\end{equation}

\begin{df}
For any isometry $\Phi$ of $\Hd$, let $\c\Phi$ denote the unique continuous extension of $\Phi$ to $\cHd$ (which is a homeomorphism of $\cHd$---see \cite[Corollary II.8.9]{BH}).
\end{df}
Put $\Phi_n=\Phi^{(1,R)}\circ(\Phi^{(1,R_n)})^{-1}$ and assume that the event $\limsup_{n\to\infty}\{M^{(1,R_n)}\ge r\}$ occurs. Then, for infinitely many values of $n$, all the following occur:
$$M^{(1,R_n)}\ge r \then \c C^{(1,R_n)}\textrm{ intersects }\c S_r \stackrel{\c\Phi_n(\cdot)}{\then} \c C^{(1,R)}=\hc C^{(1,R)}\textrm{ intersects }\c\Phi_n(\c S_r).$$
Let for any such $n$, $x_n$ be chosen from the set $\hc C^{(1,R)}\cap\c\Phi_n(\c S_r)$. Because $\hc C^{(1,R)}$ is compact, the sequence $(x_n)_n$ (indexed by a subset of $\N_+$) has an (infinite) subsequence $(x_{n_k})_{k=1}^\infty$ convergent to some point in $\hc C^{(1,R)}$. On the other hand, note that $\c\Phi_n\tends{n\to\infty}\id_{\cHd}$ uniformly in the Euclidean metric of the disc model (see Definition \ref{dfhHd}). Hence, the distance in that metric between $x_{n_k}\in\c\Phi_{n_k}(\c S_r)$ and $\c S_r$ tends to $0$ with $k\to\infty$, so
$$\lim_{k\to\infty} x_{n_k} \in \c S_r\cap \hc C^{(1,R)} = \hc S_r\cap \hc C^{(1,R)},$$
which shows that $ M^{(1,R)}\ge r$, as desired in \ref{incllimsupM}. Now,
\begin{align}
\limsup_{n\to\infty} \Pr_p(M^{(1,R_n)}\ge r) &\le \Pr_p(\limsup_{n\to\infty}\{M^{(1,R_n)}\ge r\}) \le\quad\textrm{(by an easy exercise)}\\
&\le \Pr_p( M^{(1,R)}\ge r),
\end{align}
which means exactly the upper semi-continuity of $R\mapsto \Pr_p( M^{(1,R)}\ge r)$.
\par Next, note that because for $p\in\NP$ and $R\in O(d)$, a.s.~$M^{(1,R)}<\infty$, we have
$$\Pr_p(M^{(1,R)}\ge r) \tends{r\to\infty} 0\quad\textrm{(decreasingly)}.$$

Hence, if for $r>0$ and $\e>0$ we put
$$U_\e(r)=\{R\in O(d): \Pr_p(M^{(1,R)}\ge r)<\e\},$$
then for any fixed $\e>0$,
\begin{equation}\label{Mtight.cupUe}
\bigcup_{r\nearrow\infty} U_\e(r) = O(d).
\end{equation}
$U_\e(r)$ is always an open subset of $O(d)$ by upper semi-continuity of $R\mapsto \Pr_p( M^{(1,R)}\ge r)$, so by the compactness of $O(d)$, the union \eqref{Mtight.cupUe} is indeed finite. Moreover, because $U_\e(r)$ increases as $r$ increases, it equals $O(d)$ for some $r>0$. It means that $\sup_{R\in O(d)}\Pr_p(M^{(1,R)}\ge r)\le\e$, whence $\sup_{R\in O(d)}\Pr_p(M^{(1,R)}\ge r)\tends{r\to\infty} 0$, as desired.
\end{proof}

Now, taking any $p_1\in(p,p_0)$ and $1>x_1>0$ in \eqref{s(x)} s.t.~$s(x_1) \le p_1-p$ and taking $r_1\ge a$ so large that $\gt_{p_1}(r_1)<x_1$, we obtain for $i\ge 1$, $g_i<x_i$ (by induction). Then, in the setting of \eqref{recur},
\begin{align}
p_{i+1} = p_1 - \sum_{j=1}^i 3g_i(1-\ln g_i) > p_1 - \sum_{j=1}^i 3x_i(1-\ln x_i) \ge
\intertext{(because $x\mapsto 3x(1-\ln x)$ is increasing for $x\in(0;1]$)}
\ge p_1 - s(x_1) \ge p.
\end{align}
Once we know that the recursion \eqref{recur} is well-defined, we use the constructed sequences to prove the lemma. First, note that for $k\ge 1$,
$$r_k = r_1/(g_1g_2\cdots g_{k-1}).$$
Further, the above claim implies
\begin{align}\label{chain(g)}
g_{k-1}^2 \le g_{k-1} g_{k-2}^2 \le \cdots \le g_{k-1}g_{k-2}\cdots g_2 g_1^2 = \frac{r_1}{r_k}g_1 = \frac{\d^2}{r_k},
\end{align}
where $\d=\sqrt{r_1 g_1}$. Now, let $r\ge r_1$. We have $r_k\tends{k\to\infty}\infty$ because $\frac{r_k}{r_{k+1}}=g_k \tends{k\to\infty}0$, so for some $k$, $r_{k-1}\le r<r_k$. Then,
\begin{align}
\gt_p(r) \le \gt_{p_{k-1}}(r) \le \gt_{p_{k-1}}(r_{k-1}) = g_{k-1} \le \frac{\d}{\sqrt{r_k}} < \frac{\d}{\sqrt{r}}
\end{align}
(from \eqref{chain(g)} and the monotonicity of $\gt_p(r)$ with regard to each of $p$ and $r$), which finishes the proof.
\end{proof}

\begin{rem}\label{diffprfs}
As declared in Section \ref{prflemMen}, in this remark I summarise the differences between the proof of Theorem \ref{lemMen} and the proof of Theorem (5.4) in \cite{Grim}:
\begin{enumerate}
\item First, I recall that the skeleton structure and most of the notation of the proof here  is borrowed from \cite{Grim}. The major notation that is different here, is ``$A^\d(r)$'' and ``$S_r$'' (respectively $A_n$ and $\bd S(s)$ in \cite{Grim}).
\item To be strict, the proper line of the proof borrowed from \cite{Grim} starts by considering the functions $f_p$ instead of $g_p$ or $\gt_p$, although the functional inequality \eqref{fctineq(f)} involves both functions $f_\cdot(\cdot)$ and $\gt_\cdot(\cdot)$. In fact, each of the functions $f_\cdot(\cdot)$, $\gt_\cdot(\cdot)$ and $g_\cdot(\cdot)$ is a counterpart of the function $g_\cdot(\cdot)$ from \cite{Grim} at some stage of the proof. After proving inequality \eqref{fctineq(f)}, I pass to a couple of limits with it in order to obtain inequality \eqref{fctineq(gt)1} involving only $\gt_\cdot(\cdot)$ (the step not present in \cite{Grim}). This form is  needed to perform the repeated use of inequality \eqref{fctineq(gt)1} at the end of the proof of Theorem \ref{lemMen}.
\item Obviously, the geometry used here is much different from that in \cite{Grim}. In fact, I analyse the percolation cluster in $\G\hR\cap L_\d$ using the pseudometric $d_\hbd$ (in place of the graph metric $\d$ in \cite{Grim}). Consequently, the set $\RhR$ of possible values of the random variables $\rho$ in Lemma \ref{saus} is much richer than $\N$, the respective set for the graph $\Z^d$. Moreover, the functions $\gt_p$ arise from the percolation process on the whole $\G\hR\cap L$, so the distribution of the random variables $M_i$ is not necessarily discrete. That cause the need for using integrals instead of sums, when concerned with those random variables, especially in the proof of Corollary \ref{corsaus}. All that leads also to a few other minor technical differences between the proof here and the proof in \cite{Grim}.
\item I tried to clarify the use of the assumption on $\sum_{i=1}^k r_i$ in Lemma \ref{saus} and why Wald's equation can be used in the proof of Lemma \ref{lemEN>=frac}, which I found quite hidden in \cite{Grim}. I also replaced the random variable $G$ in the proof of \cite[Lem.\ (5.12)]{Grim} with the event ``$\Gamma$ is a witness for $B$'' (in the proof of Lemma \ref{saus}), as I find the latter more precise way to explain the calculations.
\item The proof of Lemma \ref{gt<=sqrt} itself has a little changed structure (compared to the proof of Lemma (5.24) in \cite{Grim}) and contains a proof of the convergence $\gt_p(r)\tends{r\to\infty}0$ (Claim \ref{Mtight}).
\end{enumerate}
\end{rem}

\begin{rem}\label{whyLd}
When I was working on the proof of Theorem \ref{lemMen}, I tried to consider the percolation processes on the whole graph $\G\hR\cap L$ (without restricting it to $L_\d$) in order to obtain functional inequality similar to \eqref{fctineq(gt)1}, involving only one function. That approach caused many difficulties some of which I have not overcome. Restricting the situation ti $L_\d$ makes the event $A^\d$ depend on the states of only finitely many edges. That allows e.g.\ to condition the event $A^\d(r)\cap B$ on the family of events $\{\Gamma\textrm{ a witness for }B\}$, where $\Gamma$ runs over a countable set (in the proof of Lemma \ref{saus}) or to use BK inequality and Russo's formula.
\end{rem}

\backmatter

\end{document}